\documentclass[10pt, reqno]{amsart}
\usepackage{amsmath,amssymb, amsfonts, yfonts, mathrsfs, amscd}
\usepackage{amsthm}
\usepackage{mathrsfs}
\usepackage{amsfonts, dsfont}
\usepackage{kotex}
\usepackage{xcolor} 
\usepackage{subcaption}
\usepackage{graphicx}
\usepackage{csquotes}
\usepackage{epsfig}

\usepackage[normalem]{ulem}
\usepackage{cancel}

\usepackage{soul}
\usepackage[colorlinks=true]{hyperref}
\hypersetup{urlcolor=blue, citecolor=red}
\usepackage{enumerate}
\usepackage{enumitem}

\usepackage[margin=2cm,top=2cm,bottom=2cm]{geometry}

\allowdisplaybreaks

\makeatletter
\newtheorem*{rep@theorem}{\rep@title}
\newcommand{\newreptheorem}[2]{%
\newenvironment{rep#1}[1]{%
 \def\rep@title{#2 \ref{##1}}%
 \begin{rep@theorem}}%
 {\end{rep@theorem}}}
 
\makeatother

\numberwithin{equation}{section}

\newtheorem{theorem}{Theorem}[section]
\newtheorem{lemma}[theorem]{Lemma}
\newtheorem{corollary}[theorem]{Corollary}
\newtheorem{proposition}[theorem]{Proposition}
\newtheorem{remark}{Remark}[section]

\newtheorem{assumption}{Assumption}[section]

 \newcommand{\bq}{\begin{equation}}
  \newcommand{\eq}{\end{equation}}
    \newcommand{\intrr}{\iint_{\R^d \times \R^d}}

\def\XXint#1#2#3{{\setbox0=\hbox{$#1{#2#3}{\int}$ }
\vcenter{\hbox{$#2#3$ }}\kern-.6\wd0}}

\newcommand{\bfF}{{\bf F}}

\newcommand{\bfW}{{\bf W}}
\newcommand{\bfX}{{\bf X}}

\newcommand{\bbE}{\mathbb E}

\newcommand{\bbI}{\mathbb I}

\newcommand{\bbP}{\mathbb P}

\newcommand{\bbR}{\mathbb R}

\newcommand{\calC}{\mathcal C}
\newcommand{\calD}{\mathcal D}
\newcommand{\calE}{\mathcal E}
\newcommand{\calF}{\mathcal F}
\newcommand{\calG}{\mathcal G}

\newcommand{\calM}{\mathcal M}

\newcommand{\calO}{\mathcal O}
\newcommand{\calP}{\mathcal P}

\newcommand{\calT}{\mathcal T}

\newcommand{\calV}{\mathcal V}
\newcommand{\calW}{\mathcal W}

\newcommand{\swabM}{\textswab M}

\newcommand{\bp}{\begin{pmatrix}}
\newcommand{\ep}{\end{pmatrix}}

\newcommand{\N}{\mathbb{N}}
\newcommand{\R}{\mathbb{R}}

\newcommand{\dx}{\textnormal{d}x}

\newcommand{\lt}{\left}
\newcommand{\rt}{\right}
\newcommand{\pa}{\partial}

\newcommand{\intr}{\int_{\R^d}}

\newcommand{\bflambda}{{\boldsymbol{\lambda}}}
\newcommand{\Lambdaa}{\Lambda_\alpha}

\newcommand{\dcal}{\Delta_t^\calM}
\newcommand{\dm}{\Delta_t^m}

\newcommand{\mrho}{m_f^h[\bar{\rho}_t]}
\newcommand{\xm}{\bar{X}_t}
\newcommand{\bbflambda}{\bar{\bflambda}}
\newcommand{\mmu}{\bar{\mu}_t^N}

\begin{document}

\title[A modified Consensus-Based Optimization model]{A modified Consensus-Based Optimization model: consensus formation and uniform-in-time propagation of chaos}

\author[Choi]{Young-Pil Choi}
\address[Young-Pil Choi]{\newline \text{ }\quad Department of Mathematics, Yonsei University, Seoul 03722, Republic of Korea}
\email{ypchoi@yonsei.ac.kr}

\author[Lee]{Seungchan Lee}
\address[Seungchan Lee]{\newline \text{ }\quad Department of Mathematics, Yonsei University, Seoul 03722, Republic of Korea}
\email{seungchan2718@yonsei.ac.kr}

\author[Song]{Sihyun Song}
\address[Sihyun Song]{\newline \text{ }\quad Department of Mathematics, Yonsei University, Seoul 03722, Republic of Korea}
\email{ssong@yonsei.ac.kr}

\date\today

\begin{abstract}

We introduce a modified Consensus-Based Optimization model based on a regularized Gibbs weight. For each fixed inverse temperature and regularization parameter, the added positive weight gives a configuration-independent lower bound on the denominator of the consensus map and permits a quantitative analysis without an upper bound on the objective function or a cutoff in the noise. We first prove exact large-time consensus for the finite-particle system: above an explicit stability threshold, every particle converges exponentially in $L^p$, and almost surely for every $p\ge2$, to a common random limit. This consensus mechanism does not require the $C^2$ Hessian condition used in the direct weighted-energy argument. We next prove uniform-in-time propagation of chaos. The contraction and synchronous-coupling estimates have constants with no explicit ambient-dimension factor, while the empirical Wasserstein sampling term retains the standard dimension dependence. Accordingly, the full empirical-measure convergence rate displays the usual sampling dependence on $d$. As a consequence, we derive quantitative bounds for the finite-particle consensus point from finite-time concentration estimates for the mean-field law, with explicit approximation and large-time tail errors. The synchronous-coupling estimate also gives a quantitative comparison between the random particle consensus point and the deterministic mean-field consensus point. We also quantify the regularization error in the quantitative Laplace principle, making the effect of the additive regularization explicit. A direct weighted-energy estimate is obtained under stronger $C^2$ assumptions. We also give a stopped quadratic stability estimate with the actual Gibbs denominator.  Under the weak-landscape condition, full support of the initial law, and local nondegeneracy of the minimizer, we prove a high-probability finite-particle global-optimization theorem for any fixed $\lambda>\sigma^2$. The proof combines an affine-flow local-mass estimate, finite-time mean-field entrance, ordinary finite-horizon propagation of chaos, and an $\alpha$-independent convex trapping energy near the minimizer.
\end{abstract}

\maketitle
\setcounter{tocdepth}{1}
\tableofcontents

%
%
%
%
%
%
%
%
%
%

\section{Introduction}
\subsection{Historical survey}
In the broader landscape of global optimization, metaheuristic algorithms such as random search \cite{Ras63}, simulated annealing \cite{KGV83, LA87}, genetic algorithms \cite{Hol92}, and particle swarm optimization \cite{Ken10, KE95} have been widely used to tackle high-dimensional and nonconvex problems where gradient information is unavailable or unreliable. These approaches combine stochastic exploration and local exploitation mechanisms, often inspired by collective behavior or evolutionary processes, to overcome the limitations of deterministic optimization methods. However, while they have proven empirically successful, most of them lack rigorous mathematical guarantees of global convergence, primarily due to their algorithmic complexity and intrinsic stochasticity \cite{BNO03, BDGG09, BDT99}.

A major step toward a mathematically tractable metaheuristic was taken with the introduction of the \emph{Consensus-Based Optimization (CBO)} framework in \cite{PTTM17}. This model describes a system of interacting particles that collectively search for a global minimizer through attraction toward a consensus point, which is defined as a Gibbs-weighted average of their positions. The original formulation in \cite{PTTM17} already incorporated stochastic effects, but the rigorous analysis was first carried out for a simplified deterministic setting, where the discontinuous drift switch was removed and the stochastic diffusion term was neglected for analytical tractability. The corresponding stochastic model was later analyzed at the mean-field level in \cite{CCTT18}, which established convergence to global minimizers under suitable assumptions and initiated the PDE-based theory of CBO dynamics. These foundational works laid the analytical basis for interpreting CBO as a stochastic interacting-particle algorithm performing an effective ``convexification'' of nonconvex landscapes in the mean-field limit. For an overview of recent advances in CBO, we refer the reader to \cite{BHKR2025, GKHV25, Tot22} and the references therein.

Subsequent work has substantially broadened the analytical scope of CBO. The high-dimensional formulation in \cite{CJLZ21} established convergence results for particles driven by independent Brownian perturbations. Componentwise multiplicative noise based on the Hadamard product has been studied, for example, in \cite{HJK20,HJK21,WB25}. Further variants include constrained CBO models \cite{BHKLMY22,BHP23,CTV23,BCFHW26}, polarization-based dynamics \cite{BRW25}, mirror-descent-type models \cite{BHKR2025}, CBO on manifolds \cite{FHPS20,FHPS21,FHPS22,HKKKY22}, consensus-based sampling \cite{CHSV22}, rescaled or truncated-noise systems \cite{HK25a,HK25b,HKS25,FRRS25}, and models with L\'evy noise or memory effects \cite{KST23,TW20}.

The mean-field limit toward the associated McKean--Vlasov process has also been studied extensively. Quantitative finite-time estimates were obtained in \cite{GHV25,FHPS20,FHPS21,FHPS22,FKR24,KST23}, while qualitative convergence over finite horizons was investigated in \cite{HQ22}. Uniform-in-time estimates are available for the rescaled formulation \cite{HK25b}, the classical model under bounded objective functions \cite{GKHV25}, and a weak-propagation framework on bounded domains \cite{BEZ26}. A recent alternative well-posedness approach for the CBO McKean--Vlasov equation is given in \cite{Baldi26}.

These developments have considerably deepened the theory of CBO. Important questions remain when one seeks simultaneously to avoid upper bounds or global Lipschitz assumptions on the objective, dispense with noise cutoffs, and retain quantitative control over both finite-particle and mean-field dynamics.

For the present work it is useful to distinguish mean-field global optimization from direct finite-particle optimizer estimates. PDE-based results such as \cite{FKR24} allow nonsmooth and nonconvex objectives under quantitative landscape conditions, whereas direct finite-particle analyses have typically used different regularity or initialization hypotheses; see, for example, \cite{HJK20,HJK21,WB25,HHK26}. Thus, we treat exact particle consensus, uniform-in-time mean-field approximation, and optimizer identification as separate analytical mechanisms with their own hypotheses.

%
%
%
%
%
%
%
%
%
%

\subsection{Introduction of a more tractable model}

In the present work, we study a modified CBO system whose principal purpose is analytical. The additive regularization gives a deterministic lower bound on the denominator of the weighted mean for every particle configuration, which is particularly useful for global well-posedness, contraction estimates, and uniform-in-time stability without an upper bound on $f$ or a cutoff in the noise. The stochastic term is scalar multiplicative noise: for each particle, one scalar Brownian increment multiplies the vector $X_t^i-m_t^h$. The particle-contraction and synchronous-coupling constants contain no explicit ambient-dimension factor; the dimension dependence in the full propagation-of-chaos estimate enters through the empirical Wasserstein sampling term.

Mathematically, we work on a complete filtered probability space $(\Omega, \calF, \{\calF_t\}_{t\ge 0}, \bbP)$ and consider the interacting particle system
\begin{equation}\label{I: eq: main}
\begin{cases}
    dX_t^i = -\lambda (X_t^i - m_t^h) \, dt  + \sigma (X_t^i - m_t^h ) \, dW_t^i, \quad i=1,2,\ldots, N,\\
    m_t^h := \displaystyle \frac{\sum_{i=1}^N X_t^i \psi_h(X_t^i) }{\sum_{i=1}^N \psi_h(X_t^i) },  
    \end{cases}
\end{equation}
where $\lambda,\sigma>0$ are positive parameters representing, respectively, the strength of drift toward consensus and the intensity of stochastic diffusion, and $\{W_t^i\}_{i=1}^N$ are independent one-dimensional Wiener processes.  The term $m_t^h$ represents the consensus point toward which all particles are attracted.

In \eqref{I: eq: main}, $W_t^i$ is scalar, so the same scalar increment multiplies all coordinates of $X_t^i-m_t^h$, while the noises remain independent across particle labels $i$. Hence the model uses neither a common noise shared by all particles nor the componentwise $d$-dimensional noise often considered in CBO. Variants with other noise structures and their corresponding consensus thresholds are discussed in Remark \ref{rem: noise structures}.

The consensus point is the weighted average of the particle positions with respect to the regularized Gibbs weight
\[
\psi_h(x) :=\omega_f^\alpha(x)+h(\alpha),  \quad \omega_f^\alpha(x)=e^{-\alpha f(x)},  \quad \alpha > 0,
\]
where $h$ is a strictly positive function of $\alpha$. This additional term $h(\alpha)$ guarantees that the denominator in the definition of $m_t^h$ remains uniformly bounded away from zero, thereby removing possible singularities and stabilizing the dynamics. The parameter $\lambda$ will later be required to satisfy a quantitative stability threshold ensuring exponential contraction of trajectories. By construction, the dynamics are symmetric with respect to particle indices. In particular, when the initial data $\{X_0^i\}_{i=1}^N$ are i.i.d.\ and the driving Wiener processes are independent, the unique strong solution $\{X_t^i\}_{i=1}^N$ remains exchangeable for all $t\ge0$, since both the drift and diffusion coefficients depend only on the empirical measure. A short proof of this property is provided in Appendix \ref{app_exch}.

 The classical CBO consensus point is recovered by setting $h=0$:
\[
m^0_t := \frac{\sum_{i=1}^N X_t^i \omega_f^\alpha(X_t^i)}{\sum_{i=1}^N \omega_f^\alpha(X_t^i)}.
\] 
The additive term $h(\alpha)$ is introduced to obtain the configuration-independent deterministic bound
\[
\sum_{i=1}^N\psi_h(X_t^i)\ge Nh(\alpha).
\]
This lower bound prevents degeneracy of the denominator in the estimates below and is the basic regularizing mechanism used in the well-posedness and stability arguments.

To make the interpolation explicit, set
\[
D_t:=\sum_{i=1}^N\omega_f^\alpha(X_t^i),\quad \mathcal M_t^N:=\frac1N\sum_{i=1}^N X_t^i.
\]
Since $\sum_i\psi_h(X_t^i)=D_t+Nh(\alpha)$ and $\sum_iX_t^i\psi_h(X_t^i)=D_t m_t^0+Nh(\alpha)\mathcal M_t^N$, we obtain
\[
m_t^h =\frac{D_t m_t^0+Nh(\alpha)\mathcal M_t^N}{D_t+Nh(\alpha)} =\beta_t m_t^0+(1-\beta_t)\mathcal M_t^N, \quad \beta_t:=\frac{D_t}{D_t+Nh(\alpha)} =\frac{\sum_i\omega_f^\alpha(X_t^i)}{\sum_i\psi_h(X_t^i)}.
\]
Since $0<\beta_t<1$, $m_t^h$ is a convex combination of the classical Gibbs-weighted mean and the empirical mean. This representation yields coefficients suited to global well-posedness, large-time analysis, mean-field convergence, and uniform-in-time propagation of chaos, while allowing objective functions that may be unbounded and need not be globally Lipschitz.

For each \emph{fixed} $\alpha>0$, the consensus and propagation-of-chaos estimates below only require $h(\alpha)>0$. In this regime, the weight ratio entering the consensus threshold in Theorem \ref{thm_main1} is finite:
\[
\Lambda_\alpha=1+\frac{e^{-\alpha\underline f}}{h(\alpha)}<\infty.
\]
Indeed, if
\[
F:=f-\underline f,\quad \widehat h_\alpha:=e^{\alpha\underline f}h(\alpha),\quad Z_\alpha[\rho]:=\int_{\R^d}e^{-\alpha F(x)}\,\rho(dx),
\]
and
\[
v_\alpha[\rho]:=\frac{\int_{\R^d}x e^{-\alpha F(x)}\rho(dx)}{Z_\alpha[\rho]},
\]
then the regularized mean admits the exact decomposition
\begin{equation}\label{eq:reg-conv-deco}
m_f^h[\rho] =\frac{Z_\alpha[\rho]}{Z_\alpha[\rho]+\widehat h_\alpha}v_\alpha[\rho] +\frac{\widehat h_\alpha}{Z_\alpha[\rho]+\widehat h_\alpha}\mathcal M(\rho), \quad \mathcal M(\rho) = \int_{\R^d} x \,\rho(dx).
\end{equation}
Equation \eqref{eq:reg-conv-deco} makes the dependence on the regularization scale explicit. In the large-$\alpha$ optimization arguments below, we track the normalized quantity $\widehat h_\alpha=e^{\alpha\underline f}h(\alpha)$ relative to the Gibbs mass $Z_\alpha[\rho]$. Unlike the classical mean $m_f^0$, the regularized mean with fixed $h$ is not invariant under $f\mapsto f+C$ unless $h$ is rescaled by $e^{-\alpha C}$; hence all $\alpha$-dependent statements involving $h$ are formulated with this normalization.

 Throughout the paper, the objective function $f$ is assumed to be locally Lipschitz. This regularity is used for strong well-posedness of \eqref{I: eq: main} (see Appendix \ref{sec: exist}) and in the pathwise convergence argument of Theorem \ref{thm:conv-trap}. The fixed-parameter contraction estimates themselves do not require derivatives of $f$. We also assume that
\[
\exists! \, x^*\in \R^d \text{ for which } f(x^*) = \underline{f} := \inf_{x\in \R^d} f(x)>0.
\]
Since the regularized dynamics are not invariant under $f\mapsto f+C$ when $h$ is held fixed, the normalization $\underline f>0$ is part of the present parameterization.

%
%
%
%
%
%
%
%
%
%

\subsection{Main results}

The analysis developed in this work consists of three complementary components:
\begin{itemize}
\item the large-time behavior and optimization properties of the finite particle system,
\item the uniform-in-time propagation of chaos linking the particle and mean-field descriptions, and
\item the large-time behavior and optimization properties of the associated McKean--Vlasov equation.
\end{itemize}

The three components are governed by distinct mechanisms and use distinct hypotheses. Exact finite-$N$ consensus follows from dynamical contraction and does not use the Hessian condition on $f$. Uniform-in-time propagation of chaos uses stability of the regularized weighted mean, for which fixed-$\alpha$ Lipschitz continuity of $e^{-\alpha f}$ is sufficient. Optimizer information is obtained through two complementary routes. First, a quantitative concentration estimate for the mean-field law near a minimizer yields a corresponding estimate for the random finite-particle consensus point, up to an explicit finite-$N$ error and an exponentially small particle tail. Second, under explicit additional assumptions, we obtain a high-probability finite-particle global-optimization result through finite-time mean-field entrance, finite-horizon particle approximation, and convex trapping. We also record a direct It\^o weighted-energy estimate under a Hessian condition for comparison with particle-level arguments.

\begin{assumption}\label{assum: omega}
    For each fixed $\alpha>0$, the map $x\mapsto \omega_f^\alpha(x):=e^{-\alpha f(x)}$ is globally Lipschitz continuous. We denote its optimal Lipschitz constant by $L_{\omega_f^\alpha}$.
\end{assumption}

\begin{assumption}\label{assump: f hessian}
    We assume $f\in \calC^2(\bbR^d)$, and there exist constants $c_0,c_1\ge 0$ such that
\[
            \nabla^2 f(x) \preceq c_0 \bbI_d + c_1 \nabla f(x) [\nabla f(x)]^{\top},
\]
in other words
\[
            c_0 \bbI_d + c_1 \nabla f(x) [\nabla f(x)]^{\top}-\nabla ^2f(x)\text{ is positive semi-definite} \quad \forall x\in \R^d.
\]
\end{assumption}

Assumption \ref{assump: f hessian} is used only for the direct It\^o arguments below; it is not used for exact particle consensus or for the uniform-in-time propagation of chaos.

\begin{assumption}\label{assum: weak landscape}
Let $x^*$ be the unique global minimizer and set $F=f-\underline f$. There exist $R_0,\eta,f_\infty>0$ and $\nu>0$ such that
\[
|x-x^*|\le \frac{F(x)^\nu}{\eta}\quad\text{whenever }|x-x^*|\le R_0,
\]
and
\[
F(x)>f_\infty\quad\text{whenever }|x-x^*|>R_0.
\]
This type of inverse-continuity/far-field condition is standard in weak-regularity mean-field CBO theory; see, for instance, \cite{FKR24}. It will be used in the regularized quantitative Laplace estimate and, together with the explicit additional hypotheses below, in the global-optimization route of Section \ref{subsec: explicit global route}.
\end{assumption}

\begin{assumption}\label{assum: local nondegenerate}
There exists a neighborhood $U$ of $x^*$ on which $f\in\calC^2(U)$, and $\nabla^2f(x^*)$ is positive definite.
\end{assumption}

Assumption \ref{assum: local nondegenerate} is strictly local. In particular, it does not impose global convexity or the global Hessian bound in Assumption \ref{assump: f hessian}; outside a neighborhood of $x^*$ the objective may remain nonconvex, multimodal, or only locally Lipschitz.

We record two examples illustrating the scope of the assumptions above.
\begin{remark}
    Assumption \ref{assum: omega} is, for instance, readily satisfied in the following cases:
    \begin{itemize}
    \item $f$ is a polynomial with positive minimum and a positive-definite leading homogeneous part. Writing $f=p+\underline f$, polynomial growth of $\nabla p$ and coercivity of $p$ imply
\[
K_p:=\sup_x|\nabla p(x)|e^{-p(x)}<\infty,\quad L_{\omega_f^\alpha}\le K_p\alpha e^{-\alpha\underline f} \quad(\alpha\ge1).
\]
For each fixed $\alpha>0$, the same exponential domination gives a finite Lipschitz constant. Even total degree alone is not sufficient. For example, $P(x,y)=x^2(1+y^2)+y^2(xy-1)^2\ge0$ has its unique zero at $(0,0)$, but, at $y=\frac1x $,
\[
P(x,\frac1x )=1+x^2,\quad \pa_xP(x,\frac1x )=2x+\frac2x.
\]
Thus $f=1+P$ has a positive minimum and even degree, while $e^{-\alpha f}$ is not globally Lipschitz for any $\alpha>0$.

    \item $f$ is Lipschitz continuous. Indeed, the same expression above shows that
    \[L_{\omega_f^\alpha}\leq \|\nabla f\|_{L^\infty} \alpha e^{-\alpha\underline{f}} \xrightarrow[\alpha\to \infty]{} 0.\]
    \end{itemize}
\end{remark}

 \begin{remark}%
Assumption \ref{assump: f hessian} is automatically satisfied if, for instance, the objective function $f$ has globally bounded second derivatives.   Moreover, both Assumptions \ref{assum: omega} and \ref{assump: f hessian} are fulfilled by several standard benchmark functions used in global optimization.   A representative example is the \emph{Rastrigin function} (see, e.g., \cite{PTTM17}), defined by
\[
        f_R(x)=\frac{1}{d}\sum_{i=1}^d\left[(x_i-B)^2-10\cos(2\pi(x_i-B))+10\right]+C,
\]
where $B\in\R$ and $C>0$ are constant shifts. In this case, the function's linear gradient growth (resulting from its quadratic structure) and bounded Hessian ensure that Assumptions \ref{assum: omega} and \ref{assump: f hessian} are fulfilled.
\end{remark}

\begin{remark} \label{rem: bounded surrogate}
For fixed $\alpha>0$ and $\widehat h_\alpha=e^{\alpha\underline f}h(\alpha)>0$, set
\[
G_{\alpha,\widehat h_\alpha}(x) :=-\frac1\alpha\log\left( \frac{e^{-\alpha(f(x)-\underline f)}+\widehat h_\alpha}{1+\widehat h_\alpha} \right).
\]
Multiplying every weight by the same positive constant leaves the weighted mean unchanged. Hence $m_f^h[\rho]$ is exactly the usual Gibbs-weighted mean associated with $G_{\alpha,\widehat h_\alpha}$, and this surrogate has the same minimizer as $f$. Moreover,
\[
0\le G_{\alpha,\widehat h_\alpha}(x) \le \frac{\log\Lambda_\alpha}{\alpha}, \quad \operatorname{Lip}(G_{\alpha,\widehat h_\alpha}) \le \frac{e^{\alpha\underline f}L_{\omega_f^\alpha}} {\alpha\widehat h_\alpha}.
\]
Under Assumption \ref{assum: omega}, this is a bounded, globally Lipschitz objective. This observation relates the fixed-parameter regularization to the bounded-objective framework of \cite{GKHV25}, whose noise structures differ from the scalar multiplicative noise used here. The surrogate and its bounds depend on $\alpha$ and $h(\alpha)$; the reformulation alone does not establish an optimizer-compatible parameter regime. Our optimization argument instead tracks the regularization error explicitly and uses the scalar-noise affine flow and a convex trapping energy independent of $\alpha$.
\end{remark}

The first three theorems give the fixed-parameter consensus and approximation results. Their quantitative consequence is stated in Corollary \ref{cor:cons-est}. The final theorem establishes high-probability optimization through finite-time entrance and convex trapping.

 \begin{theorem}[Large-time particle consensus and a direct weighted-energy estimate]
\label{thm_main1}
Let $\{X_t^i\}_{i=1}^N$ be a global strong solution to the system \eqref{I: eq: main} with i.i.d. initial data $\{X_0^i\}_{i=1}^N$, distributed as $X_{\rm in}\in L^p(\Omega)$ for some $p\ge2$. Define the consensus threshold
\[
\bflambda_{p,\alpha,\sigma}:=(p-1)\Lambda_\alpha^{\frac2p}\sigma^2 \quad \text{with } \Lambda_\alpha := \frac{e^{-\alpha\underline{f}} + h(\alpha)}{h(\alpha)}.
\]
If $\lambda>\bflambda_{p,\alpha,\sigma}$, then the following hold:
\begin{enumerate}[label=(\roman*)]
    \item (\emph{Large-time convergence})  
    There exists a random vector $X_\infty = X_\infty^{N,\alpha} \in L^p(\Omega)$ such that
    \[
    X_t^i \xrightarrow[t\to \infty]{} X_\infty \quad \text{in } L^p(\Omega) \text{ for all } i=1,\dots,N.
    \]
    
    In fact, the convergence holds almost surely for every $p\ge2$. Moreover, with
    \[
    \gamma_p:=\lambda-\bflambda_{p,\alpha,\sigma}>0,\quad
    A_{N,p}:=\left(\Lambda_\alpha\max_{1\le i,j\le N}\bbE|X_0^i-X_0^j|^p\right)^{1/p},
    \]
    one has the quantitative tail estimate
    \begin{equation}\label{eq: particle tail intro}
    \|X_t^i-X_\infty\|_{L^p(\Omega)}
    \le C_{\rm tail}e^{-\gamma_p t},\quad
    C_{\rm tail}:=A_{N,p}\left(\frac{\lambda}{\gamma_p}+\frac{\sigma C_{\rm BDG,p}^{1/p}}{\sqrt{2\gamma_p}}\right).
    \end{equation}

    \item (\emph{Direct weighted-energy estimate under stronger smoothness})  
    Suppose furthermore that Assumptions \ref{assum: omega} and \ref{assump: f hessian} hold, that $\alpha\ge c_1$, and that 
    $\lambda > \bflambda_{p,\alpha,\sigma} \vee \bflambda_{2,\alpha,\sigma}$.
    If there exists $0<\epsilon\leq 1$ such that
\[
        (1-\epsilon)\mathbb{E}\omega_f^\alpha(X_\textnormal{in})\geq  \frac{\lambda L_{\omega_f^\alpha}}{\lambda-\boldsymbol{\lambda}_{2, \alpha, \sigma}}\left(2\Lambda_\alpha \textnormal{Var}(X_{\rm in})\right)^\frac{1}{2}+\frac{\alpha \sigma^2c_0\Lambdaa e^{-\alpha\underline{f}}}{2(\lambda-\boldsymbol{\lambda}_{2, \alpha, \sigma})} \textnormal{Var}(X_{\rm in})
\]
    
then the limit variable $X_\infty$ satisfies the fixed-parameter estimate
\[
\operatorname*{ess\,inf}_{\omega\in\Omega} f(X_\infty(\omega)) \leq -\frac{1}{\alpha} \log \mathbb{E}\!\left[e^{-\alpha f(X_{\rm in})}\right] -\frac{1}{\alpha}\log \varepsilon .
\]
\end{enumerate}
\end{theorem}

\begin{remark}
The estimate in Theorem \ref{thm_main1}(ii), as well as its mean-field analogue in Theorem \ref{thm:mf_main}(ii) below, is a fixed-parameter statement. In particular, we do not use these direct weighted-energy estimates to pass to the limit $\alpha\to\infty$. The initialization conditions appearing in these results are imposed only for the fixed parameter values under consideration and are not assumed to persist uniformly as $\alpha$ increases. The global-optimization result for large inverse temperature is instead obtained through the separate finite-horizon entrance and convex trapping argument developed in the proof of Theorem \ref{thm: finite global optimization}.
\end{remark}

\begin{remark}
The additional initialization condition in Theorem \ref{thm_main1}(ii) is used only in the direct weighted-energy estimate. Corollary \ref{cor:cons-est} gives a general comparison for the particle consensus point under the fixed-parameter consensus assumptions. The high-probability optimization result in Theorem \ref{thm: finite global optimization} is obtained by a separate argument combining finite-time mean-field entrance, finite-horizon particle approximation, and convex trapping.
\end{remark}

We next turn to uniform-in-time propagation of chaos for \eqref{I: eq: main}. The synchronous-coupling stability constants contain no explicit ambient-dimension factor, while the empirical-measure estimate also contains the standard i.i.d. Wasserstein sampling error and hence retains its usual dimension dependence. The theorem below quantifies the approximation of the particle system by its McKean--Vlasov counterpart uniformly on $[0,\infty)$.

Our approach is inspired by the framework of \cite{GKHV25}, which established uniform-in-time propagation of chaos for the original CBO equations through a novel interpretation of the weighted mean (see Section \ref{subsec: apriori obs}).

We adapt the uniform-in-time stability mechanism of \cite{GKHV25} to the regularized consensus map and scalar multiplicative noise. The original objective $f$ may be unbounded above; the relevant fixed-parameter bounds can also be understood through the surrogate in Remark \ref{rem: bounded surrogate}. Global Lipschitz continuity of $e^{-\alpha f}$ controls the weighted mean, whereas the landscape condition used in mean-field global-optimization results such as \cite{FKR24} controls its location near a minimizer. The natural quantitative statement is at the level of $\mathbb E\mathcal W_p^p$; after taking $p$-th roots, the empirical-sampling exponents are divided by $p$.

We denote by $\bar X_t$ the solution to the following McKean--Vlasov equation, driven by the one-dimensional Wiener process $\{W_t\}_{t\ge 0}$:
\begin{equation}\label{eq: coupling}
    \begin{cases}
    \displaystyle \bar{X}_t=\bar{X}_0-\lambda\int_0^t(\bar{X}_s-m_f^h[\bar{\rho}_s])ds+\sigma\int_0^t(\bar{X}_s-m_f^h[\bar{\rho}_s])dW_s, \\[3mm]
  \displaystyle m_f^h[\rho]=\frac{\intr x\psi_h(x)\rho(dx)}{\intr \psi_h(x)\rho(dx)}, \quad  \bar{\rho}_t := \textnormal{Law}(\bar X_t).
    \end{cases}
\end{equation}
The law $\bar{\rho}_t$ satisfies the nonlinear Fokker--Planck equation
\bq\label{eq_pde}
    \partial_t \bar{\rho}_t=\nabla \cdot \left(\lambda (x-m_f^h[\bar{\rho}_t])\bar{\rho}_t\right)+\frac{\sigma^2}{2}\nabla^2 :\Big[(x-m_f^h[\bar{\rho}_t])\otimes(x-m_f^h[\bar{\rho}_t])\bar{\rho}_t \Big].
\eq

Here, for a sufficiently regular matrix field $A=(A_{k\ell})_{k,\ell=1}^d$, the notation
\[
\nabla^2:A:=\sum_{k,\ell=1}^d\partial_{x_kx_\ell}A_{k\ell}
\]
denotes the double divergence of $A$. Thus the last term in \eqref{eq_pde} is the Fokker--Planck diffusion corresponding to the rank-one covariance matrix $(x-m_f^h[\bar\rho_t])\otimes(x-m_f^h[\bar\rho_t])$.

\begin{theorem}[Uniform-in-time propagation of chaos]\label{thm: propagation of chaos}
Let $p\ge 2$, $q>2$, and $\alpha>0$. Let $\{X_0^i\}_{i=1}^N$ be i.i.d. random variables with common law $\bar\rho_0\in\calP_{pq}(\bbR^d)$, and let $\{X_t^i\}_{i=1}^N$ be the strong solution to the particle system \eqref{I: eq: main} with initial data $\{X_0^i\}_{i=1}^N$. Let $\bar\rho_t$ be the law of the McKean--Vlasov process $\bar X_t$ solving \eqref{eq: coupling} with initial law $\bar\rho_0$.  Assume that $f$ satisfies Assumption \ref{assum: omega}.   Then for any
\[
\lambda > \max\lt\{ (pq-1)(1+\Lambda_\alpha^{\frac{2}{pq}})\sigma^2 + 2(p-1)(3+\Lambda_\alpha^{\frac{2}{p}})\sigma^2,\, (pq-1)\Lambda_\alpha^{\frac{2}{pq}}\sigma^2 + 2(p-1) (1 + \Lambda_\alpha) \sigma^2  \rt\},
\]
there exists a positive constant $C>0$, depending on $d,\lambda,\sigma,p,q,\alpha,h(\alpha),L_{\omega_f^\alpha}$ and $\bar\rho_0$, but not on $N$ or time, such that
\begin{equation}\label{main_res}
\sup_{t\ge0}\bbE\calW_p^p(\mu_t^N,\bar\rho_t)
\le \frac{C}{N^{\frac p2\wedge\frac{q-2}{2}}}
+C\begin{cases}
N^{-\frac12}+N^{-\frac{q-1}{q}},& p>\frac d2,\\[3pt]
N^{-\frac12}\log(1+N)+N^{-\frac{q-1}{q}},& p=\frac d2,\\[3pt]
N^{-\frac pd}+N^{-\frac{q-1}{q}},& 0<p<\frac d2.
\end{cases}
\end{equation}
Here the nonlinear copies in the synchronous coupling are initialized by $\bar X_0^i=X_0^i$ and driven by the same Wiener processes as $X_t^i$. Since the $X_0^i$ are i.i.d. with law $\bar\rho_0$, the processes $\bar X_t^i$ remain i.i.d. with law $\bar\rho_t$. Thus the coupling error at $t=0$ vanishes, while the empirical sampling error does not.

Here $\calW_p$ represents the $p$-Wasserstein distance, and $\mu_t^N := \frac{1}{N}\sum_{i=1}^N\delta_{X_t^i}$ is the empirical measure associated with the particle system \eqref{I: eq: main}.
\end{theorem}

\begin{remark}
Synchronous initialization eliminates the \emph{particle--nonlinear-copy coupling error}, but it does not make the empirical sampling distance between $N^{-1}\sum_i\delta_{X_0^i}$ and $\bar\rho_0$ vanish. The latter is precisely the Fournier--Guillin term in \eqref{main_res}. In particular, \eqref{main_res} and Markov's inequality give the correct uniform-in-time probability statement
\[
\lim_{N\to\infty}\sup_{t\ge0}\bbP\big(\calW_p(\mu_t^N,\bar\rho_t)\ge\epsilon\big)=0,
\quad \epsilon>0.
\]
This establishes uniform-in-time propagation of chaos.
\end{remark}

We complement the particle-level analysis with the corresponding mean-field dynamics. The contraction mechanism yields deterministic consensus under a fixed-parameter stability threshold. We also give a direct weighted-energy estimate under Assumptions \ref{assum: omega} and \ref{assump: f hessian}. Section \ref{subsec: weak mf bridge} proves the consensus-point comparison and develops the regularized quantitative Laplace estimate used in the finite-time entrance argument.

We summarize these conclusions in the following theorem.

\begin{theorem}[Mean-field deterministic consensus and a direct weighted-energy estimate]
\label{thm:mf_main}
Let $\bar X_t$ denote a global solution to the McKean--Vlasov equation \eqref{eq: coupling} with initial law $\bar\rho_0\in\mathcal P_p(\R^d)$ for some $p\ge2$. Define the mean-field consensus threshold
\[
\bbflambda_{p, \alpha, \sigma}: = \bflambda_{p,\alpha,\sigma} + (p-1)\sigma^2 =(p-1)(1+\Lambda_\alpha^{2/p})\sigma^2.
\]
If $\lambda>\bbflambda_{p, \alpha, \sigma}$, then the following hold:
\begin{enumerate}[label=(\roman*)]
\item \emph{(Deterministic consensus)}  
There exists a deterministic point $x_\infty^\alpha\in \R^d$ such that
\[
\bar X_t \xrightarrow[t\to \infty]{} x_\infty^\alpha \quad\text{in }L^p(\Omega),
\]
or equivalently,
\[
\mathcal W_p(\bar\rho_t,\delta_{x_\infty^\alpha})\xrightarrow[t\to \infty]{}0.
\]

\item \emph{(Direct weighted-energy estimate)}  
Assume in addition that $f$ satisfies Assumptions \ref{assum: omega} and \ref{assump: f hessian}, and that $\alpha\ge c_1$. For this second-moment argument also require $\lambda>\bbflambda_{2,\alpha,\sigma}$; the threshold for a general $p>2$ does not automatically imply this condition. If there exists $0<\epsilon\le1$ such that
    \[
    (1-\epsilon)\|\omega_f^\alpha\|_{L^1(\bar{\rho}_0)}\geq \frac{\lambda L_{\omega_f^\alpha}\sqrt{2\Lambdaa}}{\lambda-\bbflambda_{2, \alpha, \sigma}}\textnormal{Var}(\bar{X}_0)^\frac{1}{2}
        +\frac{\alpha\sigma^2c_0e^{-\alpha \underline{f}}\Lambdaa}{2(\lambda-\bbflambda_{2, \alpha, \sigma})}\textnormal{Var}(\bar{X}_0),
        \]
then the limit point $x^\alpha_\infty$ satisfies the fixed-parameter estimate
\[
f(x^\alpha_\infty) \leq -\frac{1}{\alpha} \log \left\|\omega_f^\alpha\right\|_{L^1(\bar\rho_0)} -\frac{1}{\alpha}\log \varepsilon .
\]
\end{enumerate}
\end{theorem}

The particle consensus point can be estimated directly from the mean-field law at a finite time. The following corollary records this comparison and, under the uniform-in-time coupling assumptions, its consequence for the particle and mean-field consensus points.

\begin{corollary}[Mean-field estimates for the particle consensus point]\label{cor:cons-est}
Fix $p\ge2$ and $\alpha>0$, and assume the hypotheses of Theorem \ref{thm_main1}(i). Let $\bar\rho_t$ be the law of a global solution to \eqref{eq: coupling} with initial law
$\operatorname{Law}(X_{\rm in})$. Set
\[
\gamma_p:=\lambda-\bflambda_{p,\alpha,\sigma}>0, \quad R_{N,p}:=\sup_{t\ge0} \left(\bbE\calW_p^p(\mu_t^N,\bar\rho_t)\right)^{1/p}.
\]
Then, for every $T\ge0$,
\begin{equation}\label{eq:trans-opt-int}
\|X_\infty^{N,\alpha}-x^*\|_{L^p(\Omega)} \le C_{\rm tail}e^{-\gamma_pT} +R_{N,p} +\calW_p(\bar\rho_T,\delta_{x^*}),
\end{equation}
where $C_{\rm tail}$ is the constant in \eqref{eq: particle tail intro}.

If, in addition, the hypotheses of Theorem \ref{thm:mf_main}(i) hold, then
\begin{equation}\label{eq:cons-lim-comp}
\|X_\infty^{N,\alpha}-x_\infty^\alpha\|_{L^p(\Omega)} \le R_{N,p}.
\end{equation}
Under the hypotheses of Theorem \ref{thm: propagation of chaos}, the synchronous-coupling estimate further gives
\begin{equation}\label{eq:cons-lim-r}
\bbE|X_\infty^{N,\alpha}-x_\infty^\alpha|^p
\le C N^{-\beta_{p,q}}, \quad \beta_{p,q}:=\frac p2\wedge\frac{q-2}{2},
\end{equation}
where $C$ is independent of $N$.
\end{corollary}

Corollary \ref{cor:cons-est} is a quantitative consequence of consensus and approximation. It does not by itself choose parameters that simultaneously yield global contraction and localization near $x^*$. In particular, \eqref{eq:cons-lim-comp}--\eqref{eq:cons-lim-r} compare the two consensus points without identifying either one with the minimizer. The high-probability optimization theorem below uses convex trapping in place of a global consensus-tail estimate; it is not deduced by applying this corollary.

The next theorem gives an optimization guarantee from the weak-landscape condition, full support of the initial law, and local nondegeneracy of the minimizer. Given a target radius and failure probability, the parameters can be chosen for any fixed $\lambda>\sigma^2$. Its proof in Section \ref{subsec: explicit global route} combines Proposition \ref{prop: full support entrance}, Proposition \ref{prop: finite horizon}, and Theorem \ref{thm:conv-trap} through Corollary \ref{cor: convex finite transfer}; the large-$\lambda$ threshold of Theorem \ref{thm: propagation of chaos} is not required. Related probabilistic global-convergence statements have been obtained for CBO in different settings. The mean-field global-convergence theory in \cite{FKR24}, combined with quantitative particle approximation, yields probabilistic guarantees for numerical CBO. The weak uniform-in-time result \cite{BEZ26} gives joint convergence in population size and time on a bounded search domain, while the recent time-discrete finite-particle analysis \cite{HHK26} proves almost-sure consensus and near-optimality of the limiting consensus energy in a Hilbert-space setting under suitably prepared initial data. The theorem below has a different form: for the continuous-time scalar-noise regularized system on $\R^d$, the parameters are fixed throughout the evolution and, with prescribed high probability, all particles converge as $t\to\infty$ to a common random point in a prescribed neighborhood of $x^*$.

\begin{theorem}[Finite-particle global optimization from explicit input assumptions]
\label{thm: finite global optimization}
Assume Assumptions \ref{assum: omega}, \ref{assum: weak landscape}, and \ref{assum: local nondegenerate}. Let $\bar\rho_0\in\calP_{2+\delta_0}(\R^d)$ for some $\delta_0>0$ and suppose
\[
\operatorname{supp}\bar\rho_0=\R^d.
\]
Fix $\sigma>0$ and $\lambda>\sigma^2$. Then, for every target radius $\delta>0$ and every failure probability $\varepsilon_{\rm fail}\in(0,1)$, there exist an inverse temperature $\alpha>0$, a regularization value $h(\alpha)>0$, and $N_0\in\N$ such that, for every $N\ge N_0$, the particle system \eqref{I: eq: main} with i.i.d. initial law $\bar\rho_0$ satisfies
\[
\bbP\!\left(
\begin{array}{c}
\text{there exists a finite random vector }X_\infty^N\text{ such that}\\
X_t^i\to X_\infty^N\ \text{for every }i=1,\ldots,N, \quad |X_\infty^N-x^*|<\delta \end{array} \right)\ge1-\varepsilon_{\rm fail}.
\]
The parameters are fixed throughout the evolution; the proof does not stop the algorithm or change the coefficients after the finite entrance time used in the analysis.
\end{theorem}

\begin{remark}
In Theorem \ref{thm: finite global optimization}, the accuracy and failure probability are prescribed first. The proof then fixes the convex trapping energy, chooses $\alpha$ and $h(\alpha)$ to obtain finite-time entrance, and finally chooses $N$ sufficiently large. These parameters remain fixed during the evolution. By contrast, the uniform-in-time estimates hold for fixed parameters satisfying their stability thresholds. A simultaneous limit in $N$, $\alpha$, and $h$ would require additional uniform estimates and is not asserted here.
\end{remark}

%
%
%
%
%
%
%
%
%
%

\subsection{Organization of the paper}

The rest of the paper is organized as follows. Section \ref{sec_par} proves finite-particle well-posedness, exact consensus, the quantitative tail estimate, and the direct weighted-energy estimate. Section \ref{sec_propa} develops uniform-in-time propagation of chaos, together with localized quadratic stability and finite-horizon approximation without a contraction threshold. Section \ref{sec:mfs} begins with mean-field well-posedness and deterministic consensus. Section \ref{subsec: weak mf bridge} proves Corollary \ref{cor:cons-est} and gives the regularized quantitative Laplace estimate. Section \ref{subsec: explicit global route} combines finite-time entrance and convex trapping to prove Theorem \ref{thm: finite global optimization}. The appendices contain the exchangeability argument, finite-particle well-posedness details, and uniform-in-time stability estimates.

%
%
%
%
%
%
%
%
%
%

\section{Well-posedness and large-time analysis of the particle system}\label{sec_par}

This section proves Theorem \ref{thm_main1}. We proceed in three steps: (i) well-posedness of the finite-particle SDE system; (ii) quantitative $L^p(\Omega)$ contraction, a tail estimate, and almost sure convergence for every $p\ge2$; and (iii) the direct weighted-energy estimate under Assumptions \ref{assum: omega} and \ref{assump: f hessian}. The tail estimate is subsequently used in Corollary \ref{cor:cons-est}.

%
%
%
%
%
%
%
%
%
%

\subsection{Existence and uniqueness of strong solutions} 
We begin by establishing the well-posedness of the finite-particle system \eqref{I: eq: main},  which serves as the foundation for all subsequent estimates.  Since the analysis that follows relies on pathwise properties of $\{X_t^i\}_{i=1}^N$, we first confirm the global existence and uniqueness of strong solutions. Although the argument parallels that of \cite[Theorem 2.1]{CCTT18}, we provide a complete proof in Appendix \ref{sec: exist} for the sake of self-containment.

\begin{theorem}[Global well-posedness]\label{theorem: uniqueness}
Let $N\in\mathbb{N}$ and $d\ge1$.  
For each $i=1,\dots,N$, let $X_t^i:\Omega\to\mathbb{R}^d$ satisfy the SDE system
\eqref{I: eq: main} with initial data $X_0^i\in L^2(\Omega;\mathbb{R}^d)$.
Then the system admits a unique strong solution $\{X_t^i\}_{t\ge0}$ satisfying
\[
\mathbb{E}|X_t^i|^2 < \infty \quad \text{for all } t\ge0.
\]
\end{theorem}

%
%
%
%
%
%
%
%
%
%

\subsection{Convergence analysis in $L^{p}(\Omega)$}

In this subsection, we investigate the large-time behavior of the system in the $L^p(\Omega)$ sense. Our goal is to show that, under a suitable threshold condition on the interaction strength $\lambda$,  the pairwise distances between particles decay exponentially in time. This contraction property forms the cornerstone of the proof of Theorem \ref{thm_main1} (i), since it implies convergence of all trajectories to a common random limit, the consensus state.

A key simplification arises by expressing the dynamics in terms of relative positions. The next lemma quantifies the $L^p$-contraction of pairwise distances.
\begin{lemma}\label{lem: diff Lp}
    Let $p\geq 2$ and assume $\{X_0^i\}_{i=1}^N\subset L^p(\Omega)$. Then for each $1\leq i, j\leq N$ and $t\geq 0$, we have $X_t^i-X_t^j\in L^p(\Omega)$ and 
\[
        \bbE|X_t^i-X_t^j|^p\leq \max_{1 \leq i,j \leq N}\bbE|X_0^i-X_0^j|^pe^{-p(\lambda-(p-1)\Lambdaa^{2/p}\sigma^2)t},
\]
    where 
    \[
      \Lambda_\alpha=\frac{e^{-\alpha\underline{f}}+h(\alpha)}{h(\alpha)}.
      \]
    Consequently, if $\lambda>\bflambda_{p, \alpha, \sigma}=(p-1)\Lambdaa^{2/p}\sigma^2$, then
\[
        X_t^i-X_t^j\xrightarrow[t\to\infty]{}0\quad\textnormal{in }  L^p(\Omega).
\]
    In particular, we have
\bq\label{ineq_key}
       \max_{1 \leq i \leq N}\bbE|X_t^i - m_t^h|^{p} \le \Lambdaa  \max_{1 \leq i,j \leq N}\bbE|X_0^i - X_0^j|^{p} e^{-p(\lambda - \bflambda_{p, \alpha, \sigma})t}.
\eq
\end{lemma}
\begin{proof}
It suffices to consider $i\ne j$ in the It\^o calculation, since the difference is identically zero when $i=j$. Terms involving $|X_t^i-X_t^j|^{p-4}$ are understood through the Hessian of $|\cdot|^p$; for $p=2$ the terms with prefactor $p-2$ are absent. Localization and the finite-horizon moment bounds justify taking expectations.

We begin with the equation for $X_t^i-X_t^j$:
\[
    d(X_t^i-X_t^j)=-\lambda (X_t^i-X_t^j)dt+\sigma (X_t^i-m_t^h)dW_t^i-\sigma(X_t^j-m_t^h)dW_t^j.
\]
Applying the It\^o formula to $|X_t^i-X_t^j|^p$ yields
\begin{align*}
    d|X_t^i-X_t^j|^p=&-p\lambda |X_t^i-X_t^j|^pdt +\frac{p\sigma^2}{2}|X_t^i-X_t^j|^{p-2}\left(|X_t^i-m_t^h|^2+|X_t^j-m_t^h|^2\right)dt\\
    &+\frac{p(p-2)\sigma^2}{2}|X_t^i-X_t^j|^{p-4}\langle X_t^i-X_t^j, X_t^i-m_t^h\rangle ^2dt\\
    &+\frac{p(p-2)\sigma^2}{2}|X_t^i-X_t^j|^{p-4}\langle X_t^i-X_t^j, X_t^j-m_t^h\rangle ^2dt\\
    & +p\sigma|X_t^i-X_t^j|^{p-2}\langle X_t^i-X_t^j, X_t^i-m_t^h\rangle dW_t^i\\
    &-p\sigma|X_t^i-X_t^j|^{p-2}\langle X_t^i-X_t^j, X_t^j-m_t^h\rangle dW_t^j.
\end{align*}
Taking expectations eliminates the martingale terms and gives
\begin{align*}
    \frac{d}{dt}\bbE|X_t^i-X_t^j|^p=&-p\lambda \bbE|X_t^i-X_t^j|^p +\frac{p\sigma^2}{2}\bbE\left[|X_t^i-X_t^j|^{p-2}\left(|X_t^i-m_t^h|^2+|X_t^j-m_t^h|^2\right)\right]\\
    &+\frac{p(p-2)\sigma^2}{2}\bbE\left[|X_t^i-X_t^j|^{p-4}\langle X_t^i-X_t^j, X_t^i-m_t^h\rangle ^2\right]\\
    &+\frac{p(p-2)\sigma^2}{2}\bbE\left[|X_t^i-X_t^j|^{p-4}\langle X_t^i-X_t^j, X_t^j-m_t^h\rangle ^2\right].
\end{align*}

Next, from the definition of $m_t^h$ and the bounds $h(\alpha) \le \psi_h \le e^{-\alpha \underline{f}} + h(\alpha)$, we obtain by H\"older's inequality
\begin{align}\label{eq: xt - mt}
\begin{aligned} 
    \bbE|X_t^i-m_t^h|^p &=  \bbE\left|\frac{\sum_{k=1}^N(X_t^i-X_t^k)\psi_h(X_t^k)}{\sum_{k=1}^N\psi_h(X_t^k)}\right|^p \cr
  & \leq \bbE \lt| \frac{\lt(\sum_k |X_t^i-X_t^k|^p \psi_h(X_t^k)\rt)^{1/p} \lt(\sum_k \psi_h(X_t^k) \rt)^{(p-1)/p} }{ \sum_k \psi_h (X_t^k) } \rt|^p  \cr
&  \leq \frac{\Lambdaa}{N}\sum_{k=1}^N\bbE|X_t^i-X_t^k|^p \cr 
&  \leq \Lambdaa\max_{1 \leq i,j \leq N}\bbE|X_t^i-X_t^j|^p. 
\end{aligned}
\end{align}

Using this bound and again applying H\"older's inequality, we have for $k=i,j$:
\begin{align*} 
        \bbE\left[|X_t^i-X_t^j|^{p-2}| X_t^k-m_t^h|^2\right]
       & \leq \left(\bbE|X_t^i-X_t^j|^p\right)^{\frac{p-2}{p}}\left(\bbE|X_t^k-m_t^h|^p\right)^\frac{2}{p}\\
        & \leq \left(\bbE|X_t^i-X_t^j|^p\right)^{\frac{p-2}{p}}\left(\Lambdaa\max_{1 \leq i,j \leq N}\bbE|X_t^i-X_t^j|^p\right)^\frac{2}{p} \\
        & \leq \Lambdaa^\frac{2}{p}\max_{1 \leq i,j \leq N}\bbE|X_t^i-X_t^j|^p.
\end{align*}
Substituting these estimates into the previous differential inequality, we arrive at
\[
    \frac{d}{dt}\bbE|X_t^i-X_t^j|^p\leq -p\lambda \bbE|X_t^i-X_t^j|^p+p(p-1)\Lambdaa^\frac{2}{p}\sigma^2\max_{1 \leq i,j \leq N}\bbE|X_t^i-X_t^j|^p.
\]

To justify the maximum carefully, define
\[
M_p(t):=\max_{1\le i,j\le N}\bbE|X_t^i-X_t^j|^p.
\]
Each function $t\mapsto \bbE|X_t^i-X_t^j|^p$ is locally absolutely continuous, and the maximum is taken over finitely many pairs; hence $M_p$ is locally absolutely continuous. At almost every time, all pairwise moment functions and $M_p$ are differentiable; choose an active pair $(i_t,j_t)$ satisfying $M_p(t)=\bbE|X_t^{i_t}-X_t^{j_t}|^p$. The differential inequality above then yields
\[
M_p'(t)\le -p\lambda M_p(t)+p(p-1)\Lambda_\alpha^{2/p}\sigma^2M_p(t)
=-p\big(\lambda-(p-1)\Lambda_\alpha^{2/p}\sigma^2\big)M_p(t)
\]
for almost every $t\ge0$. The almost-everywhere differential inequality is sufficient since $M_p$ is locally absolutely continuous. Thus the maximizing pair is allowed to depend on time, and Gr\"onwall's lemma applied to $M_p$ gives
\[
M_p(t)\le M_p(0)e^{-p(\lambda-(p-1)\Lambda_\alpha^{2/p}\sigma^2)t}
=\max_{1\le i,j\le N}\bbE|X_0^i-X_0^j|^p e^{-p(\lambda-(p-1)\Lambda_\alpha^{2/p}\sigma^2)t}.
\]

Finally, combining this with \eqref{eq: xt - mt} deduces the desired estimate \eqref{ineq_key}.
\end{proof}

Next, we show that the particles themselves remain uniformly bounded in $L^p(\Omega)$ for all times.  In particular, Lemma \ref{lem: diff Lp} already guarantees exponential decay of the inter-particle distances in $L^p(\Omega)$, and we now use this property to derive a uniform bound for each trajectory $X_t^i$.

\begin{lemma}\label{lem: Lp bound}
    Let $p\geq 2$. If $\lambda> \boldsymbol{\lambda}_{p,\alpha,\sigma}$ and $\{X_0^i\}_{i=1}^N\subset L^{p}(\Omega)$, then we have
\[
        X_t^i\in L^\infty((0,\infty); L^{p}(\Omega)) \quad \text{uniformly in $i$}.
\]
    More specifically, there is a constant $C_*$ dependent only on $X_0^i,\lambda,p,\Lambda_\alpha$, and $\sigma$ such that
\[
        \max_{1\leq i\leq N}\sup_{t\geq 0}\mathbb{E}|X_t^i|^{p}\leq C_{*}.
\]

\end{lemma}
\begin{proof}
    Using the equation for each particle $X_t^i$ and the convexity of $z\mapsto z^{p}$, we estimate
    \begin{align*}
        |X_t^i|^{p}&=\left|X_0^i-\lambda\int_0^t(X_s^i-m_s^h)ds+\sigma\int_0^t(X_s^i-m_s^h)dW_s^i\right|^{p}\\
        &\leq 3^{p-1}|X_0^i|^{p}+3^{p-1}\lambda^{p}\left(\int_0^t|X_s^i-m_s^h|ds\right)^{p}+3^{p-1}\sigma^{p}\left|\int_0^t(X_s^i-m_s^h)dW_s^i\right|^{p}.
    \end{align*}
    
  Here, the stochastic integral
\[
      M_t:=  \int_0^t(X_s^i-m_s^h)dW_s^i
\]
    is a continuous $\mathbb{P}$-martingale. Indeed, using It\^o's isometry and Lemma \ref{lem: diff Lp}, we obtain
\[
        \mathbb{E}|M_t|^2=\mathbb{E}\int_0^t|X_s^i-m_s^h|^2ds \leq \int_0^t\left(\mathbb{E}|X_s^i-m_s^h|^{p}\right)^\frac{2}{p}ds \leq \frac{1}{2\left(\lambda -\boldsymbol{\lambda}_{p, \alpha, \sigma}\right)}\left(\Lambdaa\max_{1 \leq i,j \leq N}\mathbb{E}|X_0^i-X_0^j|^{p}\right)^\frac{2}{p} .
\]
    The finiteness of this bound ensures that $M_t$ belongs to $L^2(\Omega)$ for each $t\ge0$, 
and hence $M_t$ is a square-integrable continuous martingale.

We next control higher moments of $M_t$. By the Burkholder--Davis--Gundy inequality (see, e.g. \cite[Theorem 19.20]{S21}), there exists a constant $C_{\textnormal{BDG},p}>0$ depending only on $p$ such that 
\bq\label{eq: xs-ms ws}
        \mathbb{E}\left[\left|\int_0^t(X_s^i-m_s^h)dW_s^i\right|^{p}\right]\leq C_{\textnormal{BDG},p}\mathbb{E}\left(\int_0^t |X_s^i-m_s^h |^2ds\right)^{\frac{p}{2}}.
\eq
Taking expectations in the above inequality for $|X_t^i|^p$ and using \eqref{eq: xs-ms ws}, we get
\[
        \mathbb{E}|X_t^i|^{p}\le  3^{p-1}\mathbb{E} |X_0^i |^{p}+3^{p-1}\lambda^{p} I_t^i +3^{p-1}\sigma^{p}C_{\textnormal{BDG},p} II_t^i, 
\]
    where
\[
        I_t^i := \bbE\left(\int_0^t |X_s^i - m_s^h| \,ds\right)^{p}, \quad II_t^i := \bbE\left(\int_0^t |X_s^i - m_s^h|^2\,ds\right)^\frac{p}{2}.
\]
    
    We first estimate $I_t^i$. By Minkowski's integral inequality and Lemma \ref{lem: diff Lp}, we have
        \begin{equation*}%
        \begin{split}
I_t^i     &\leq \left(\int_0^t\left(\mathbb{E} |X_s^i-m_s^{h} |^{p}\right)^\frac{1}{p} \, ds\right)^p\\
            &\leq \left(\int_0^t\left(\Lambdaa \max_{1 \leq i,j \leq N}\mathbb{E}|X_0^i-X_0^j|^{p}\right)^\frac{1}{p}e^{-(\lambda-\boldsymbol{\lambda}_{p,\alpha,\sigma})s}ds\right)^p\\
            &\leq \frac{\Lambdaa}{\left(\lambda -\boldsymbol{\lambda}_{p,\alpha,\sigma}\right)^p}\max_{1 \leq i,j \leq N}\mathbb{E}|X_0^i-X_0^j|^{p}.
        \end{split} 
        \end{equation*}
A similar argument yields for $II_t^i$,
        \begin{equation} \label{eq: xs - ms quad p}
        \begin{split}
            II_t^i   &\leq \left(\int_0^t\left(\Lambdaa \max_{1 \leq i,j \leq N}\mathbb{E}|X_0^i-X_0^j|^{p}\right)^\frac{2}{p}e^{-2(\lambda-\bflambda_{p,\alpha,\sigma})s}ds\right)^{\frac{p}{2}} \leq \frac{\Lambdaa}{\left(2(\lambda -\boldsymbol{\lambda}_{p,\alpha,\sigma} )\right)^{\frac{p}{2}}}\max_{1 \leq i,j \leq N}\mathbb{E}|X_0^i-X_0^j|^{p} .
        \end{split}
        \end{equation}
    Collecting the estimates for $I_t^i$ and $II_t^i$ above, we obtain the desired uniform bound
\[
        \max_{1\leq i\leq N}\sup_{t\geq 0}\mathbb{E}|X_t^i|^{p}\leq C_*,
\]
where  $C_* > 0$ is constant depending only on $\lambda,\sigma,p,\Lambda_\alpha$, and the initial data. 
\end{proof}
 
In the previous lemma, we established that each process $X_t^i$ remains uniformly bounded in $L^{p}(\Omega)$ under suitable assumptions on the initial data and the parameter $\lambda$. 
We now turn to the large-time behavior of the particle positions in $L^{p}(\Omega)$.  The next result shows that all particle trajectories converge in $L^{p}(\Omega)$ to a common random limit.

\begin{lemma} \label{lem: rand vec}
    Let $p\ge2$, $\lambda > \boldsymbol{\lambda}_{p,\alpha,\sigma}$, and assume $\{X_0^i\}_{i=1}^N\subset L^{p}(\Omega)$. Then there exists a random variable $X_\infty\in L^{p}(\Omega)$ such that 
\[
 \lim_{t\to\infty}X_t^i=X_\infty \quad \text{in $L^{p}(\Omega)$}, \quad i=1,2,\ldots, N.
\]
Moreover, setting
\[
\gamma_p:=\lambda-\bflambda_{p,\alpha,\sigma},\quad
A_{N,p}:=\left(\Lambda_\alpha\max_{1\le i,j\le N}\bbE|X_0^i-X_0^j|^p\right)^{1/p},
\]
one has, for every $t\ge0$ and every $i$,
\begin{equation}\label{eq: particle tail}
\|X_t^i-X_\infty\|_{L^p(\Omega)}\le
A_{N,p}\left(\frac{\lambda}{\gamma_p}+\frac{\sigma C_{\rm BDG,p}^{1/p}}{\sqrt{2\gamma_p}}\right)e^{-\gamma_p t}.
\end{equation}
\end{lemma}
\begin{proof}
We start from the integral representation of $X_t^i$:
\[
        X_t^i =X_0^i-\lambda \int_0^t(X_s^i-m_s^h)ds+\sigma \int_0^t(X_s^i-m_s^h)dW_s^i =: X_0^i-\lambda I_t^i +\sigma II_t^i .
\]
    
We first analyze the martingale term $II_t^i$. From estimates \eqref{eq: xs-ms ws} and \eqref{eq: xs - ms quad p}, we recall that      
\[
\mathbb{E} \lt[|II_t^i|^p\rt] \leq \frac{C_{\textnormal{BDG},p} \Lambdaa}{\left(2(\lambda-\boldsymbol{\lambda}_{p,\alpha,\sigma})\right)^\frac{p}{2}} \max_{1 \leq i,j \leq N}\mathbb{E}|X_0^i-X_0^j|^{p},
\]
    which implies in particular that
\[
        \sup_{t\geq 0}\mathbb{E} \lt[|II_t^i|^p\rt] < \infty.
\]
Hence the sequence $\{II_t^i\}_{t \ge 0}$ is uniformly bounded in $L^p(\Omega)$. Since $II_t^i$ is the stochastic integral of a progressively measurable and square-integrable process, it defines an $L^p$-martingale. By the martingale convergence theorem, there exists a random variable $II_\infty^i \in L^p(\Omega)$ such that
\[
    II_t^i \xrightarrow[t \to \infty]{} II_\infty^i 
    \quad \text{in } L^p(\Omega) \text{ and almost surely.}
\]

We next examine the drift term $I_t^i$.  Using Minkowski's inequality, for any $t \ge 0$ we estimate
    \begin{align*}
     \bbE \left|\int_0^\infty (X_s^i - m_s^h) \,ds - I_t^i  \right|^{p} &= \mathbb{E}\left|\int_t^\infty(X_s^i-m_s^h)ds\right|^{p} \\
        &\leq \left(\int_t^\infty\left(\mathbb{E}|X_s^i-m_s^h|^{p}\right)^\frac{1}{p}ds\right)^{p}\\
        &\leq \left(\int_t^\infty \left(\Lambdaa\max_{1 \leq i,j \leq N}\mathbb{E}|X_0^i-X_0^j|^{p}\right)^\frac{1}{p}e^{-(\lambda-\boldsymbol{\lambda}_{p, \alpha, \sigma})s}ds\right)^{p}\\
        &\leq \frac{\Lambdaa\max_{1 \leq i,j \leq N}\mathbb{E}|X_0^i-X_0^j|^{p}}{(\lambda -\boldsymbol{\lambda}_{p, \alpha, \sigma})^{p}}e^{-p(\lambda-\boldsymbol{\lambda}_{p, \alpha, \sigma})t} \\
        &\xrightarrow[t \to \infty]{} 0.
    \end{align*}
Hence, we have
\[
    \int_0^t (X_s^i - m_s^h)\,ds    \xrightarrow[t \to \infty]{}  \int_0^\infty (X_s^i - m_s^h)\,ds \quad \text{in } L^p(\Omega).
\]
 
 Combining both limits, we obtain the existence of 
    \[
    X_\infty^i := X_0^i - \lambda \int_0^\infty (X_s^i - m_s^h)\,ds + \sigma \int_0^\infty(X_s^i-m_s^h)\,dW_s^i \in L^p(\Omega),
\]
such that
\bq\label{eq: lim i}
        \lim_{t\to\infty} X_t^i  = X_\infty^i \quad \text{in }L^{p}(\Omega).
\eq

The same estimates give the quantitative tail. Indeed, by Minkowski's inequality and Lemma \ref{lem: diff Lp},
\[
\left\|\int_t^\infty(X_s^i-m_s^h)\,ds\right\|_{L^p}
\le \frac{A_{N,p}}{\gamma_p}e^{-\gamma_p t}.
\]
For the stochastic tail, the Burkholder--Davis--Gundy and Minkowski inequalities yield
\[
\left\|\int_t^\infty(X_s^i-m_s^h)\,dW_s^i\right\|_{L^p} \le C_{\rm BDG,p}^{1/p} \left\|\left(\int_t^\infty|X_s^i-m_s^h|^2ds\right)^{1/2}\right\|_{L^p} \le \frac{C_{\rm BDG,p}^{1/p}A_{N,p}}{\sqrt{2\gamma_p}}e^{-\gamma_p t}.
\]
Combining these two bounds with the integral representation of $X_\infty-X_t^i$ proves \eqref{eq: particle tail}.

Finally, to identify the common limit, we note that Lemma \ref{lem: diff Lp} ensures
\[
    \mathbb{E}|X_t^i - X_t^j|^p     \to 0 \quad \text{as } t \to \infty.
\]
This and \eqref{eq: lim i} imply $X_\infty^i = X_\infty^j$ for all $1 \le i, j \le N$, and the proof is complete.    
\end{proof}

%
%
%
%
%
%
%
%
%
%

\subsection{Almost sure convergence}%

Almost sure convergence follows directly from the integrability estimates in Lemma \ref{lem: diff Lp} and the martingale convergence theorem.

\begin{proof}[Proof of the almost sure part of Theorem \ref{thm_main1} (i)]
Let $p\ge2$ and $\gamma_p=\lambda-\bflambda_{p,\alpha,\sigma}>0$. By \eqref{ineq_key},
\[
\int_0^\infty\bbE|X_s^i-m_s^h|\,ds \le \int_0^\infty\left(\bbE|X_s^i-m_s^h|^p\right)^{1/p}ds \le \frac{A_{N,p}}{\gamma_p}<\infty.
\]
Hence Fubini's theorem gives
\[
\int_0^\infty|X_s^i-m_s^h|\,ds<\infty \quad\text{almost surely},
\]
so the drift integral converges absolutely almost surely. On the other hand,
\[
M_t^i:=\int_0^t(X_s^i-m_s^h)\,dW_s^i
\]
is an $L^p$-bounded martingale by \eqref{eq: xs - ms quad p}; since $p\ge2>1$, the martingale convergence theorem implies that $M_t^i$ converges almost surely and in $L^p$. Hence, the integral representation of $X_t^i$ shows that $X_t^i$ converges almost surely to a random vector $X_\infty^i$. The $L^p$ convergence already proved in Lemma \ref{lem: rand vec}, together with
\[
\bbE|X_t^i-X_t^j|^p\longrightarrow0,
\]
shows that $X_\infty^i=X_\infty^j$ almost surely for all $i,j$. Thus every particle converges almost surely to the same $X_\infty$ for every $p\ge2$.
\end{proof}

The preceding analysis has focused on a model where a $d$-dimensional state is driven by a one-dimensional Wiener process. It is worth noting that the core result of Lemma \ref{lem: diff Lp}, namely the exponential contraction condition, can be adapted for different noise structures. The following remark summarizes how the consensus threshold $\bflambda_{p, \alpha, \sigma}$
  is adjusted for several common variations.
\begin{remark}\label{rem: noise structures}
The baseline consensus threshold established in Lemma \ref{lem: diff Lp},
    \begin{align*}
        \bflambda_{p, \alpha, \sigma}=(p-1)\Lambdaa^{2/p}\sigma^2 \quad \text{with }  \Lambda_\alpha=\frac{e^{-\alpha\underline{f}}+h(\alpha)}{h(\alpha)}
    \end{align*}
corresponds to the case where each $d$-dimensional state vector is driven by an independent one-dimensional Wiener process.  Depending on the dimensionality and structure of the noise, this threshold changes as follows.  
In all cases below, the same assumptions as in Lemma \ref{lem: diff Lp} still apply. 
    \begin{itemize}
        \item {\bf Model with one-dimensional common-direction noise.} Consider
        \begin{equation*}
            dX_t^i=-\lambda (X_t^i-m_t^h)dt+\sigma|X_t^i-m_t^h|\mathbf{1}_d^\top dW_t^i,
        \end{equation*}
        where $\{W_t^i\}_{i=1}^N$ are one-dimensional independent Wiener processes. In this case, the effective diffusion acts identically along all $d$ components, leading to a threshold increased by a factor of $d$:
        \begin{equation*}
            \bbE|X_t^i-X_t^j|^{p}\leq \max_{1 \leq i,j \leq N}\bbE|X_0^i-X_0^j|^{p}e^{-p(\lambda -d\bflambda_{p, \alpha, \sigma})t}.
        \end{equation*}
        \item {\bf Model with $d$-dimensional anisotropic noise.} Consider
        \begin{equation*}
            dX_t^i=-\lambda(X_t^i-m_t^h)dt + \sigma(X_t^i-m_t^h) \odot dW_t^i,
        \end{equation*}
        where $\{W_t^i\}_{i=1}^N$ are $d$-dimensional independent Wiener processes and $\odot$ denotes the Hadamard (component) product. Since the stochastic perturbations are applied componentwise, the previous argument carries over directly, and the threshold remains unchanged:
        \begin{align*}
            \bbE|X_t^i-X_t^j|^p\leq \max_{1 \leq i,j \leq N}\bbE|X_0^i-X_0^j|^pe^{-p(\lambda-\bflambda_{p, \alpha, \sigma})t}.
        \end{align*}
        \item {\bf Model with $d$-dimensional isotropic noise.} Consider
        \begin{equation*}
            dX_t^i=-\lambda(X_t^i-m_t^h)dt+\sigma|X_t^i-m_t^h|dW_t^i,
        \end{equation*}
        where $\{W_t^i\}_{i=1}^N$ are $d$-dimensional independent Wiener processes. Here, the magnitude of the stochastic forcing scales with the state discrepancy $|X_t^i - m_t^h|$, leading to a modified threshold
        \begin{equation*}
            \bbE|X_t^i-X_t^j|^p\leq \max_{1 \leq i,j \leq N}\bbE|X_0^i-X_0^j|^pe^{-p(\lambda-\bflambda_{p,d, \alpha, \sigma})t},
        \end{equation*}
where $\bflambda_{p, d, \alpha, \sigma}:=\bflambda_{p, \alpha, \sigma}+(d-1)\Lambdaa^{2/p}\sigma^2$.
    \end{itemize}
\end{remark}


\subsection{Direct weighted-energy estimate}

We now derive a direct weighted-energy estimate for the random consensus variable under the stronger smoothness assumptions. The resulting essential-infimum bound complements the $L^p$ location estimates in Corollary \ref{cor:cons-est}.

\begin{lemma}\label{lem_imp}
    Let $\{X_t^i\}_{1\le i \le N}$ be the global strong solution to the system \eqref{I: eq: main}. Under Assumptions \ref{assum: omega} and \ref{assump: f hessian}, assume that $X_{\rm in}\in L^p(\Omega)$, $\lambda>\bflambda_{p, \alpha, \sigma}$, and $\alpha \ge c_1$. Then, in the almost-everywhere (equivalently, integrated) sense,
    \[
     \frac{d}{dt}\left(\frac{1}{N}\sum_{i=1}^N\mathbb{E}\omega_f^\alpha(X_t^i)\right) 
     \ge - \frac{\lambda L_{\omega_f^\alpha}}{N} \sum_{i=1}^N(\bbE|X_t^i-m_t^h|^2)^{\frac{1}{2}}-\frac{\alpha \sigma^2c_0e^{-\alpha\underline{f}}}{2N}\sum_{i=1}^N\mathbb{E}|X_t^i-m_t^h|^2.
    \]
    \end{lemma}
    \begin{proof}

The computation is performed first up to exit from bounded sets. The lower Hessian bound $\nabla^2\omega_f^\alpha\succeq-\alpha c_0e^{-\alpha\underline f}I$ for $\alpha\ge c_1$, together with the moment bounds and bounded $\nabla\omega_f^\alpha$, justifies passage to the integrated lower inequality by localization and Fatou's lemma.
  
 Applying the It\^o chain rule to $\frac{1}{N}\sum_{i=1}^N \omega_f^\alpha(X_t^i)$ yields
    \begin{align*}
        d\left(\frac{1}{N}\sum_{i=1}^N\omega_f^\alpha(X_t^i)\right)&=-\frac{\alpha}{N} \sum_{i=1}^N\omega_f^\alpha(X_t^i)\nabla f(X_t^i)\cdot dX_t^i+\frac{1}{2N}\sum_{i=1}^N\text{Tr}(dX_t^i[dX_t^i]^\top\nabla^2\omega_f^\alpha(X_t^i)) .
    \end{align*}

For completeness, the second-order term can be written explicitly. Since $\omega_f^\alpha=e^{-\alpha f}$,
\[
\nabla\omega_f^\alpha=-\alpha\omega_f^\alpha\nabla f, \quad \nabla^2\omega_f^\alpha =\alpha\omega_f^\alpha\big(\alpha\nabla f\nabla f^\top-\nabla^2f\big).
\]
Moreover, since $W_t^i$ is one-dimensional,
\[
d[X^i]_t
=\sigma^2(X_t^i-m_t^h)(X_t^i-m_t^h)^\top dt.
\]
Using $\operatorname{Tr}((u u^\top)A)=u^\top A u$, the It\^o correction for the $i$th particle is
\[
\frac{\sigma^2\alpha}{2}\omega_f^\alpha(X_t^i) (X_t^i-m_t^h)^\top \big(\alpha\nabla f(X_t^i)\nabla f(X_t^i)^\top-\nabla^2f(X_t^i)\big) (X_t^i-m_t^h)\,dt.
\]

Substituting the dynamics \eqref{I: eq: main} into the above, we obtain
    \begin{align*}
        d\left(\frac{1}{N}\sum_{i=1}^N\omega_f^\alpha(X_t^i)\right)&=\frac{\alpha\lambda}{N}\sum_{i=1}^N\omega_f^\alpha(X_t^i)\nabla f(X_t^i)\cdot (X_t^i-m_t^h)dt -\frac{\alpha\sigma}{N}\sum_{i=1}^N\omega_f^\alpha(X_t^i)\nabla f(X_t^i)\cdot (X_t^i-m_t^h) dW_t^i\\
        &\quad +\frac{\sigma^2\alpha}{2N}\sum_{i=1}^N\omega_f^\alpha(X_t^i)(X_t^i-m_t^h)^{\top}(\alpha\nabla f(X_t^i)\nabla f(X_t^i)^{\top}-\nabla^2 f(X_t^i))(X_t^i-m_t^h)dt.
    \end{align*}
Taking expectations eliminates the stochastic integral term and gives
        \begin{align*}
        \frac{d}{dt}\left(\frac{1}{N}\sum_{i=1}^N \bbE \omega_f^\alpha(X_t^i)\right)&=\frac{\alpha\lambda}{N} \sum_{i=1}^N \mathbb{E}\Big[\omega_f^\alpha(X_t^i)\nabla f(X_t^i)\cdot (X_t^i-m_t^h)\Big]\\
    &\quad +\frac{\alpha\sigma^2}{2N}\sum_{i=1}^N\bbE\left[\omega_f^\alpha(X_t^i)(X_t^i-m_t^h)^{\top}(\alpha\nabla f(X_t^i)\nabla f(X_t^i)^{\top}-\nabla^2 f(X_t^i))(X_t^i-m_t^h)\right].
    \end{align*}
    
For the first term on the right-hand side, we use $-\alpha\omega_f^\alpha\nabla f=\nabla\omega_f^\alpha$ and Assumption \ref{assum: omega} to obtain
    \[
        \mathbb{E} \Big[\alpha \omega_f^\alpha(X_t^i)\nabla f(X_t^i)\cdot [X_t^i-m_t^h] \Big]  \geq -L_{\omega_f^\alpha}\bbE |X_t^i-m_t^h|.
    \]  
For the second term, by Assumption \ref{assump: f hessian}, we have
    \begin{align*}
   &     \bbE\left[\omega_f^\alpha(X_t^i)(X_t^i-m_t^h)^{\top}\nabla ^2f(X_t^i)(X_t^i-m_t^h)\right] \\
   &\quad \leq c_0e^{-\alpha\underline{f}}\bbE|X_t^i-m_t^h|^2 +c_1\bbE\left[\omega_f^\alpha(X_t^i)(\nabla f(X_t^i)\cdot (X_t^i-m_t^h))^2\right].
    \end{align*}
This gives
    \begin{align*}
&  \bbE\left[\omega_f^\alpha(X_t^i)(X_t^i-m_t^h)^{\top}(\alpha\nabla f(X_t^i)\nabla f(X_t^i)^{\top}-\nabla^2 f(X_t^i))(X_t^i-m_t^h)\right] \cr
&\quad \geq -  c_0e^{-\alpha\underline{f}} \,\mathbb{E}|X_t^i-m_t^h|^2 +   (\alpha-c_1)   \mathbb{E}\Big[\omega_f^\alpha(X_t^i)(\nabla f(X_t^i)\cdot (X_t^i-m_t^h))^2\Big]\cr
&\quad \geq  -  c_0e^{-\alpha\underline{f}} \,\mathbb{E}|X_t^i-m_t^h|^2,
    \end{align*}
where the last inequality follows from $\alpha \ge c_1$.
    
Combining the two bounds above yields the desired differential inequality.
\end{proof}

We are now in a position to provide the proof of Theorem \ref{thm_main1} (ii).

\begin{proof}[Proof of Theorem \ref{thm_main1} (ii)]
If $N=1$, then $m_t^h=X_t^1$ and the particle is constant, so the conclusion is immediate. Hence assume $N\ge2$. By Lemma \ref{lem_imp}, we find
\begin{align*}
 \frac{d}{dt}\left(\frac{1}{N}\sum_{i=1}^N\mathbb{E}\omega_f^\alpha(X_t^i)\right) 
    &\ge-\frac{\lambda L_{\omega_f^\alpha}}{N}\sum_{i=1}^N\left(\bbE|X_t^i-m_t^h|^2\right)^\frac{1}{2}-\frac{\alpha \sigma^2c_0e^{-\alpha\underline{f}}}{2N}\sum_{i=1}^N\mathbb{E}|X_t^i-m_t^h|^2\\
    &\geq -\lambda L_{\omega_f^\alpha}\left(\Lambdaa \max_{1 \leq i,j \leq N}\mathbb{E}|X_0^i-X_0^j|^{2}\right)^\frac{1}{2}e^{-(\lambda-\boldsymbol{\lambda}_{2, \alpha, \sigma})t}\\
    &\quad \,-\frac{\alpha \sigma^2c_0e^{-\alpha\underline{f}}}{2}\left(\Lambdaa  \max_{1 \leq i,j \leq N}\mathbb{E}|X_0^i-X_0^j|^{2}\right)e^{-2(\lambda-\boldsymbol{\lambda}_{2, \alpha, \sigma})t},
\end{align*}
where the second inequality follows by Lemma \ref{lem: diff Lp} (applied with $p=2$). Now integrating in $t$ gives
\begin{align*}
    \frac{1}{N}\sum_{i=1}^N\mathbb{E}\omega_f^\alpha(X_t^i)&\geq \mathbb{E}\omega_f^\alpha(X_{\rm in})-\lambda L_{\omega_f^\alpha}\left(2\Lambdaa \text{Var}(X_{\rm in})\right)^\frac{1}{2}\frac{1-e^{-(\lambda-\boldsymbol{\lambda}_{2, \alpha, \sigma})t}}{\lambda-\boldsymbol{\lambda}_{2, \alpha, \sigma}}\\
    &\quad\;-\frac{\alpha \sigma^2c_0e^{-\alpha\underline{f}}\Lambdaa(1-e^{-2(\lambda-\boldsymbol{\lambda}_{2, \alpha, \sigma})t})}{2(\lambda-\boldsymbol{\lambda}_{2, \alpha, \sigma})}\text{Var}(X_{\rm in}),
\end{align*}
For $i\ne j$, independence gives
\[
\mathbb{E}|X_0^i-X_0^j|^2=2\operatorname{Var}(X_{\rm in}),
\]
whereas the left-hand side is $0$ when $i=j$. Consequently, for $N\ge2$,
\[
\max_{1\le i,j\le N}\mathbb E|X_0^i-X_0^j|^2=2\operatorname{Var}(X_{\rm in}),
\]
which is the identity actually used above.

Using that $X_t^i\to X_\infty$ in $L^1(\Omega)$, and that $\omega_f^\alpha$ is Lipschitz, we can pass to the limit $t\to\infty$ to obtain
\begin{align*}
    \mathbb{E}\omega_f^\alpha(X_\infty)&\geq \mathbb{E}\omega_f^\alpha(X_{\rm in})-\frac{\lambda L_{\omega_f^\alpha}}{\lambda-\boldsymbol{\lambda}_{2, \alpha, \sigma}}\left(2\Lambdaa \text{Var}(X_{\rm in})\right)^\frac{1}{2}-\frac{\alpha \sigma^2c_0\Lambdaa e^{-\alpha\underline{f}}}{2(\lambda-\boldsymbol{\lambda}_{2, \alpha, \sigma})}\text{Var}(X_{\rm in})\\
    &\geq \epsilon\mathbb{E}\omega_f^\alpha(X_\textnormal{in}).
\end{align*}
Hence, taking logarithms, we have
\[
-\frac{1}{\alpha} \log \mathbb{E}\!\left[e^{-\alpha f(X_\infty)}\right] \leq -\frac{1}{\alpha} \log \mathbb{E}\!\left[e^{-\alpha f(X_{\rm in})}\right] -\frac{1}{\alpha}\log\varepsilon .
\]
Since, for every $\alpha>0$,
\[
\operatorname*{ess\,inf}_{\omega\in\Omega} f(X_\infty(\omega)) \leq -\frac{1}{\alpha} \log \mathbb{E}\!\left[e^{-\alpha f(X_\infty)}\right],
\]
we conclude that
\[
\operatorname*{ess\,inf}_{\omega\in\Omega} f(X_\infty(\omega)) \leq -\frac{1}{\alpha} \log \mathbb{E}\!\left[e^{-\alpha f(X_{\rm in})}\right] -\frac{1}{\alpha}\log\varepsilon .
\]
This proves the claimed fixed-parameter estimate.
\end{proof}

%
%
%
%
%
%
%
%
%
%

\section{Uniform-in-time propagation of chaos}\label{sec_propa}

In this section, we establish our second main result, Theorem \ref{thm: propagation of chaos}, 
which provides a quantitative mean-field limit for the modified CBO dynamics. More precisely, we show that the empirical distribution associated with the particle system \eqref{I: eq: main} converges uniformly in time toward the law of its McKean--Vlasov counterpart, with a rate depending only on the initial data and model parameters. This result rigorously confirms that the collective dynamics of a large ensemble of particles can be faithfully described by the mean-field equation \eqref{eq: coupling}, even in the long-time regime when the system approaches consensus.

To prepare for the analysis, we first introduce basic notation that will be used throughout this section. For a probability measure $\mu\in\calP(\R^d)$ we denote
\[
  \calM(\mu) :=\intr x\mu(dx), \quad V_p(\mu) :=\intr|x-\calM(\mu)|^p\mu(dx), \quad \textswab{M}_p(\mu) :=\intr |x|^p\mu(dx),
\]
corresponding respectively to the mean, centered $p$-th moment, and absolute $p$-th moment of $\mu$.

For the particle system \eqref{I: eq: main}, we introduce the empirical measure
\begin{align*}
    \mu_t^N = \frac{1}{N}\sum_{i=1}^N \delta_{X_t^i}
\end{align*}
so that the particle dynamics can be written concisely as
\begin{equation} \label{eq:dyn-rew}
        dX_t^i = -\lambda (X_t^i-m_{ f}^{h}[\mu_t^N])dt+\sigma(X_t^i-m_f^{h}[\mu_t^N])dW_t^i.
\end{equation}
%
%
%
%
%
%
%
%
%
%
\subsection{Preliminary estimates and stability observations}\label{subsec: apriori obs}

We begin with several elementary but essential estimates that will serve as preliminary ingredients for the proof of Theorem \ref{thm: propagation of chaos}. They provide quantitative control on the fluctuations of various averaged quantities, in particular on the sensitivity of the weighted mean $m_f^h[\mu]$ with respect to perturbations of the underlying measure $\mu$.

The key difficulty in deriving a stability estimate for \eqref{eq:dyn-rew} lies in controlling the nonlinear dependence of the drift and diffusion terms on the empirical law through $m_f^h[\mu_t^N]$. While the usual mean operator $\calM$ satisfies the optimal Lipschitz bound
\[
    |\calM(\mu) - \calM(\nu)| \le \calW_1 (\mu, \nu), \quad \forall\, \mu, \nu \in \calP_1(\R^d),
\]
the same is not immediately true for the weighted mean $m_f^h[\mu]$, whose definition involves the nonlinearity of the (regularized) Gibbs weight. To overcome this difficulty, one can interpret $m_f^h[\mu]$ as a perturbation of the simple mean, i.e.,
\[
    m_f^h[\mu] - m_f^h[\nu] = (m_f^h[\mu] - \calM(\mu)) + (\calM(\mu) - \calM(\nu)) + (\calM(\nu) - m_f^h[\nu]).
\]
This decomposition, exploited in \cite{GKHV25} for the classical CBO system, allows one to treat the nonlinear weighted mean as a small deviation from $\calM$ that can be quantified in terms of the Wasserstein distance and the centered moments of $\mu$ and $\nu$.

In the sequel, we adapt this perturbative strategy to the regularized Gibbs weight $\psi_h(x)=\omega_f^\alpha(x)+h(\alpha)$ and scalar multiplicative noise. For fixed parameters, the estimates depend on the weight ratio and the Lipschitz constant of $\omega_f^\alpha$, rather than an upper bound on the original objective. The main propagation-of-chaos result in \cite{GKHV25} is formulated in $L^2$, while \cite[Remark 5.11]{GKHV25} also notes the extension of its concentration estimate to higher moments. Here we carry out the coupling and fluctuation estimates in $L^p$, $p\ge2$, with their dependence on the available moments stated explicitly.

For later reference, define
\[
K_\alpha:=\frac{L_{\omega_f^\alpha}\Lambda_\alpha}{h(\alpha)}.
\]
This is the weighted-mean stability constant appearing in the next lemma. Constants in the subsequent coupling and stability estimates may depend on $K_\alpha$ in addition to $\Lambda_\alpha$ and the other stated parameters.

\begin{lemma}\label{lem: bounding calB}
    Let $p\ge2 $ and assume Assumption \ref{assum: omega}. Then for all $\mu, \nu \in\calP_{p}(\bbR^d)$, we have
    \[
    |m_{ f}^h[\mu]-\mathcal{M}(\mu)-m_{ f}^h[\nu]+\mathcal{M}(\nu)|\leq \frac{L_{\omega_f^\alpha}\Lambdaa}{h(\alpha)}\left(V_{p}^\frac{1}{p}(\mu)+V_{p}^\frac{1}{p}(\nu)\right)\mathcal{W}_{p}(\mu, \nu).
    \]
\end{lemma}
\begin{proof}
We give the proof to record explicitly the dependence on the regularized weight. It follows the weighted-mean decomposition of \cite{GKHV25}; the relation to bounded surrogate objectives is explained in Remark \ref{rem: bounded surrogate}.
    
Set 
\[
Z_{\mu}:=\intr \psi_h(x)\mu(dx), \quad g(x):=\left(x-\calM(\nu)\right)\left(\psi_h(x)-Z_{\mu}\right),
\]
and let $\pi\in \Pi(\mu, \nu)$ be any coupling. Then, by a simple rearrangement, we get
    \begin{align*}
        m_f^{h}[\mu]-\calM(\mu)-m_f^h[\nu]+\calM(\nu)&=\intr \left(x-\calM(\nu)\right)\left(\psi_h(x)-Z_{\mu}\right)\left(\frac{\mu}{Z_{\mu}}-\frac{\nu}{Z_{\nu}}\right)\\
        &=\frac{1}{Z_{\mu}}\intrr (g(x)-g(y))\pi(dxdy) +\left(\frac{1}{Z_{\mu}}-\frac{1}{Z_{\nu}}\right)\intr g(x)\nu(dx) \\
        &=: I + II.
    \end{align*}

We first estimate $I$. Since $\psi_h(x) = \omega_f^\alpha(x) + h(\alpha)$ is $L_{\omega_f^\alpha}$-Lipschitz by Assumption \ref{assum: omega}, the difference $g(x)-g(y)$ satisfies
\[
        |g(x)-g(y)|\leq |x-y|\left|\psi_h(x)-Z_{\mu}\right|+L_{\omega_f^\alpha}\left|y-\calM(\nu)\right||x-y|.
\]
Applying H\"older's inequality and using that the marginals of $\pi$ are $\mu$ and $\nu$, we find
    \begin{align*}
        &\intrr|g(x)-g(y)|\pi(dxdy)\cr
        &\quad\leq \left(\intrr|x-y|^2\pi(dxdy)\right)^{\frac{1}{2}}\left(\intrr|\psi_h(x)-Z_{\mu}|^2\pi(dxdy)\right)^{\frac{1}{2}}\\
        &\quad +L_{\omega_f^\alpha}\left(\intrr|x-y|^2\pi(dxdy)\right)^{\frac{1}{2}}\left(\intrr|y-\mathcal{M}(\nu)|^2\pi(dxdy)\right)^{\frac{1}{2}}.
    \end{align*}
    Furthermore, since $\psi_h$ is $L_{\omega_f^\alpha}$-Lipschitz, the first integral satisfies 
\[
        \intrr|\psi_h(x)-Z_{\mu}|^2\pi(dxdy)=\intr|\psi_h(x)-Z_{\mu}|^2\mu(dx)\leq L_{\omega_f^\alpha}^2V_2(\mu).
\]
Hence, we have
\[
|I|\le \frac{L_{\omega_f^\alpha}}{Z_\mu} \big(\sqrt{V_2(\mu)}+\sqrt{V_2(\nu)}\big)\,\calW_2(\mu,\nu) \le \frac{L_{\omega_f^\alpha}}{h(\alpha)} \big(\sqrt{V_2(\mu)}+\sqrt{V_2(\nu)}\big)\,\calW_2(\mu,\nu),
\]
due to $Z_\mu\ge h(\alpha)$.
 
We now turn to $II$.  A straightforward computation yields
    \begin{align*}
        |II|&\leq \frac{1}{h(\alpha)^2}|Z_{\mu}-Z_{\nu}|\intr|x-\mathcal{M}(\nu)|e^{-\alpha f(x)}\nu(dx)\\
        &\leq \frac{e^{-\alpha \underline{f}}}{h(\alpha)^2}\intrr|e^{-\alpha f(x)}-e^{-\alpha f(y)}|\pi(dxdy)\intr|x-\mathcal{M}(\nu)|\nu(dx)\\
        &\leq \frac{L_{\omega_f^\alpha}e^{-\alpha\underline{f}}}{h(\alpha)^2}\intrr |x-y|\pi(dxdy)\intr|x-\mathcal{M}(\nu)|\nu(dx)\\
        &\leq \frac{L_{\omega_f^\alpha}e^{-\alpha\underline{f}}}{h(\alpha)^2}\mathcal{W}_1(\mu, \nu)V_1(\nu).
    \end{align*}
    
    Combining the two estimates gives
    \begin{align*}
        |m_{ f}^h[\mu]-\mathcal{M}(\mu)-m_{ f}^h[\nu]+\mathcal{M}(\nu)|&\leq \frac{L_{\omega_f^\alpha}}{h(\alpha)}\left(\sqrt{V_2(\mu)}+\sqrt{V_2(\nu)}\right)\mathcal{W}_2(\mu, \nu) +\frac{L_{\omega_f^\alpha}e^{-\alpha\underline{f}}}{h(\alpha)^2}\mathcal{W}_1(\mu, \nu)V_1(\nu).
    \end{align*}
    Finally, we use the relations $\calW_1\le \calW_2 \le \calW_p$, $V_1\le V_2^{1/2} \le V_p^{1/p}$ for $p\ge2$, to conclude the desired result.
\end{proof}

%
%
%
%
%
%
%
%
%
%

\subsection{Exponential decay and moment estimates for the synchronous coupling}
In this subsection, we derive quantitative estimates for the mean-field process $\bar X_t$, which solves the limiting SDE \eqref{eq: coupling}. We note that \eqref{eq: coupling} admits a unique strong solution under Assumption \ref{assum: omega}, together with suitable moment bounds on the initial data. This justifies all computations carried out in this section. For readability, however, we postpone the discussion of well-posedness until Section \ref{subsec: well posedness}. Since \eqref{eq: coupling} is well-posed, the associated law $\bar\rho_t = \mathrm{Law}(\bar X_t)$ defines a weak solution to \eqref{eq_pde}.

The following lemma establishes that $\bar X_t$ converges exponentially fast towards its mean 
and that its fluctuations around the weighted average $m_f^h[\bar\rho_t]$ also decay exponentially in time.

\begin{lemma}\label{lem: mean field limit exponential decay}
    Let $p\geq 2$ and assume $\bar{\rho}_0\in\calP_p(\bbR^d)$. Then for $t\geq 0$, we have  
\[
        \bbE|\xm-\bbE\xm|^p\leq C_p \swabM_p(\bar{\rho}_0)e^{-p(\lambda-\bar{\bflambda}_{p, \alpha, \sigma})t}
\]
and
\[
        \bbE|\bar{X}_t-m_f^h[\bar{\rho}_t]|^p\leq C_{p,\Lambda_\alpha} \swabM_p(\bar{\rho}_0)e^{-p(\lambda-\bbflambda_{p, \alpha, \sigma})t},
\]
    where 
    \[
    \bar{\bflambda}_{p, \alpha, \sigma}:=(p-1)(1+\Lambda_\alpha^{2/p})\sigma^2 = \bflambda_{p,\alpha,\sigma} + (p-1)\sigma^2.
    \]
\end{lemma}
\begin{proof}
We start by analyzing the evolution of the mean of $\bar X_t$. Taking the expectation in \eqref{eq: coupling} gives 
\[
\frac{d}{dt}\mathbb{E}\bar X_t=-\lambda(\mathbb{E}\bar X_t-m_f^h[\bar\rho_t]).
\]
Substituting this into the It\^o's formula for $|\bar{X}_t-\bbE\bar{X}_t|^p$ yields
    \begin{equation} \label{eq: bar x pvar}
    \begin{split}
        d|\bar{X}_t-\bbE\bar{X}_t|^p=&-p\lambda |\bar{X}_t-\bbE\bar{X}_t|^pdt +\frac{p\sigma^2}{2}|\bar{X}_t-\bbE\bar{X}_t|^{p-2}|\bar{X}_t-\mrho|^2dt\\
        &+\frac{p(p-2)\sigma^2}{2}|\bar{X}_t-\bbE\bar{X}_t|^{p-4}\langle \bar{X}_t-\bbE\bar{X}_t, \bar{X}_t-\mrho\rangle^2dt\\
        &+p\sigma |\bar{X}_t-\bbE\bar{X}_t|^{p-2}\langle \bar{X}_t-\bbE\bar{X}_t, \bar{X}_t-\mrho\rangle dW_t.
    \end{split}
    \end{equation}
 Similarly as in the proof of Lemma \ref{lem: diff Lp}, using the Cauchy--Schwarz and H\"older's inequalities, then taking expectations, we find
    \begin{equation} \label{eq: coupled cont lp}
    \begin{split}
        \frac{d}{dt}\bbE|\bar{X}_t-\bbE\bar{X}_t|^p\leq &-p\lambda\bbE|\bar{X}_t-\bbE\bar{X}_t|^p +\frac{p(p-1)\sigma^2}{2}\left[\left(\bbE|\bar{X}_t-\bbE\bar{X}_t|^p\right)^\frac{p-2}{p}\left(\bbE|\xm-\mrho|^p\right)^\frac{2}{p}\right].
    \end{split}
    \end{equation}
    
Next, we control the distance between the mean $\mathbb{E}\bar X_t$ and the weighted mean $m_f^h[\bar\rho_t]$.  
By direct calculation,
    \begin{equation}\label{eq: EbarX-mrho}
        \begin{split}
        |\bbE\xm-\mrho|^p &=\left|\frac{\intr(x-\calM(\bar{\rho}_t))\psi_h(x)\bar{\rho}_t(dx)}{\intr\psi_h(x)\bar{\rho}_t(dx)}\right|^p\\
        &\leq \lt| \frac{ \lt(\intr |x-\calM(\bar{\rho}_t)|^p \,\psi_h(x)\,\bar{\rho}_t(\dx) \rt)^{1/p} \lt(\intr \psi_h(x) \,\bar{\rho}_t(\dx) \rt)^{(p-1)/p} }{\intr \psi_h(x) \bar{\rho}_t(\dx)} \rt|^p \\
        &\leq \Lambdaa \intr|x-\calM(\bar{\rho}_t)|^p\bar{\rho}_t(dx)   \\
        &= \Lambdaa \bbE |\bar X_t - \bbE \bar X_t |^p,
        \end{split}
    \end{equation}
    where we used $h(\alpha)\le \psi_h\le e^{-\alpha\underline f}+h(\alpha)$.
    
This gives
\[
        \left(\bbE|\xm-\mrho|^p\right)^\frac{2}{p}  \leq \left[\left(\bbE|\xm-\bbE\xm|^p\right)^\frac{1}{p}+\left(|\bbE\xm-\mrho|^p\right)^\frac{1}{p}\right]^2 \leq 2 (1+\Lambdaa^\frac{2}{p})\left(\bbE|\xm-\bbE\xm|^p\right)^\frac{2}{p}.
\]
Plugging this estimate into \eqref{eq: coupled cont lp} yields
\[
        \frac{d}{dt}\bbE|\xm-\bbE\xm|^p\leq -\left(p\lambda- p(p-1)(1+\Lambdaa^\frac{2}{p})\sigma^2\right)\bbE|\xm-\bbE\xm|^p.
\]
    Then by Gr\"onwall's inequality, we have
\[
        \bbE|\xm-\bbE\xm|^p\leq \bbE|\bar{X}_0-\bbE\bar{X}_0|^pe^{-p(\lambda-(p-1)(1+\Lambdaa^{2/p})\sigma^2)t} \leq 2^p\swabM_p(\bar{\rho}_0)e^{-p(\lambda-(p-1)(1+\Lambdaa^{2/p})\sigma^2)t}.
\]
    
To estimate $\bbE|\bar X_t-m_f^h[\bar\rho_t]|^p$, we use the decomposition
\[
|\bar X_t-m_f^h[\bar\rho_t]|^p \le 2^{p-1}|\bar X_t-\bbE\bar X_t|^p +2^{p-1}|\bbE\bar X_t-m_f^h[\bar\rho_t]|^p,
\]
together with \eqref{eq: EbarX-mrho}.   Taking expectations yields
\[
\bbE|\bar X_t-m_f^h[\bar\rho_t]|^p \le 2^{p-1}(1+\Lambda_\alpha) \bbE|\bar X_t-\bbE\bar X_t|^p,
\]
and substituting the previous decay estimate concludes the proof.    
\end{proof}
 
 The exponential decay established above immediately implies that the mean-field trajectories remain uniformly concentrated around their average. We now show that this concentration yields a uniform-in-time $L^p$ bound for $\bar X_t$.

\begin{lemma}\label{lem: barX Lp bound}
    Let $p\ge 2$. If $\lambda>\bbflambda_{p, \alpha, \sigma}$ and $\bar{\rho}_0\in \calP_{p}(\bbR^d)$, then there exists a constant $\bar C_* = \bar{C}_*(\lambda, \sigma, p,\Lambdaa, \bar{\rho}_0) > 0$ such that 
\[
        \sup_{t\geq 0}\mathbb{E}|\bar{X}_t|^{p}\leq \bar{C}_*.
\]
\end{lemma}
\begin{proof}
The proof follows the same strategy as Lemma \ref{lem: Lp bound}, now exploiting the explicit exponential decay established in Lemma \ref{lem: mean field limit exponential decay}.  From \eqref{eq: coupling} and the triangle inequality, we first estimate
    \begin{align*}
\bbE|\bar X_t|^p
&\le3^{p-1}\bbE|\bar X_0|^p
+3^{p-1}\lambda^p\bbE\left(\int_0^t|\bar X_s-m_f^h[\bar\rho_s]|ds\right)^p\\
&\quad+3^{p-1}\sigma^pC_{\mathrm{BDG},p}
\bbE\left(\int_0^t|\bar X_s-m_f^h[\bar\rho_s]|^2ds\right)^{p/2}\\
&=:3^{p-1}\bbE|\bar X_0|^p+3^{p-1}\lambda^p I_t
+3^{p-1}\sigma^pC_{\mathrm{BDG},p}II_t,
    \end{align*}
    where the last term comes from the Burkholder--Davis--Gundy inequality.  

To control $I$, we apply Minkowski's integral inequality together with Lemma \ref{lem: mean field limit exponential decay}, obtaining
        \begin{align*}
           I_t &\leq \left(\int_0^t\left(\bbE|\bar{X}_s-m_f^h[\bar{\rho}_s]|^p\right)^\frac{1}{p}ds\right)^p\\
            &\leq 2^{2p-1}(1+\Lambdaa)\swabM_p(\bar{\rho}_0)\left(\int_0^te^{-(\lambda-\bbflambda_{p, \alpha, \sigma})s}ds\right)^p\\
            &\leq \frac{2^{2p-1}(1+\Lambdaa)\swabM_p(\bar{\rho}_0)}{(\lambda-\bbflambda_{p, \alpha, \sigma})^p}.
        \end{align*}
Similarly, for $II$ we have
        \begin{align*}
            II_t &\leq \left(\int_0^t\left(\bbE|\bar{X}_s-m_f^h[\bar{\rho}_s]|^p\right)^\frac{2}{p}ds\right)^\frac{p}{2}\\
            &\leq 2^{2p-1}(1+\Lambdaa)\swabM_p(\bar{\rho}_0)\left(\int_0^te^{-2(\lambda-\bbflambda_{p, \alpha, \sigma})s}ds\right)^\frac{p}{2}\\
            &\leq \frac{2^{2p-1}(1+\Lambdaa)\swabM_p(\bar{\rho}_0)}{\left(2(\lambda-\bbflambda_{p, \alpha, \sigma})\right)^\frac{p}{2}}.
        \end{align*}
        
Combining all bounds, we deduce that there exists a constant $\bar C_*>0$, depending only on the parameters
$(\lambda,\sigma,p,\Lambda_\alpha,\bar\rho_0)$, such that
\[
\sup_{t\ge0}\mathbb E|\bar X_t|^p\le \bar C_*.
\]
This completes the proof. 
\end{proof}
%
%
%
%
%
%
%
%
%
%
\subsection{Concentration estimates for the synchronous coupling}

Constants in concentration estimates also depend on the fixed, strictly admissible exponential rate (such as $\kappa$ or $\bar\kappa$). They are not claimed to remain bounded as that rate approaches the endpoint of its admissible interval.

We now quantify how tightly the empirical weighted mean concentrates around its population counterpart in the synchronous coupling. This provides the probabilistic ingredient required later to pass from the moment estimates in $L^p$ to the uniform-in-time propagation of chaos.

For clarity, we recall the relevant notation. Let $\{W_t^i\}_{i=1}^N$ be a family of independent one-dimensional Wiener processes and let $\bar X_t^i$, $i=1,\dots,N$, denote the corresponding i.i.d. solutions of the mean-field SDE \eqref{eq: coupling} with common initial law $\bar\rho_0$. We denote by $\bar\rho_t=\mathrm{Law}(\bar X_t^i)$ their common distribution, and by
\[
\mmu:=\frac{1}{N}\sum_{i=1}^N\delta_{\bar X_t^i}
\]
the associated empirical measure of the system. 

The next lemma quantifies the deviation between $m_f^h[\bar\mu_t^N]$ and its population value $m_f^h[\bar\rho_t]$,  showing that it is of order $N^{-1/2}$ in $L^p$, with an exponential decay in time as obtained from Lemma \ref{lem: mean field limit exponential decay}. 

\begin{lemma}\label{lem: mrho-mmu}
    Let $p\ge 2$ and assume $\bar{\rho}_0\in \calP_p(\bbR^d)$. Then there exists a constant $C=C(p, \Lambdaa, \bar{\rho}_0)>0$ such that for all $t \ge 0$,
    \begin{equation*}
        \bbE|m_f^h[\mmu]-m_f^h[\bar{\rho}_t]|^p\leq  \frac{C}{N^\frac{p}{2}}e^{-p(\lambda-\bbflambda_{p, \alpha, \sigma})t}.
    \end{equation*}
\end{lemma}
\begin{proof}
    By definition of $m_f^h[\cdot]$ and using that $\psi_h\ge h(\alpha)$, we find
    \begin{align*}
        \bbE|m_f^h[\mmu]-\mrho|^p &=\bbE\left|\frac{\sum_{i=1}^N(\bar{X}_t^i-m_f^h[\bar{\rho}_t])\psi_h(\bar{X}_t^i)}{\sum_{i=1}^N\psi_h (\bar{X}_t^i)}\right|^p \leq \frac{1}{h(\alpha)^p}\bbE\left|\frac{1}{N}\sum_{i=1}^N(\bar{X}_t^i-m_f^h[\bar{\rho}_t])\psi_h(\bar{X}_t^i)\right|^p.
    \end{align*}
The random variables 
\[
Y_i:=(\bar X_t^i-m_f^h[\bar\rho_t])\psi_h(\bar X_t^i)
\]
are independent and centered (since $\mathbb E Y_i=0$). Thus, by applying the Marcinkiewicz--Zygmund inequality (see, e.g., \cite[Theorem 2, \S 10.3]{ST97}), Jensen's inequality, and using Lemma \ref{lem: mean field limit exponential decay}, we deduce
    \begin{align*}
        \bbE|m_f^h[\mmu]-\mrho|^p&\leq \frac{C_{\textnormal{MZ},p}}{N^\frac{p}{2}h(\alpha)^p}\bbE\left(\frac{1}{N}\sum_{i=1}^N|\bar{X}_t^i-m_f^h[\bar{\rho}_t]|^2\psi_h(\bar{X}_t^i)^2\right)^\frac{p}{2}\\
        &\leq \frac{C_{\textnormal{MZ},p}\Lambdaa^p}{N^\frac{p}{2}}\left[\frac{1}{N}\sum_{i=1}^N\bbE|\bar{X}_t^i-m_f^h[\bar{\rho}_t]|^p\right]\\
        &\leq \frac{C_{\textnormal{MZ},p}\Lambdaa^p}{N^\frac{p}{2}}\bbE|\bar{X}_t^1-m_f^h[\bar{\rho}_t]|^p\\
        &\leq \frac{C_{p,\Lambda_\alpha}}{N^\frac{p}{2}}\swabM_p(\bar{\rho}_0)e^{-p(\lambda-\bbflambda_{p, \alpha, \sigma})t}.
    \end{align*}
This completes the proof.    
\end{proof}

We now extend the $L^p$ control to a uniform-in-time concentration inequality for the synchronous coupling. 
Under a quantitative contraction on $\lambda$, an exponentially weighted energy functional yields a drift-martingale estimate leading to a sharp-in-$N$ tail bound for $\sup_{t\ge0}e^{\bar\kappa t}V_p(\bar\mu_t^N)$, consistent with the exponential decay above.

\begin{lemma}[Concentration for the synchronous coupling] \label{lem: coupling concentration}
Let $p,q\geq 2$ and assume $\bar{\rho}_0 \in \calP_{pq}(\bbR^d)$. For the nonlinear trajectories, we suppose that $\bar{X}_0^i$ are i.i.d. with $\bar{X}_0^i \sim \bar{\rho}_0$ for all $i=1,\dots, N$. Denoting 
\[
\bar{c}_{\textnormal{con},p}:=2\bar{\bflambda}_{p, \alpha, \sigma}+4(p-1)\sigma^2=2(p-1)(3+\Lambda_\alpha^{2/p})\sigma^2, 
\]
further assume
\bq\label{eq: bar lambda prime}
        \lambda > \bbflambda'_{p, q, \alpha,\sigma} :=\bbflambda_{p q,\alpha, \sigma} + \bar{c}_{\textnormal{con},p}.
\eq
    Then, for all small enough $\bar\kappa>0$ and every $A>0$, the following concentration estimate holds:
    \begin{align}\label{eq: coupling concentration}
        \bbP\left[\sup_{t\geq 0}e^{\bar\kappa t}V_p(\mmu)\geq 2^p V_p(\bar{\rho}_0)+2^pA\right]\leq \frac{C}{A^{q}N^\frac{q}{2}},
    \end{align}
    where $C>0$ depends only on $\lambda, p, q,\sigma, \Lambdaa$ and $\bar{\rho}_0$. In particular, \eqref{eq: coupling concentration} holds for all fixed $0 < \bar\kappa < p (\lambda - (\bar c_{\textnormal{con},p} \vee \bbflambda_{pq,\alpha,\sigma}))$.
\end{lemma}

\begin{proof} 
\noindent \textbf{Step 1: Reduction to centered $p$-moments and a differential inequality.}
    We first observe
    \begin{equation}\label{eq: frakM}
    \begin{split}
        V_p(\mmu) &:= \frac1N \sum_{i=1}^N |\bar X_t^i - \calM(\bar\mu_t^N)|^p\cr
        &\leq \frac{2^{p-1}}{N} \sum_{i=1}^N \lt( |\bar X_t^i - \calM(\bar\rho_t)|^p + \lt|\calM(\bar\rho_t) - \frac1N \sum_{i=1}^N \bar X_t^i \rt|^p \rt) \\
        &\leq \frac{2^p}{N}\sum_{i=1}^N|\bar{X}_t^i-\calM(\bar{\rho}_t)|^p,
    \end{split}
    \end{equation}
    and thus the evolution of $V_p(\bar\mu_t^N)$ can be estimated by that of $\frac1N\sum_{i=1}^N |\bar X_t^i - \bbE \bar X_t^i|^p$.
    
    From the identity in \eqref{eq: bar x pvar} and Jensen's inequality, we have
    \begin{align*}
        d\left(\frac{1}{N}\sum_{i=1}^N|\bar{X}_t^i-\bbE\bar{X}_t^i|^p\right)&\leq -\frac{p\lambda}{N}\sum_{i=1}^N |\bar{X}_t^i-\bbE\bar{X}_t^i|^pdt +\frac{p(p-1)\sigma^2}{2N}\sum_{i=1}^N|\bar{X}_t^i-\bbE\bar{X}_t^i|^{p-2}|\bar{X}_t^i-m_f^h[\bar{\rho}_t]|^2dt\\
        &\quad +\frac{p\sigma}{N}\sum_{i=1}^N |\bar{X}_t^i-\bbE\bar{X}_t^i|^{p-2}\langle \bar{X}_t^i-\bbE\bar{X}_t^i, \bar{X}_t^i-\mrho\rangle dW_t^i.
    \end{align*}
Here we estimate the second term on the right-hand side as follows. By H\"older's inequality, we first get
\bq\label{est_sec}
        \frac{1}{N}\sum_{i=1}^N|\bar{X}_t^i-\bbE\bar{X}_t^i|^{p-2}|\bar{X}_t^i-m_f^h[\bar{\rho}_t]|^2 \le \left(\frac{1}{N}\sum_{i=1}^N|\bar{X}_t^i - \bbE \bar{X}_t^i|^p \right)^\frac{p-2}{p}\left( \frac{1}{N} \sum_{i=1}^N|\bar{X}_t^i - m_f^h[\bar{\rho}_t]|^p\right)^\frac{2}{p},
\eq

To bound the last factor without comparing empirical central moments
with population-centered moments, put
\[
A_t=N^{-1}\sum_i|\bar X_t^i-\calM(\bar\rho_t)|^p,\quad
J_t=m_f^h[\bar\mu_t^N]-\calM(\bar\rho_t),\quad
\xi_t=m_f^h[\bar\mu_t^N]-m_f^h[\bar\rho_t].
\]
Weighted Jensen gives $|J_t|^p\le\Lambda_\alpha A_t$. Thus
\begin{align*}
\left(N^{-1}\sum_i|\bar X_t^i-m_f^h[\bar\rho_t]|^p\right)^{2/p}
 \le3\left(A_t^{2/p}+|J_t|^2+|\xi_t|^2\right) \le4(1+\Lambda_\alpha^{2/p})A_t^{2/p}+4|\xi_t|^2 +4|\calM(\bar\rho_t)-\calM(\bar\mu_t^N)|^2.
\end{align*}
The final, nonnegative term is retained to match the decomposition below.

Moreover, by using Young's inequality, we obtain
    \begin{align*}
            4\left[\frac{1}{N}\sum_{i=1}^N|\bar{X}_t^i-\bbE\bar{X}_t^i|^p\right]^\frac{p-2}{p}|\calM(\bar{\rho}_t)-\calM(\mmu)|^2&\leq \frac{4(p-2)}{p}\left[\frac{1}{N}\sum_{i=1}^N|\bar{X}_t^i-\bbE\bar{X}_t^i|^p\right]+\frac{8}{p}|\calM(\bar{\rho}_t)-\calM(\mmu)|^p\\
            &\leq 4\left[\frac{1}{N}\sum_{i=1}^N|\bar{X}_t^i-\bbE\bar{X}_t^i|^p\right]+\frac{8}{p}|\calM(\bar{\rho}_t)-\calM(\mmu)|^p,
                \end{align*}
                and similarly,
                \[
        4\left[\frac{1}{N}\sum_{i=1}^N|\bar{X}_t^i-\bbE\bar{X}_t^i|^p\right]^\frac{p-2}{p}|m_f^h[\mmu]-m_f^h[\bar{\rho}_t]|^2 
            \leq 4\left[\frac{1}{N}\sum_{i=1}^N|\bar{X}_t^i-\bbE\bar{X}_t^i|^p\right]+\frac{8}{p}|m_f^h[\mmu]-m_f^h[\bar{\rho}_t]|^p.
\]
These combined with \eqref{est_sec} yield
    \begin{equation} \label{eq: bar pvar 2}
    \begin{split}
        d\left(\frac{1}{N}\sum_{i=1}^N|\bar{X}_t^i-\bbE\bar{X}_t^i|^p\right)\leq &-p(\lambda-\bar c_{\mathrm{con},p})\left[\frac{1}{N}\sum_{i=1}^N |\bar{X}_t^i-\bbE\bar{X}_t^i|^p\right]dt\\
        &\\
        &+4(p-1)\sigma^2\left[|\calM(\bar{\rho}_t)-\calM(\mmu)|^p+|m_f^h[\mmu]-\mrho|^p\right]dt\\
        &+\frac{p\sigma}{N}\sum_{i=1}^N |\bar{X}_t^i-\bbE\bar{X}_t^i|^{p-2}\langle \bar{X}_t^i-\bbE\bar{X}_t^i, \bar{X}_t^i-\mrho\rangle dW_t^i.
    \end{split}
    \end{equation}
 By the hypothesis \eqref{eq: bar lambda prime}, we have $p(\lambda-\bar c_{\mathrm{con},p})>p \boldsymbol\lambda_{pq,\alpha,\sigma}>0$. Thus we fix
\[
0<\bar\kappa<p(\lambda-\bar c_{\mathrm{con},p}),
\]
to be further restricted later, and proceed to the weighted energy estimate in Step 2.

\medskip
\noindent \textbf{Step 2: Exponential weight and decomposition into drift + martingale.}  Motivated from \cite{GKHV25}, we define
\begin{align*}
        \bar{\calE}_t&:=\left[\frac{1}{N}\sum_{i=1}^N|\bar{X}_t^i-\bbE\bar{X}_t^i|^p\right]e^{\bar\kappa t},\quad  \bar{Z}_t :=4(p-1)\sigma^2e^{\bar\kappa t}\left[|\calM(\bar{\rho}_t)-\calM(\mmu)|^p+|m_f^h[\mmu]-\mrho|^p\right], 
        \end{align*}
        and
        \[
        d\bar{M}_t:=\left[\frac{p\sigma}{N}\sum_{i=1}^N |\bar{X}_t^i-\bbE\bar{X}_t^i|^{p-2}\langle \bar{X}_t^i-\bbE\bar{X}_t^i, \bar{X}_t^i-\mrho\rangle dW_t^i\right]e^{\bar \kappa t}.
\]
Multiplying \eqref{eq: bar pvar 2} by $e^{\bar\kappa t}$ and using $\bar\kappa<p(\lambda-\bar c_{\mathrm{con},p})$, we derive 
    \begin{align*}
        d\bar{\calE}_t\leq \bar{Z}_tdt+d\bar{M}_t \implies \bar{\calE}_t\leq \bar{\calE}_0+\int_0^t\bar{Z}_sds+\int_0^td\bar{M}_s.
    \end{align*}
Using Markov's inequality, we can split and estimate the event of concentration as
    \begin{align}\label{eq: three-split}
        \begin{aligned}
             \bbP\left[\sup_{s\in[0, t]}\bar{\calE}_s\geq \bbE\bar{\calE}_0+A\right] 
            &  \leq \bbP\left[\sup_{s\in[0, t]}\left(\bar{\calE}_0+\int_0^s\bar{Z}_rdr+\int_0^sd\bar{M}_r\right)\geq \bbE\bar{\calE}_0+A\right]\\
        & \leq \bbP\left[\bar{\calE}_0-\bbE\bar{\calE}_0\geq \frac{A}{3}\right]+\bbP\left[\sup_{s\in[0, t]}\left|\int_0^s\bar{Z}_rdr\right|\geq \frac{A}{3}\right] +\bbP\left[\sup_{s\in[0, t]}\left|\int_0^sd\bar{M}_r\right|\geq \frac{A}{3}\right]\\
        & \leq \frac{3^q}{A^q}\bbE|\bar{\calE}_0-\bbE\bar{\calE}_0|^q+\frac{3^q}{A^q}\bbE\left[\sup_{s\in[0, t]}\left|\int_0^s\bar{Z}_rdr\right|^q\right] +\frac{3^q}{A^q}\bbE\left[\sup_{s\in[0, t]}\left|\int_0^sd\bar{M}_r\right|^q\right] \\
        & =: \frac{3^q}{A^q}(I + II + III).
        \end{aligned}
    \end{align}
    
\noindent \textbf{ Step 3: Bounds for $I$, $II$, and $III$.} First, by the Marcinkiewicz--Zygmund inequality and independence of $\{\bar X_0^i\}$, we estimate
    \begin{align*}
       I &=\frac{1}{N^{q}}\bbE\left|\sum_{i=1}^N\left(|\bar{X}_0^i-\calM(\bar{\rho}_0)|^{p}-\bbE|\bar{X}_0^1-\calM(\bar{\rho}_0)|^{p}\right)\right|^{q}\\
        &\leq \frac{C_{\textnormal{MZ}, q}}{N^{\frac{q}{2}}}\bbE\left[\frac{1}{N}\sum_{i=1}^N\left[|\bar{X}_0^i-\calM(\bar{\rho}_0)|^{p}-\bbE|\bar{X}_0^i-\calM(\bar{\rho}_0)|^{p}\right]^2\right]^\frac{q}{2}\\
        &\leq \frac{C_{\textnormal{MZ},q}}{N^\frac{q}{2}}\left[\frac{1}{N}\sum_{i=1}^N\bbE\left||\bar{X}_0^i-\calM(\bar{\rho}_0)|^{p}-\bbE|\bar{X}_0^i-\calM(\bar{\rho}_0)|^{p}\right|^q\right]\\
        &\leq \frac{2^{q}C_{\textnormal{MZ},q}}{N^\frac{q}{2}}\left[\frac{1}{N}\sum_{i=1}^N\bbE|\bar{X}_0^i-\calM(\bar{\rho}_0)|^{pq}\right]\\
        &\leq \frac{C}{N^\frac{q}{2}},
    \end{align*}
    where $C>0$ depends on $p,q$, and $\swabM_{pq}(\bar\rho_0)$, and the last line follows from Lemma \ref{lem: mean field limit exponential decay}.
    
    For the second term, we
    apply Minkowski's integral inequality and Lemmas \ref{lem: mean field limit exponential decay}, \ref{lem: mrho-mmu}, to find
    \begin{align*}
II &\leq \bbE\left[\left|\int_0^\infty \bar{Z}_rdr\right|^q\right]\\
        &\leq \left(\int_0^\infty \left(\bbE|\bar{Z}_r|^q\right)^\frac{1}{q}dr\right)^q\\
        &\leq 4^q(p-1)^q\sigma^{2q}\left(\int_0^\infty \left[\left(\bbE|\calM(\bar{\rho}_r)-\calM(\bar\mu^N_r)|^{pq}\right)^\frac{1}{q}+\left(\bbE|m_f^h[\bar\mu^N_r]-m_f^h[\bar\rho_r]|^{pq}\right)^\frac{1}{q}\right]e^{\bar \kappa r}dr\right)^q\\
        &\leq  \frac{4^q(p-1)^q\sigma^{2q}}{N^\frac{pq}{2}}\left(\int_0^\infty C e^{-\left(p(\lambda-\bbflambda_{pq, \alpha, \sigma})-\bar \kappa \right)r}dr\right)^q\\
        &\leq \frac{C}{N^\frac{pq}{2}\left(p(\lambda-\bbflambda_{pq, \alpha, \sigma})-\bar \kappa\right)^q} ,
    \end{align*}
    where $C>0$ depends on $p,q,\Lambda_\alpha,\swabM_{pq}(\bar\rho_0)$, and the last line is valid provided that $\bar\kappa < p(\lambda-\bbflambda_{pq, \alpha, \sigma})$. Since $\lambda - \bbflambda_{pq,\alpha,\sigma} > \bar c_{\textnormal{con},p}$, it is clear that we can assume $\bar\kappa$ is small enough so that it lies in this range.

We finally estimate $III$. Under the given condition of $\lambda$, and owing to Lemma \ref{lem: mean field limit exponential decay}, we note that the process $\{\bar{M}_s\}_{s\geq 0}$ is a $\bbP$-martingale. Then we have
    \begin{align*}
        \bbE\left[\sup_{s\leq t}\left|\bar{M}_s\right|^q\right]&\leq C_{\textnormal{BDG}, q}\bbE\left[\left\langle\bar{M} \right\rangle_t^\frac{q}{2}\right]\\
        &\leq \frac{C_{\textnormal{BDG}, q}p^q\sigma^q}{N^\frac{q}{2}}\bbE\left[\frac{1}{N}\sum_{i=1}^N\left(\int_0^te^{2\bar \kappa s} |\bar{X}_s^i-\bbE\bar{X}_s^i|^{2p-2}|\bar{X}_s^i-m_f^h[\bar{\rho}_s]|^2 ds\right)^\frac{q}{2}\right]\\
        &= \frac{C_{\textnormal{BDG}, q}p^q\sigma^q}{N^\frac{q}{2}}\bbE\left[\int_0^te^{2\bar \kappa s} |\bar{X}_s^1-\bbE\bar{X}_s^1|^{2p-2}|\bar{X}_s^1-m_f^h[\bar{\rho}_s]|^2 ds\right]^\frac{q}{2}\\
        &\leq \frac{C_{\textnormal{BDG}, q}p^q\sigma^q}{N^\frac{q}{2}}\left[\int_0^te^{2\bar \kappa s}\left(\bbE\left[|\bar{X}_s^1-\bbE\bar{X}_s^1|^{(p-1)q}|\bar{X}_s^1-m_f^h[\bar{\rho}_s]|^q\right]\right)^\frac{2}{q}ds\right]^\frac{q}{2}\\
        &\leq \frac{C_{\textnormal{BDG}, q}p^q\sigma^q}{N^\frac{q}{2}}\left[\int_0^te^{2\bar \kappa s}\left[\left(\bbE|\bar{X}_s^1-\bbE\bar{X}_s^1|^{pq}\right)^\frac{p-1}{p}\left(\bbE|\bar{X}_s^1-m_f^h[\bar{\rho}_s]|^{pq}\right)^\frac{1}{p}\right]^\frac{2}{q}ds\right]^\frac{q}{2}\\
        &\leq \frac{C(p,q,\sigma,\swabM_{pq}(\bar\rho_0),\Lambda_\alpha)}{N^\frac{q}{2}}\left[\int_0^te^{(2\bar \kappa-2p(\lambda-\bbflambda_{pq, \alpha, \sigma}))s}ds\right]^\frac{q}{2} \\
        &\leq \frac{C}{N^\frac{q}{2}\left(p(\lambda-\bbflambda_{pq, \alpha, \sigma})-\bar \kappa \right)^\frac{q}{2}}.
    \end{align*}
In the above, the first line is due to the BDG inequality; the second line is by definition of the quadratic variation, the Cauchy--Schwarz inequality, and the following estimate
\begin{align*}
    \lt(\frac{1}{N^2}\sum_{i=1}^N \int_0^t |\sigma_i(s)|^2\,ds \rt)^{\frac{q}{2}} &= \frac{1}{N^{\frac{q}{2}}} \lt(\frac{1}{N}\sum_{i=1}^N \int_0^t |\sigma_i(s)|^2\,ds \rt)^{\frac{q}{2}} \le \frac{1}{N^{\frac{q}{2}}} \frac{1}{N} \sum_{i=1}^N \lt(\int_0^t |\sigma_i(s)|^2 \,ds \rt)^{\frac{q}{2}},
\end{align*}
which holds by Jensen for general sufficiently integrable processes $\sigma_i$. The third line then follows as the $\bar{X}_t^i$ are i.i.d. Next, the fourth line is by Minkowski's integral inequality, and the fifth is due to H\"older. Finally, the second-to-last line is due to Lemma \ref{lem: mean field limit exponential decay}  and the last line follows from direct computation.

\medskip
\noindent \textbf{ Step 4: Conclusion and choice of $\bar\kappa$.}
The bounds for $I,II,III$ obtained in Step 3 are independent of $t$. Hence, by continuity from below of probability, we can pass to the limit $t\to\infty$ in \eqref{eq: three-split} and deduce
\begin{align*}
    \bbP\lt[\sup_{t\ge 0} \bar{\calE}_t \ge \bbE\bar{\calE}_0 + A \rt] 
    &\le \frac{C}{N^{\frac{q}{2}} A^q} \Bigg[ 1 + \frac{1}{N^{\frac{q}{2}(p-1)}\left(p(\lambda-\bbflambda_{pq, \alpha, \sigma})-\bar\kappa \right)^q} + \frac{1}{\left(p(\lambda-\bbflambda_{pq, \alpha, \sigma})-\bar \kappa \right)^\frac{q}{2}} \Bigg].
\end{align*}
Recalling \eqref{eq: frakM}, we have the relations
\[
2^p\bar{\calE}_t \ge e^{\bar\kappa t}V_p(\mmu), \quad \bbE\bar{\calE}_0 = V_p(\bar\rho_0).
\]
Consequently,
\[
        \bbP\left[\sup_{t\ge 0}e^{\bar\kappa t}V_p(\mmu)\geq 2^p \bbE V_p(\bar{\rho}_0) + 2^p A\right] \leq \bbP\left[\sup_{t\ge 0}\bar{\calE}_t\geq \bbE V_p(\bar{\rho}_0)+ A\right]  .
\]
Combining the two estimates yields the claimed concentration inequality
\[
\bbP\!\left[\sup_{t\ge0}e^{\bar\kappa t}V_p(\mmu)
\ge 2^pV_p(\bar\rho_0)+2^pA\right]
\le
\frac{C}{A^qN^{\frac q2}},
\]
where $C>0$ depends only on $\lambda,p,q,\sigma,\Lambdaa$, and $\bar\rho_0$.  
Finally, $\bar\kappa$ can be chosen in the admissible range
\[
0 < \bar\kappa < p (\lambda - (\bar c_{\textnormal{con},p} \vee \bbflambda_{pq,\alpha,\sigma})),
\]
which ensures all exponential integrals above are finite.  
\end{proof}

An analogous argument applies to the particle trajectories $\{X_t^i\}_{i=1}^N$, yielding a concentration estimate of the same form as in Lemma \ref{lem: coupling concentration}. 
The proof closely follows the reasoning of \cite[Remark 5.10]{GKHV25}, with modifications analogous to those used in Lemma \ref{lem: coupling concentration}. 
For completeness, we include a detailed proof in Appendix \ref{app: stability}.

\begin{lemma}[Concentration for the particle system] \label{lem: again concentration estimate}Let $p,q\geq 2$ and assume $X_{\rm in} \in L^{pq}(\Omega)$. For the particle trajectories, we suppose that $X_0^i$ are i.i.d. with $X_0^i \sim X_{\rm in}$ for all $i=1,\dots, N$. Denoting  
\[
c_{\textnormal{con},p}:= 2\Lambdaa^{1-\frac{2}{p}}\bflambda_{p, \alpha, \sigma}+2(p-1)\sigma^2=2(p-1)(1+\Lambda_\alpha)\sigma^2, 
\]
further assume
\begin{equation}\label{eq: c con p 1}
    \lambda >\bflambda_{p,q, \alpha, \sigma}' :=\bflambda_{pq, \alpha, \sigma} + c_{\textnormal{con},p}.
\end{equation}
    Then for all small enough $\kappa>0$ and every $A>0$, the following concentration estimate holds:
\bq \label{eq: concentration} 
\bbP\left[\sup_{t\geq 0}e^{\kappa t}V_{p}(\mu_t^N)\geq \bbE\left[V_{p}(\mu_0^N)\right]+A\right]\leq \frac{C}{A^{q}N^\frac{q}{2}}, 
\eq
for a constant $C>0$ independent of $A$ and $N$. In particular, \eqref{eq: concentration} holds for all fixed $0 < \kappa < p(\lambda -( c_{\textnormal{con},p} \vee \bflambda_{pq,\alpha,\sigma}))$.
\end{lemma}

%
%
%
%
%
%
%
%
%
%
\subsection{Proof of Theorem \ref{thm: propagation of chaos}}\label{subsec: proof of prop of chaos}
We now establish the uniform-in-time propagation of chaos.   Throughout this subsection, we assume that
\bq\label{eq: lambda condition}
\begin{split}
    \lambda > \widetilde \bflambda_{p,q,\alpha,\sigma}
    &:= \max\{\bar\bflambda'_{p,q,\alpha,\sigma}, \bflambda'_{p,q,\alpha,\sigma}\}\\
    &= \max\lt\{ (pq-1)(1+\Lambda_\alpha^{\frac{2}{pq}})\sigma^2 + 2(p-1)(3+\Lambda_\alpha^{\frac{2}{p}})\sigma^2,\,
    (pq-1)\Lambda_\alpha^{\frac{2}{pq}}\sigma^2 + 2(p-1)(1+\Lambda_\alpha)\sigma^2 \rt\},
\end{split}
\eq      
where we recall that $\bar\bflambda'_{p,q,\alpha,\sigma}$ is given in \eqref{eq: bar lambda prime}, and $\bflambda'_{p,q,\alpha,\sigma}$   in \eqref{eq: c con p 1}.  This condition ensures that both Lemmas \ref{lem: coupling concentration} and \ref{lem: again concentration estimate} are applicable.  Note that we can choose $\widetilde\kappa>0$ satisfying
\begin{equation}\label{eq: tilde kappa}
    \widetilde\kappa < p \lt[ \Big( \lambda - \left(\bar c_{\textnormal{con},p} \vee \bbflambda_{pq,\alpha,\sigma}\right) \Big) \wedge \Big(\lambda - \left(c_{\textnormal{con},p} \vee \bflambda_{pq,\alpha,\sigma}\right) \Big) \rt],
\end{equation}
so that both concentration estimates \eqref{eq: coupling concentration} and \eqref{eq: concentration} hold simultaneously with $\widetilde\kappa$ in place of $\bar\kappa$ and $\kappa$, respectively.

Moreover, the inequalities \eqref{eq: lambda condition} and \eqref{eq: bar lambda prime} together imply
\begin{align*}
    \lambda \ge  \bbflambda'_{p,q,\alpha,\sigma} > \bar c_{\textnormal{con},p} > 2\bbflambda_{p,\alpha,\sigma} > 2\bflambda_{p,\alpha,\sigma}.
\end{align*}
This guarantees that all preceding results, such as Lemmas \ref{lem: diff Lp} and \ref{lem: mean field limit exponential decay}, can be employed in the forthcoming analysis.

For the proof of Theorem \ref{thm: propagation of chaos}, we use the canonical synchronous construction: set $\bar X_0^i=X_0^i$ almost surely and drive $X_t^i$ and $\bar X_t^i$ with the same Wiener process $W_t^i$. Since the initial variables and the Wiener processes are independent across $i$, the nonlinear copies $\{\bar X_t^i\}_{i=1}^N$ remain i.i.d. with common law $\bar\rho_t$. In particular, $\bar{\calG}_{0,p}=0$.

Following the synchronous-coupling strategy used, for instance, in \cite{GKHV25}, we introduce the fluctuation variable
\begin{align*}
    Z_t^i := X_t^i - \bar X_t^i.
\end{align*}
For convenience, we also set
\begin{equation}\label{eq: notations}\begin{split}
    &\dm =m_f^h[\mu_t^N]-m_f^h[\mmu], \quad\dcal=\calM(\mu_t^N)-\calM(\mmu), \\
    &\bar{\calG}_{t, p}:=\frac{1}{N}\sum_{i=1}^N|Z_t^i|^{p}, \quad \bar{\calD}_{t, p}:=\frac{1}{N}\sum_{i=1}^N|Z_t^i-\dcal|^{p}, \quad \bar{\calO}_{t, p}:=|\dcal|^{p}.
\end{split} \end{equation}
Then each $Z_t^i$ satisfies
    \begin{align*}
        d Z_t^i =&-\lambda \left(Z_t^i -\left(m_f^h[\mu_t^N]-m_f^h[\bar{\rho}_t]\right)\right)dt +\sigma\left(Z_t^i - \left(m_f^h[\mu_t^N]-m_f^h[\bar{\rho}_t]\right)\right)dW_t^i.
    \end{align*}

%
%
%
%
%
%
%
%
%
%
 \subsubsection{Estimates for the fluctuation functionals}

   As discussed in Section \ref{subsec: apriori obs}, we regard the right-hand side as a perturbation of the mean-field dynamics.  Hence, it is crucial to control the discrepancy $\Delta_t^\calM - \Delta_t^m$,  which quantifies the deviation between the empirical means of the interacting particle system and the mean-field process.   The following lemma, relying on the concentration estimates from Lemmas \ref{lem: coupling concentration} and \ref{lem: again concentration estimate}, provides an $L^p$-estimate for this quantity.

\begin{lemma}[Control of mean discrepancies]\label{lem: deltas}
Let $p\ge 2$ and $q>2$. Let the law $\bar\rho_0$ (corresponding to $\bar X_0$) have finite $pq$-moments. Then whenever $\lambda > \widetilde\bflambda_{p,q,\alpha,\sigma}$, and $\tilde\kappa$ satisfies \eqref{eq: tilde kappa}, there exists $0 < \theta <  \widetilde\kappa$ for which
    \begin{equation*}
    \begin{split}
        \bbE|\dcal-\dm|^{p}&\leq C\left[\frac{1}{N^{(q-2)/2}}+ \bbE\bar{\calG}_{t, p}\right]e^{-\theta t},
    \end{split}
\end{equation*}
    where $C>0$ depends only on $\lambda,\sigma,p,q$, $\Lambda_\alpha$, $K_\alpha$, $\bar\rho_0$ and the chosen admissible decay rate $\widetilde\kappa$.
\end{lemma}

\begin{proof}
    We recall from Lemma \ref{lem: bounding calB} that for $p\ge2$:
\begin{align}\label{eq: recall deltas}
    |\dcal-\dm|^{p}\leq \frac{1}{2}\left(\frac{2L_{\omega_f^\alpha}\Lambdaa}{h(\alpha)}\right)^{p}\left(V_{p}(\mu_t^N)+V_{p}(\mmu)\right)\calW_{p}^{p}(\mu_t^N, \mmu).
\end{align}
Thus, it suffices to estimate the expectation of 
\begin{equation}\label{eq: v w}
\left(V_{p}(\mu_t^N)+V_{p}(\mmu)\right)\calW_{p}^{p}(\mu_t^N, \mmu).
\end{equation}

We will compute the expectation of \eqref{eq: v w} by writing $\Omega = \Omega_{\tilde\kappa}\cup (\Omega\setminus \Omega_{\tilde\kappa})$, where
\begin{align*}
\Omega_{\widetilde\kappa}  =\left\{\omega:\sup_{t\ge0}e^{\widetilde\kappa t}V_p(\mu_t^N) \ge 2^pV_p(\bar\rho_0)+1\right\} \cup\left\{\omega:\sup_{t\ge0}e^{\widetilde\kappa t}V_p(\bar\mu_t^N) \ge 2^pV_p(\bar\rho_0)+2^p\right\}.
\end{align*}

Here $\bbE  V_p(\mu_0^N)\le2^pV_p(\bar\rho_0)$ by Jensen's inequality and the decomposition about the deterministic mean. For $p>2$ the sharper comparison with coefficient one is not used.

By definition of $\Omega_{\widetilde\kappa}$, we have
\begin{align*}
    \forall t\geq 0, \quad \left(V_{p}(\mu_t^N)+V_{p}(\mmu)\right) \mathds{1}_{\Omega \backslash \Omega_{\widetilde\kappa}} \leq 2^{p+1}(V_p(\bar{\rho}_0)+1)e^{-\widetilde\kappa t}.
\end{align*}
On the other hand, by subadditivity of $\bbP$ and Lemmas \ref{lem: coupling concentration}, \ref{lem: again concentration estimate}, we find
\begin{align*}
    \mathbb{P}\left[\Omega_{\widetilde\kappa}\right]\leq \frac{C}{N^\frac{q}{2}}.
\end{align*}
Thus, the expectation of \eqref{eq: v w} is computed as
\begin{equation}\label{eq: vvw}
\begin{split}
    &\mathbb{E}\big[\big(V_{p}(\mu_t^N)+V_{p}(\mmu)\big)\mathcal{W}_{p}^{p}(\mu_t^N, \mmu)\big]\\
    &\quad= \mathbb{E}\left[\left(V_{p}(\mu_t^N)+V_{p}(\mmu)\right)\mathcal{W}_{p}^{p}(\mu_t^N,\mmu) \,\mathds{1}_{\Omega_{\widetilde\kappa}} \right]+\mathbb{E}\left[\left(V_{p}(\mu_t^N)+V_{p}(\mmu)\right)\mathcal{W}_{p}^{p}(\mu_t^N,\mmu)\,\mathds{1}_{\Omega\backslash\Omega_{\widetilde\kappa}}\right]\\
    &\quad\leq \mathbb{P}\big[\Omega_{\widetilde\kappa}\big]^\frac{q-2}{q} \left(\mathbb{E}\left[\left(V_{p}(\mu_t^N)+V_{p}(\mmu)\right)^{q/2}\mathcal{W}_{p}^{pq/2}(\mu_t^N,\mmu)\right] \right)^{\frac{2}{q}}+e^{-\widetilde\kappa t}  C_1 \bbE\bar{\calG}_{t, p},
\end{split}
\end{equation}
where $C_1:= 2^{p+1}(V_p(\bar{\rho}_0)+1)$. It remains to compute the term in parentheses. Applying the Cauchy--Schwarz and Jensen's inequalities, and using \eqref{eq: frakM}, we obtain
\begin{align*}
    \mathbb{E}\left[\left(V_{p}(\mu_t^N)+V_{p}(\mmu)\right)^{q/2}\mathcal{W}_{p}^{pq/2}(\mu_t^N,\mmu)\right] 
    & \leq \left(\mathbb{E}\left[\left(V_{p}(\mu_t^N)+V_{p}(\mmu)\right)^{q}\right]\right)^\frac{1}{2}\left(\mathbb{E}\mathcal{W}_{p}^{pq}(\mu_t^N,\mmu)\right)^\frac{1}{2}\\
    & \leq \left(2^{q-1}\mathbb{E}V_{pq}(\mu_t^{N})+2^{q-1}\mathbb{E}V_{pq}(\mmu)\right)^\frac{1}{2}\left(\bbE\bar{\calG}_{t, pq}\right)^\frac{1}{2}\\
    & \leq \left(2^{q-1}\mathbb{E}V_{pq}(\mu_t^{N})+\frac{2^{(p+1)q-1}}{N}\sum_{i=1}^N\bbE|\bar{X}_t^i-\calM(\bar{\rho}_t)|^{pq}\right)^\frac{1}{2}\left(\bbE\bar{\calG}_{t, pq}\right)^\frac{1}{2}.
\end{align*}
The right-hand side can be estimated further as follows. We use Lemmas \ref{lem: diff Lp} and \ref{lem: mean field limit exponential decay} to have
\begin{align*}
    2^{q-1}\bbE V_{pq}(\mu_t^N) =\frac{2^{q-1}}{N}\sum_{i=1}^N\bbE|X_t^i-\calM(\mu_t^N)|^{pq} \leq 2^{q-1}\bbE|X_0^1-X_0^2|^{pq}e^{-pq(\lambda-\bflambda_{pq, \alpha, \sigma})t} \leq C \swabM_{pq}(\bar{\rho}_0)e^{-pq(\lambda-\bflambda_{pq, \alpha, \sigma})t}  
\end{align*}
and
\[    
    \frac{2^{(p+1)q-1}}{N}\sum_{i=1}^N\bbE|\bar{X}_t^i-\calM(\bar{\rho}_t)|^{pq} \le C \swabM_{pq}(\bar{\rho}_0)e^{-pq(\lambda -\bbflambda_{pq, \alpha, \sigma})t}.  
\]
Using Lemmas \ref{lem: Lp bound} and \ref{lem: barX Lp bound}, we also find
\[
  \bbE\calG_{t, pq} =\frac{1}{N}\sum_{i=1}^N\bbE|X_t^i-\bar{X}_t^i|^{pq} \le \frac{2^{pq-1}}{N} \sum_{i=1}^N (\bbE|X_t^i|^{pq} + \bbE|\bar X_t^i|^{pq} ) \leq C,
  \]
since we assume $\bar\rho_0$ has finite $pq$-th moments. Thus, combining the above estimates yields
\[
\mathbb{E}\left[\left(V_{p}(\mu_t^N)+V_{p}(\mmu)\right)^{q/2}\mathcal{W}_{p}^{pq/2}(\mu_t^N,\mmu)\right] \leq C e^{-\frac{pq}{2}(\lambda-\bbflambda_{pq, \alpha, \sigma})t},
\]
where $C>0$ depends on $p,q,\sigma,\Lambda_\alpha,\bar\rho_0$ but is independent of $N$ and $t$. Subsequently, inserting this into \eqref{eq: vvw} gives
\[
\bbE[(V_p(\mu_t^N) + V_p(\bar\mu_t^N)) \calW_p^p(\mu_t^N,\bar\mu_t^N)] \leq \frac{C}{N^{\frac{q-2}{2}}}e^{-p(\lambda-\bbflambda_{pq, \alpha, \sigma})t}+C_1\bbE\bar{\calG}_{t, p}e^{-\widetilde\kappa t}. 
\]
Finally, taking expectations in \eqref{eq: recall deltas}, then inserting the above bound, we conclude the desired result.
\end{proof}

At this point we remind the reader of the notations in \eqref{eq: notations}. Our goal is to control the evolution of $\bar\calG_{t,p}$ uniformly in time and with explicit $N$-dependence. Using the elementary inequality
\begin{equation}\label{eq: relation}
            \bar\calG_{t, p} \leq 2^{p-1}\left(\bar\calD_{t, p}+\bar\calO_{t, p}\right) \leq \left(2^{2p-1}+2^{p-1}\right)\bar\calG_{t, p},
    \end{equation}
it suffices to estimate $\bar\calD_{t,p}$ and $\bar\calO_{t,p}$. We begin with $\bar\calD_{t,p}$.
\begin{lemma} \label{lem: barcalD}
Assume that the assumptions of Theorem \ref{thm: propagation of chaos} hold. Then for any
\[
\widetilde\theta\in \lt(0,\frac{\theta}{p}\wedge (\lambda - \bbflambda_{p,\alpha,\sigma})\rt)
\] 
(where $\theta$ is the constant from Lemma \ref{lem: deltas}), there exists a positive constant $C>0$, depending only on $\lambda,\sigma,p,q$, $\Lambda_\alpha$, $K_\alpha$, $\bar\rho_0$ and the fixed admissible decay rates, such that
    \begin{equation}\label{eq: barcalD}
    \frac{d}{dt}\bbE\bar{\calD}_{t, p}\leq C \lt( \bbE\bar{\calD}_{t,p}e^{-\widetilde\theta t} + \bbE\bar{\calO}_{t, p}e^{-\widetilde\theta t}+\frac{1}{N^{\frac{p}{2}\wedge \frac{q-2}{2}}}e^{-\widetilde\theta t} \rt).
\end{equation}
    \end{lemma}
\begin{proof}

We take $N\ge2$ in the displayed quadratic-variation decomposition, so $1-2/N\ge0$. The final theorem for $N=1$ follows directly from the uniform moment bounds, after enlarging the constant.

    Applying the It\^o formula to $\bar{\calD}_{t, p}=\frac{1}{N}\sum_{i=1}^N|Z_t^i-\dcal|^{p}$ and using the SDE for $Z_t^i$ yields
    \begin{align}
            d\bar{\calD}_{t, p}=&-p\lambda \bar{\calD}_{t, p}dt \nonumber\\
            &+\frac{p\sigma^2}{2N}\left(1-\frac{2}{N}\right)\sum_{i=1}^N|Z_t^i-\dcal|^{p-2}|Z_t^i-(m_f^h[\mu_t^N]-m_f^h[\bar{\rho}_t])|^2dt \label{eq: bar d 2}\\
            &+\frac{p\sigma^2}{2N^3}\sum_{i=1}^N\sum_{j=1}^N|Z_t^i-\dcal|^{p-2}|Z_t^j-(m_f^h[\mu_t^N]-m_f^h[\bar{\rho}_t])|^2dt \label{eq: bar d 3}\\
            &+\frac{p(p-2)\sigma^2}{2N}\left(1-\frac{2}{N}\right)\sum_{i=1}^N|Z_t^i-\dcal|^{p-4}\langle Z_t^i-\dcal,Z_t^i-(m_f^h[\mu_t^N]-m_f^h[\bar{\rho}_t])\rangle ^2dt \label{eq: bar d 4}\\
            &+\frac{p(p-2)\sigma^2}{2N^3}\sum_{i=1}^N\sum_{j=1}^N|Z_t^i-\dcal|^{p-4}\langle Z_t^i       -\dcal,Z_t^j-(m_f^h[\mu_t^N]-m_f^h[\bar{\rho}_t])\rangle ^2dt \label{eq: bar d 5}\\
            &+\frac{p\sigma}{N}\sum_{i=1}^N|Z_t^i-\dcal|^{p-2}\langle Z_t^i-\dcal, Z_t^i-(m_f^h[\mu_t^N]-m_f^h[\bar{\rho}_t])\rangle dW_t^i \nonumber\\
            &-\frac{p\sigma}{N^2}\sum_{i=1}^N\sum_{j=1}^N|Z_t^i-\dcal|^{p-2}\langle Z_t^i-\dcal, Z_t^j-(m_f^h[\mu_t^N]-m_f^h[\bar{\rho}_t])\rangle dW_t^j. \nonumber
        \end{align}
Taking expectations, the stochastic integrals vanish. Regarding the expectation of \eqref{eq: bar d 2}, we use simple algebraic manipulations such as
    \begin{align*}
       \sum_{i=1}^N \bbE |Z_t^i - \Delta_t^\calM|^{p-2} |Z_t^i - (m_f^h[\mu_t^N] - m_f^h[\bar\rho_t])|^2  
        &  = \sum_{i=1}^N \bbE |Z_t^i - \Delta_t^\calM|^{p-2} |Z_t^i - \dcal +\dcal  -\dm -m_f^h[\bar{\mu}_t^N]+ m_f^h[\bar\rho_t]|^2 \\
        &  \le  3N\bbE\bar\calD_{t,p}+3\sum_{i=1}^N\bbE|Z_t^i-\dcal|^{p-2}|\Delta_t^\calM  - \dm|^2\\
        &\quad +3\sum_{i=1}^N\bbE|Z_t^i-\dcal|^{p-2}|m_f^h[\bar\mu_t^N]-m_f^h[\bar\rho_t]|^2.
    \end{align*}       
    An analogous estimate holds for \eqref{eq: bar d 4}. Grouping \eqref{eq: bar d 3} and \eqref{eq: bar d 5} together, we then have at this point
    \begin{equation*}
        \begin{split}
            \frac{d}{dt}\bbE\bar{\calD}_{t, p}&\leq-p\left({\lambda}-\frac{3(p-1)\sigma^2}{2}\right) \bbE\bar{\calD}_{t, p} +\frac{3p(p-1)\sigma^2}{2N}\sum_{i=1}^N\bbE\left[|Z_t^i-\dcal|^{p-2}|\dcal-\dm|^2\right]\\
        &\quad \;+\frac{3p(p-1)\sigma^2}{2N}\sum_{i=1}^N\bbE\left[|Z_t^i-\dcal|^{p-2}|m_f^h[\bar{\mu}_t^N]-m_f^h[\bar{\rho}_t]|^2\right]\\
        &\quad\;+\frac{p(p-1)\sigma^2}{2N^3}\sum_{i=1}^N\sum_{j=1}^N\bbE\left[|Z_t^i-\dcal|^{p-2}|Z_t^j-(m_f^h[\mu_t^N]-\mrho)|^2\right]\\
        &=:-p\left({\lambda}-\frac{3(p-1)\sigma^2}{2}\right) \bbE\bar{\calD}_{t, p}+ I + II + III.
        \end{split}
    \end{equation*}
    By our assumption on $\lambda$, the first term in the inequality above is negative, thus it suffices to estimate the remaining terms, $I, II$, and $III$.  
    
    For the estimate of $I$, by using H\"older's inequality, Lemma \ref{lem: deltas}, and \eqref{eq: relation}, we obtain
    \begin{align*}
   I&\leq \frac{3p(p-1)\sigma^2}{2}\left(\bbE\bar{\calD}_{t, p}\right)^\frac{p-2}{p}\left(\bbE|\dcal-\dm|^p\right)^\frac{2}{p}\\
        &\leq C \left(\bbE\bar{\calD}_{t, p}\right)^\frac{p-2}{p}\left[\frac{1}{N^{\frac{q-2}{2}}} + \bbE\bar{\calG}_{t, p}\right]^\frac{2}{p}e^{-\frac{2}{p}\theta t}   \\
        &\leq C \left(\bbE\bar{\calD}_{t, p}\right)^\frac{p-2}{p}\left[\frac{1}{N^{\frac{q-2}{2}}}+2^{p-1}(\bbE\bar{\calD}_{t, p}+\bbE\bar{\calO}_{t, p})\right]^\frac{2}{p}e^{-\frac{2}{p}\theta t}  \\
        &\leq C \left[\frac{1}{N^\frac{q-2}{p}}\left(\bbE\bar{\calD}_{t, p}\right)^\frac{p-2}{p}+4\bbE\bar{\calD}_{t, p}+4\left(\bbE\bar{\calD}_{t, p}\right)^\frac{p-2}{p}\left(\bbE\bar{\calO}_{t, p}\right)^\frac{2}{p}\right]e^{-\frac{2}{p}\theta t}\\
        &\leq C \left[\frac{1}{N^{\frac{q-2}{2}}} + \bbE\bar{\calD}_{t, p}+\bbE\bar{\calO}_{t, p} \right]e^{-\frac{2}{p}\theta t} ,
    \end{align*}
    where the last line follows from Young's inequality.

    In a similar manner, by using Lemma \ref{lem: mrho-mmu} and Young's inequality, we estimate
\begin{align*}
II&\le\frac{3p(p-1)\sigma^2}{2}
 (\bbE \bar\calD_{t,p})^{(p-2)/p}
 \left(\bbE |m_f^h[\bar\mu_t^N]-m_f^h[\bar\rho_t]|^p\right)^{2/p}\\
&\le\frac C N(\bbE \bar\calD_{t,p})^{(p-2)/p}
 e^{-2(\lambda-\bbflambda_{p,\alpha,\sigma})t}\\
&\le C\left(N^{-p/2}+\bbE \bar\calD_{t,p}\right)
 e^{-2(\lambda-\bbflambda_{p,\alpha,\sigma})t}.
\end{align*}
The average over particle labels has already been included in
$\bar\calD_{t,p}$; it does not provide an additional factor $N^{-1}$.

For $III$, we recall (see the beginning of Section \ref{subsec: proof of prop of chaos}) that both Lemmas \ref{lem: diff Lp} and \ref{lem: mean field limit exponential decay} are available due to $\lambda$ being large enough. Hence, we have
    \begin{align*}
III &\leq \frac{p(p-1)\sigma^2}{2N^3}\sum_{i=1}^N\sum_{j=1}^N\left(\bbE|Z_t^i-\dcal|^p\right)^\frac{p-2}{p}\left(\bbE|Z_t^j-(m_f^h[\mu_t^N]-\mrho)|^p\right)^\frac{2}{p}\\
        &\leq \frac{p(p-1)\sigma^2}{N}\left(\bbE\bar{\calD}_{t, p}\right)^\frac{p-2}{p}\left(\frac{1}{N}\sum_{i=1}^N\left(\bbE|X_t^i-m_f^h[\mu_t^N]|^p\right)^\frac{2}{p}+\left(\bbE|\bar{X}_t^i-\mrho|^p\right)^\frac{2}{p}\right)\\
        &\leq \frac{C}{N}\left(\bbE\bar{\calD}_{t, p}\right)^\frac{p-2}{p} e^{-2(\lambda - \bbflambda_{p,\alpha,\sigma})t} \\ 
        &\leq C \left[\frac{1}{N^\frac{p}{2}} + \bbE\bar{\calD}_{t, p}\right]e^{-2(\lambda-\bbflambda_{p, \alpha, \sigma})t}.
    \end{align*}
    
Collecting the bounds for $I,II,III$ and choosing $\widetilde\theta>0$ sufficiently small ($\widetilde\theta<\frac{\theta}{p}$ and $\widetilde\theta<\lambda-\bbflambda_{p,\alpha,\sigma}$), we absorb the exponentially decaying factors into $e^{-\widetilde\theta t}$ and obtain \eqref{eq: barcalD} with a suitable constant $C>0$ that is independent of $t$ and $N$.
\end{proof}

\begin{lemma}  \label{lem: baecalO}
Assume that the assumptions of Theorem \ref{thm: propagation of chaos} hold. Then there exists a constant $C>0$, depending only on $\lambda,\sigma,p,q$, $\Lambda_\alpha$, $K_\alpha$, $\bar\rho_0$ and the fixed admissible decay rates, such that for the same $\widetilde\theta>0$ as in Lemma \ref{lem: barcalD}:
\begin{equation}\label{eq: barcalO}
    \frac{d}{dt}\bbE\bar{\calO}_{t, p}\leq C\lt( \bbE\bar{\calO}_{t,p}e^{-\widetilde\theta t} + \bbE\bar{\calD}_{t, p}e^{-\widetilde\theta t}+\frac{1}{N^{\frac{p}{2}\wedge\frac{q-2}{2}}}e^{-\widetilde\theta t} \rt).
\end{equation}
\end{lemma}

\begin{proof}
Applying the It\^o formula to $\bar{\calO}_{t, p}$ gives
\begin{align*}
    d\bar{\calO}_{t, p}= & -p\lambda|\Delta_t^\mathcal{M}|^{p-2}\langle \Delta_t^\mathcal{M}, \Delta_t^\mathcal{M} -(m_f^h[\mu_t^N]-m_f^h[\bar{\rho}_t])\rangle dt\\
    &+\frac{p\sigma^2}{2N^2}|\Delta_t^\mathcal{M}|^{p-2}\sum_{i=1}^N |Z_t^i -(m_f^h[\mu_t^N]-m_f^h[\bar{\rho}_t])|^2dt\\
    &+ \frac{p(p-2)\sigma^2}{2N^2}|\Delta_t^\mathcal{M}|^{p-4}\sum_{i=1}^N\langle \dcal , Z_t^i-(m_f^h[\mu_t^N]-m_f^h[\bar{\rho}_t])\rangle ^2 dt\\
    &+\frac{p\sigma}{N}|\dcal|^{p-2}\sum_{i=1}^N\langle \dcal , Z_t^i-(m_f^h[\mu_t^N]-m_f^h[\bar{\rho}_t])\rangle dW_t^i.
\end{align*}
Taking expectations and using the Cauchy--Schwarz inequality, we get
\begin{align*}
     \frac{d}{dt}\bbE\bar{\calO}_{t, p}&\leq  -p\lambda\bbE\left[|\Delta_t^\mathcal{M}|^{p-2}\langle \Delta_t^\mathcal{M}, \Delta_t^\mathcal{M} - \dm \rangle \right]  +p\lambda \bbE\left[|\dcal|^{p-2}\langle \dcal, m_f^h[\mmu]-\mrho\rangle\right]\\
    &\quad \;+\frac{p(p-1)\sigma^2}{2N^2}\bbE\left[|\Delta_t^\mathcal{M}|^{p-2}\sum_{i=1}^N |Z_t^i -(m_f^h[\mu_t^N]-m_f^h[\bar{\rho}_t])|^2\right]\\
    &=:I+II+III,
\end{align*}
where by using a similar argument as in the proof of Lemma \ref{lem: barcalD}, we estimate
\begin{align*}
    I & \leq p\lambda \bbE\left[|\dcal|^{p-1}|\dcal-\dm|\right]\\
    &\leq p\lambda \left(\bbE\bar{\calO}_{t, p}\right)^\frac{p-1}{p}\left(\bbE|\dcal-\dm|^p\right)^\frac{1}{p}\\
    &\leq C \left(\bbE\bar{\calO}_{t, p}\right)^\frac{p-1}{p}\left[\frac{1}{N^\frac{q-2}{2}} + \bbE\bar{\calG}_{t, p}\right]^\frac{1}{p}e^{-\frac{1}{p}\theta t}  \\
    &\leq C \left(\bbE\bar{\calO}_{t, p}\right)^\frac{p-1}{p}\left[\frac{1}{N^{\frac{q-2}{2}}}+2^{p-1}(\bbE\bar{\calD}_{t, p}+\bbE\bar{\calO}_{t, p})\right]^\frac{1}{p}e^{-\frac{1}{p}\theta t}   \\
    &\leq C \left[\frac{1}{N^\frac{q-2}{2p}}\left(\bbE\bar{\calO}_{t, p}\right)^\frac{p-1}{p}+2\bbE\bar{\calO}_{t, p}+2\left(\bbE\bar{\calO}_{t, p}\right)^\frac{p-1}{p}\left(\bbE\bar{\calD}_{t, p}\right)^\frac{1}{p}\right]e^{-\frac{1}{p}\theta t}\\
    &\leq C\left[\frac{1}{N^{\frac{q-2}{2}}} + \bbE\bar{\calO}_{t, p} + \bbE\bar{\calD}_{t, p}  \right]e^{-\frac{1}{p}\theta t},
\end{align*}
\begin{align*}
    II & \leq p\lambda\bbE\left[|\dcal|^{p-1}|m_f^h[\mmu]-\mrho|\right]\\
    &\leq p\lambda \left(\bbE\bar{\calO}_{t, p}\right)^\frac{p-1}{p}\left(\bbE|m_f^h[\mmu]-\mrho|^p\right)^\frac{1}{p}\\
    &\leq\frac{C}{N^\frac{1}{2}} \left(\bbE\bar{\calO}_{t, p}\right)^\frac{p-1}{p}e^{-(\lambda-\bbflambda_{p, \alpha, \sigma})t}\\
    &\leq C\left( \frac{1}{N^\frac{p}{2}} + \bbE\bar{\calO}_{t, p} \right) e^{-(\lambda-\bbflambda_{p, \alpha, \sigma})t},
\end{align*}
and
\begin{align*}
III &\leq \frac{p(p-1)\sigma^2}{2N}\left[\frac{1}{N}\sum_{i=1}^N \left(\bbE\bar{\calO}_{t, p}\right)^\frac{p-2}{p}\left(\bbE|Z_t^i -(m_f^h[\mu_t^N]-m_f^h[\bar{\rho}_t])|^p\right)^\frac{2}{p}\right]\\
    &\leq \frac{p(p-1)\sigma^2}{N}\left(\bbE\bar{\calO}_{t, p}\right)^\frac{p-2}{p}\left(\frac{1}{N}\sum_{i=1}^N\left(\bbE|X_t^i-m_f^h[\mu_t^N]|^p\right)^\frac{2}{p}+\left(\bbE|\bar{X}_t^i-\mrho|^p\right)^\frac{2}{p}\right)\\
    &\leq \frac{C}{N} \left(\bbE\bar{\calO}_{t, p}\right)^\frac{p-2}{p}e^{-2(\lambda-\bbflambda_{p, \alpha, \sigma})t}\\
    &\leq C\left[\frac{1}{N^\frac{p}{2}} + \bbE\bar{\calO}_{t, p} \right]e^{-2(\lambda-\bbflambda_{p, \alpha, \sigma})t}.
\end{align*}

Collecting the bounds and choosing $\widetilde\theta\in\big(0,\frac{\theta}{p}\wedge(\lambda-\bbflambda_{p,\alpha,\sigma})\big)$, we see that the decay rate of $I,II,III$ is given by $e^{-\widetilde\theta t}$, and we thus obtain \eqref{eq: barcalO} up to constructive constants.
\end{proof}

%
%
%
%
%
%
%
%
%
%
\subsubsection{Uniform-in-time bound for the fluctuation energy}

Combining \eqref{eq: barcalD} and \eqref{eq: barcalO} yields  
\[
    \frac{d}{dt}\left(\bbE\bar{\calD}_{t, p}+\bbE\bar{\calO}_{t, p}\right) \leq C \left(\bbE\bar{\calD}_{t, p}+\bbE\bar{\calO}_{t, p}\right)e^{-\widetilde\theta t} + \frac{C}{N^{\frac{p}{2}\wedge \frac{q-2}{2}}}e^{-\widetilde\theta t}
\]
for some constant $C>0$ depending only on $\lambda,\sigma,p,q$, $\Lambda_\alpha$, $K_\alpha$, $\bar\rho_0$ and the fixed admissible decay rates.
Thus, applying Gr\"onwall's inequality yields
\begin{equation*}
    \bbE\bar{\calD}_{t, p}+\bbE\bar{\calO}_{t, p}\leq \left[\bbE\bar{\calD}_{0, p}+\bbE\bar{\calO}_{0, p}+\frac{C}{N^{\frac{p}{2}\wedge \frac{q-2}{2}}}\right]e^{C/\widetilde{\theta}},
\end{equation*}
and by \eqref{eq: relation}, we deduce
\[
    \bbE\bar{\calG}_{t, p}\leq \left[(2^{2p-1}+2^{p-1})\bbE\bar{\calG}_{0, p}+\frac{2^{p-1}C}{N^{\frac{p}{2}\wedge \frac{q-2}{2}}}\right]e^{C/\widetilde{\theta}}.
\]
Hence, for $\lambda>\widetilde\bflambda_{p,q,\alpha,\sigma}$, there exists $C>0$, depending only on $\lambda,\sigma,p,q$, $\Lambda_\alpha$, $K_\alpha$, $\bar\rho_0$ and the fixed admissible decay rates, such that
\begin{equation}\label{eq: synchronous coupling rate}
\sup_{t\ge0}\frac1N\sum_{i=1}^N\bbE|X_t^i-\bar X_t^i|^p
\le \frac CN\sum_{i=1}^N\bbE|X_0^i-\bar X_0^i|^p
+\frac{C}{N^{\frac p2\wedge\frac{q-2}{2}}}.
\end{equation}
For the synchronous initialization used here, the initial-error term vanishes. This estimate is the input for the consensus-point rate \eqref{eq:cons-lim-r}; no empirical sampling estimate is needed for that application.

%
%
%
%
%
%
%
%
%
%

\subsubsection{Mean-field convergence and propagation of chaos}

We now rigorously deduce the uniform-in-time propagation of chaos from the quantitative estimates obtained above. The key quantity of interest is the $p$-Wasserstein distance between the empirical distribution of the particle system and the limiting mean-field distribution. By the triangle inequality, we can decompose it as
\bq\label{eq: wasserstein}
        \calW_p(\mu_t^N, \bar{\rho}_t)\leq \calW_p(\mu_t^N, \mmu)+\calW_p(\mmu, \bar{\rho}_t),
\eq
where $\mmu := \frac{1}{N}\sum_{i=1}^N \delta_{\bar X_t^i}$ denotes the empirical measure associated with the nonlinear (mean-field) processes $\{\bar X_t^i\}_{i=1}^N$.  
Each term in \eqref{eq: wasserstein} has a distinct meaning: the first term quantifies the average discrepancy between the interacting particle trajectories and their mean-field counterparts, while the second term measures the deviation of the mean-field empirical distribution from its law $\bar\rho_t$. 

To estimate the second term, we recall a classical result from probability theory concerning the convergence rate of empirical measures in the Wasserstein distance. The following statement, adapted from Fournier and Guillin \cite[Theorem 1]{FG15}, provides a quantitative bound that depends only on the number of particles $N$, the dimension $d$, and the available moment bounds of the limiting distribution.

\begin{lemma}\label{lem: decay rate corrected}
For $\mu\in\calP(\R^d)$, let $(X_k)_{k\ge1}$ be i.i.d. with law $\mu$ and set $\mu_N=N^{-1}\sum_{k=1}^N\delta_{X_k}$. Assume $p>1$ and $\swabM_r(\mu)<\infty$ for some $r>2p$. Then there exists $C=C(p,d,r)>0$ such that
\[
\bbE\calW_p^p(\mu_N,\mu)
\le C\swabM_r^{p/r}(\mu)
\begin{cases}
N^{-1/2}+N^{-(r-p)/r},&p>d/2,\\[3pt]
N^{-1/2}\log(1+N)+N^{-(r-p)/r},&p=d/2,\\[3pt]
N^{-p/d}+N^{-(r-p)/r},&0<p<d/2.
\end{cases}
\]
In particular, bounds for $\bbE\calW_p$ are obtained only after taking the $p$-th root (by Jensen's inequality). This is the sufficient high-moment case of \cite[Theorem 1]{FG15} used below. The condition $r>2p$ holds automatically for $r=pq$ since $q>2$; no exceptional moment case is needed in Theorem \ref{thm: propagation of chaos}.
\end{lemma}

We work directly with the $p$-th power. The canonical coupling of the empirical measures gives
\[
\calW_p^p(\mu_t^N,\bar\mu_t^N) \le \frac1N\sum_{i=1}^N|X_t^i-\bar X_t^i|^p.
\]
Since the synchronous construction gives $X_0^i=\bar X_0^i$, the fluctuation estimate above yields
\begin{equation}\label{eq: strong poc corrected}
\sup_{t\ge0}\bbE\calW_p^p(\mu_t^N,\bar\mu_t^N) \le \frac{C}{N^{\frac p2\wedge\frac{q-2}{2}}}.
\end{equation}
For the empirical nonlinear processes, apply Lemma \ref{lem: decay rate corrected} with $r=pq$. Since $q>2$, the exceptional case $r=2p$ does not occur. The uniform $pq$-moment bound from Lemma \ref{lem: barX Lp bound} gives
\begin{equation}\label{eq: empirical FG corrected}
\sup_{t\ge0}\bbE\calW_p^p(\bar\mu_t^N,\bar\rho_t) \le C
\begin{cases}
N^{-1/2}+N^{-(q-1)/q},&p>d/2,\\[3pt]
N^{-1/2}\log(1+N)+N^{-(q-1)/q},&p=d/2,\\[3pt]
N^{-p/d}+N^{-(q-1)/q},&0<p<d/2.
\end{cases}
\end{equation}
Finally,
\[
\calW_p^p(\mu_t^N,\bar\rho_t) \le2^{p-1}\calW_p^p(\mu_t^N,\bar\mu_t^N) +2^{p-1}\calW_p^p(\bar\mu_t^N,\bar\rho_t).
\]
Combining \eqref{eq: strong poc corrected} and \eqref{eq: empirical FG corrected} proves \eqref{main_res}. Taking $p$-th roots gives the corresponding estimate for $\bbE\calW_p$; in particular, the empirical exponents in that formulation are divided by $p$. This completes the proof.


\subsection{Dynamic denominators and localized quadratic stability}
The global estimates above use the deterministic lower bound $h(\alpha)$. For optimization, it is useful to retain the actual denominator instead. Throughout this subsection, set
\[
F=f-\underline f,\quad w_\alpha=e^{-\alpha F},\quad \widehat h_\alpha=e^{\alpha\underline f}h(\alpha),\quad L_\alpha=\operatorname{Lip}(w_\alpha),
\]
and, for $\rho\in\calP_2(\R^d)$, write
\[
z_\rho=\int w_\alpha\,d\rho,\quad s_\rho=z_\rho+\widehat h_\alpha, \quad M_\rho=\calM(\rho),\quad V_\rho=V_2(\rho),\quad B_\rho=m_f^h[\rho]-M_\rho.
\]
This normalization rescales the entire weight and leaves $m_f^h$ unchanged.

\begin{lemma}[Stability with the actual denominator]
For every $\rho\in\calP_2(\R^d)$,
\begin{equation}\label{eq: dynamic offset}
B_\rho=\frac{\int(x-M_\rho)w_\alpha(x)\rho(dx)}{s_\rho}, \quad |B_\rho|^2\le\frac{z_\rho}{s_\rho^2}V_\rho\le\frac{V_\rho}{s_\rho}.
\end{equation}
If $s_\mu,s_\nu\ge\zeta>0$, then
\begin{equation}\label{eq: dynamic stability}
|B_\mu-B_\nu| \le\left[\frac{L_\alpha}{\zeta}(\sqrt{V_\mu}+\sqrt{V_\nu}) +\frac{L_\alpha^2}{\zeta^2}V_\nu\right]\calW_2(\mu,\nu).
\end{equation}
The same estimate holds with $\mu$ and $\nu$ interchanged.
\end{lemma}
\begin{proof}
The constant part of the weight cancels in the centered numerator. Since $0<w_\alpha\le1$, weighted Cauchy--Schwarz gives
\[
\left|\int(x-M_\rho)w_\alpha\,d\rho\right|^2 \le z_\rho\int|x-M_\rho|^2w_\alpha\,d\rho\le z_\rho V_\rho.
\]
Let $(X,Y)$ have a coupling of $\mu,\nu$, put $Z=X-Y$, and write
$N_\rho=\int(x-M_\rho)w_\alpha\,d\rho$. Centering gives the exact identity
\[
N_\mu-N_\nu =\bbE[(Z-\bbE Z)(w_\alpha(X)-z_\mu)] +\bbE[(Y-M_\nu)(w_\alpha(X)-w_\alpha(Y))].
\]
Moreover, using an independent copy $X'$,
\[
\operatorname{Var}(w_\alpha(X)) =\tfrac12\bbE|w_\alpha(X)-w_\alpha(X')|^2\le L_\alpha^2V_\mu.
\]
Consequently,
\[
|N_\mu-N_\nu|\le L_\alpha(\sqrt{V_\mu}+\sqrt{V_\nu}) (\bbE|X-Y|^2)^{1/2}.
\]
Also $|s_\mu-s_\nu|\le L_\alpha\calW_2(\mu,\nu)$ and
\[
|N_\nu|=\left|\int(y-M_\nu)(w_\alpha(y)-w_\alpha(M_\nu))\nu(dy)\right| \le L_\alpha V_\nu.
\]
Apply these bounds to
$B_\mu-B_\nu=(N_\mu-N_\nu)/s_\mu+ N_\nu(s_\nu-s_\mu)/(s_\mu s_\nu)$ and optimize over the coupling.
\end{proof}

For the mean-field law, the centered second moment satisfies the exact identity
\[
\frac{d}{dt}V_2(\bar\rho_t) =-(2\lambda-\sigma^2)V_2(\bar\rho_t)+\sigma^2|B_{\bar\rho_t}|^2.
\]
Indeed, $d(\bar X_t-M_{\bar\rho_t})=-\lambda(\bar X_t-M_{\bar\rho_t})dt+ \sigma(\bar X_t-m_f^h[\bar\rho_t])dW_t$, and $\int|x-m_f^h[\rho]|^2\rho(dx)=V_\rho+|B_\rho|^2$. Thus, on any interval on which $s_{\bar\rho_t}\ge\zeta$,
\[
\frac{d}{dt}V_2(\bar\rho_t)\le-\gamma_\zeta V_2(\bar\rho_t), \quad \gamma_\zeta=2\lambda-\sigma^2(1+\zeta^{-1}).
\]
This is a conditional estimate on the denominator, not a replacement of the hypotheses of Theorem \ref{thm: propagation of chaos}.

For the synchronous coupling, define
\[
Z_t^i=X_t^i-\bar X_t^i,\quad \Delta M_t=N^{-1}\sum_iZ_t^i,\quad G_t=N^{-1}\sum_i|Z_t^i|^2,\quad D_t=G_t-|\Delta M_t|^2,
\]
\[
e_t=B_{\mu_t^N}-B_{\bar\mu_t^N},\quad \xi_t=m_f^h[\bar\mu_t^N]-m_f^h[\bar\rho_t],\quad q_t=e_t+\xi_t.
\]
Since $m_f^h[\mu_t^N]-m_f^h[\bar\rho_t]=\Delta M_t+q_t$, It\^o's formula gives
\begin{equation}\label{eq: exact quadratic coupling}
dG_t=\left[-(2\lambda-\sigma^2)D_t +2\lambda\Delta M_t\cdot q_t+\sigma^2|q_t|^2\right]dt+d\calM_t,
\end{equation}
where $\calM_t$ is the corresponding local martingale. In particular, for $p=2$ the centered and mean errors can be combined before estimating them.

\begin{proposition}[Stopped stability without an inverse regularization factor]
Assume $0<\widehat h_\alpha\le1$ and $2\lambda\ge\sigma^2$. Let $\tau_N$ be a stopping time. Suppose that some fixed $A,\beta,\zeta>0$ satisfy, for all $t\ge0$,
\[
s_{\bar\rho_t}\ge\zeta,\quad V_2(\bar\rho_t)\le Ae^{-\beta t},
\]
and, almost surely for $t<\tau_N$,
\[
s_{\mu_t^N},s_{\bar\mu_t^N}\ge\zeta,\quad V_2(\mu_t^N)+V_2(\bar\mu_t^N)\le Ae^{-\beta t}.
\]
Then
\begin{equation}\label{eq: stopped uit}
\sup_{t\ge0}\bbE  G_{t\wedge\tau_N} \le C\left(\bbE G_0+\frac1N\right),
\end{equation}
where $C$ depends on $A,\beta,\zeta,L_\alpha,\lambda,\sigma$, but contains no factor $\widehat h_\alpha^{-1}$.
\end{proposition}
\begin{proof}
By \eqref{eq: dynamic stability}, on $\{t<\tau_N\}$, $|e_t|\le C e^{-\beta t/2}\sqrt{G_t}$. For the sampling error, let
\[
Y_t^i=(\bar X_t^i-m_f^h[\bar\rho_t]) (w_\alpha(\bar X_t^i)+\widehat h_\alpha).
\]
For each fixed $t$, these random vectors are independent and centered, and $\xi_t=(N^{-1}\sum_iY_t^i)/s_{\bar\mu_t^N}$. Hence,
\begin{align*}
\bbE [\mathds1_{\{t<\tau_N\}}|\xi_t|^2] &\le\frac{1}{\zeta^2}\bbE \left|N^{-1}\sum_iY_t^i\right|^2 =\frac{\bbE |Y_t^1|^2}{N\zeta^2} \le\frac{4(1+\zeta^{-1})A}{N\zeta^2}e^{-\beta t},
\end{align*}
where \eqref{eq: dynamic offset} was used for the population law. The indicator was removed only after bounding the denominator; independence is not asserted after conditioning on the stopping event. Using $|\Delta M_t|\le\sqrt{G_t}$ in \eqref{eq: exact quadratic coupling} and Young's inequality yields, for $t<\tau_N$,
\[
\text{drift}(G_t)\le C e^{-\beta t/2}G_t+C e^{\beta t/2}|\xi_t|^2.
\]
After localization of the martingale, integration and expectation give
\[
\bbE  G_{t\wedge\tau_N} \le\bbE  G_0+C\int_0^t e^{-\beta s/2}\bbE  G_{s\wedge\tau_N}\,ds +\frac{C}{N}\int_0^t e^{-\beta s/2}\,ds.
\]
Gr\"onwall's inequality and Fatou's lemma prove the claim.
\end{proof}

\begin{remark}
An unconditional estimate requires control of the exit event and of the error after exit. Fixed-time concentration of an empirical Gibbs mass does not by itself control its infimum over all times. Likewise, decay of $\bbE  V_2(\mu_t^N)$ is not a pathwise deterministic envelope for $V_2(\mu_t^N)$. These distinctions prevent replacing the established uniform-in-time theorem by \eqref{eq: stopped uit} without an additional exit-probability argument.
\end{remark}

\subsection{Finite-horizon approximation without a contraction threshold}
The next result supplies the finite-time approximation needed for the post-entry argument. It does not impose the large-$\lambda$ condition of Theorem \ref{thm: propagation of chaos}.

\begin{proposition}[Finite-horizon synchronous approximation]\label{prop: finite horizon}
Fix $\alpha,h(\alpha),\lambda,\sigma>0$, and assume Assumption \ref{assum: omega}. Let $\bar\rho_0\in\calP_{2+\delta}(\R^d)$ for some $\delta>0$, and initialize the nonlinear copies synchronously. For every $T<\infty$ and $R\ge1$,
\begin{equation}\label{eq: finite horizon localized}
\sup_{0\le t\le T}\bbE \frac1N\sum_i|X_t^i-\bar X_t^i|^2 \le\frac{C_{T,R}}{N}+C_T R^{-\delta/(2+\delta)}.
\end{equation}
The constants are independent of $N$, but may depend on the fixed parameters, including $h(\alpha)^{-1}$. In particular,
\begin{equation}\label{eq: finite horizon convergence}
\lim_{N\to\infty}\sup_{0\le t\le T} \bbE \calW_2^2(\mu_t^N,\bar\rho_t)=0.
\end{equation}
A global $C_T/N$ strong-coupling rate is not asserted under these hypotheses.
\end{proposition}
\begin{proof}
Write $r=2+\delta$. The bounded ratio of weights gives, pathwise,
\[
|m_f^h[\mu_t^N]|^r\le\Lambda_\alpha N^{-1}\sum_i|X_t^i|^r.
\]
The same estimate holds for the population law. The integral SDE, the Burkholder--Davis--Gundy inequality, and Gr\"onwall's inequality on $[0,T]$ then give
\begin{equation}\label{eq: finite horizon moments}
\frac1N\sum_i\bbE \sup_{t\le T}|X_t^i|^r +\bbE \sup_{t\le T}|\bar X_t|^r\le C_T.
\end{equation}
For example, the drift term is bounded by $T^{r-1}\int_0^t\bbE |X_s^i-m_f^h[\mu_s^N]|^r ds$, and the stochastic term by $C_rT^{r/2-1}$ times the same integral; average over $i$ and apply the displayed weight bound. This argument first applies to stopped solutions and then passes to the limit.

Let $\tau_R$ be the first time that $\swabM_2(\mu_t^N)+\swabM_2(\bar\mu_t^N)\ge R$. Equation \eqref{eq: finite horizon moments} implies $\bbP (\tau_R\le T)\le C_T/R$. On $[0,\tau_R)$, Lemma \ref{lem: bounding calB} gives
\[
|m_f^h[\mu_t^N]-m_f^h[\bar\mu_t^N]|\le C_R\sqrt{G_t}.
\]
Independence of the nonlinear copies and $\psi_h\ge h(\alpha)$ also give
\[
\bbE |\xi_t|^2\le\frac{1}{Nh(\alpha)^2} \bbE | (\bar X_t-m_f^h[\bar\rho_t])\psi_h(\bar X_t)|^2 \le\frac{C_T}{N}.
\]
Apply It\^o's formula to $G_{t\wedge\tau_R}$ and use these two bounds. Since $G_0=0$, one obtains $\sup_{t\le T}\bbE  G_{t\wedge\tau_R}\le C_{T,R}/N$. On the exit event, H\"older's inequality and \eqref{eq: finite horizon moments} yield
\[
\bbE [G_t\mathds1_{\{\tau_R\le t\}}] \le(\bbE G_t^{r/2})^{2/r}\bbP (\tau_R\le T)^{1-2/r} \le C_T R^{-\delta/(2+\delta)}.
\]
This proves \eqref{eq: finite horizon localized}. First, let $N\to\infty$ for fixed $R$, and then $R\to\infty$. Finally, the empirical nonlinear copies converge in expected squared Wasserstein distance, uniformly for $t\le T$, by their uniform $r$-moment bound and the empirical-measure estimate of \cite{FG15}, applied with a nonexceptional moment $r_0\in(2,r)$ (chosen to avoid the finitely many endpoint values). The triangle inequality proves \eqref{eq: finite horizon convergence}.
\end{proof}

\section{Well-posedness and large-time analysis of the mean-field system}\label{sec:mfs}

This section begins with well-posedness and deterministic consensus for the mean-field equation \eqref{eq: coupling}, together with its direct weighted-energy estimate. We then prove Corollary \ref{cor:cons-est} and establish the regularized quantitative Laplace estimate. Finally, finite-time mean-field entrance, finite-horizon particle approximation, and convex trapping are combined to prove the high-probability optimization result of Theorem \ref{thm: finite global optimization}.

\subsection{Well-posedness} \label{subsec: well posedness}
 
The mean-field equation \eqref{eq: coupling} differs fundamentally from the finite-particle system since its drift and diffusion coefficients depend on the evolving law $\bar\rho_t$.  While each linearized equation with a prescribed path can be handled by classical SDE theory,  the full McKean--Vlasov dynamics require closing a nonlinear self-consistency condition between the process and its own distribution.  For this reason, the well-posedness of \eqref{eq: coupling} does not follow automatically from standard existence results and must be established separately.

The theorem below confirms that the McKean--Vlasov equation admits a unique strong solution on every finite time interval.

\begin{theorem}\label{thm: WP MFS}
     Let $f$ satisfy Assumption \ref{assum: omega}, and suppose $\bar{\rho}_0\in\calP_2(\bbR^d)$. Then the following nonlinear stochastic differential equation
     \[d\bar{X}_t=-\lambda(\bar{X}_t-m_f^h[\bar{\rho}_t])dt+\sigma(\bar{X}_t-m_f^h[\bar{\rho}_t])dW_t, \quad \bar{\rho}_t:=\textnormal{Law}(\bar{X}_t),\]
     has a unique continuous strong solution $\bar{X}\in \calC([0, T]; \bbR^d)$ for each time horizon $T>0$.
\end{theorem}

\begin{remark}
For every finite $T$, the proof below gives $\sup_{t\le T}\bbE |\bar X_t|^2\le C_T\bbE |\bar X_0|^2$ without a large-$\lambda$ hypothesis. Applying the theorem on arbitrary finite horizons and using pathwise uniqueness yields a global strong solution. The uniform-in-time bound of Lemma \ref{lem: barX Lp bound} additionally requires its stated contraction threshold and is not used for this continuation.
\end{remark}

\begin{proof}[Proof of Theorem \ref{thm: WP MFS}]
    For each $v\in \calC([0, T]; \bbR^d)$, let $Y_t = Y_t(v)$ be the unique strong solution to the linear SDE:
    \[dY_t=-\lambda(Y_t-v_t)dt+\sigma(Y_t-v_t)dW_t, \quad Y_0\sim \bar{\rho}_0\in \calP_2(\bbR^d).\]
    We refer to \cite{Durrett1996} as a classical reference for the well-posedness of the above equation. Then, we denote $\eta_t:=\textnormal{law}(Y_t)$, and define the operator $\calT :\calC([0, T]; \bbR^d)\to \calC([0, T]; \bbR^d)$ with $\calT : v \mapsto m_f^h[\eta]$. Note that the solution to \eqref{eq: coupling} is realized by a fixed point of $\calT$. We will thus show in the steps to follow that $\calT$ satisfies all assumptions of the Leray--Schauder--Schaefer fixed point theorem.

    (Compactness) We denote $B_R:=\left\{v\in \calC([0, T]; \bbR^d): \|v\|_\infty\leq R\right\}$, with $\|v\|_{\infty} := \sup_{t\in [0,T]} |v_t|$. Through elementary calculations, we can first check that $\sup_{t\in [0, T]}\bbE|Y_t|^2\leq K$ for some constant $K=K(R, T, \bar{\rho}_0)>0.$ Thus, using It\'o's isometry: 
    \begin{align*}
        \bbE\left|Y_t-Y_s\right|^2&\leq 2\lambda^2\bbE\left|\int_s^t (Y_r-v_r)dr\right|^2+2\sigma^2\int_s^t\bbE\left|Y_r-v_r\right|^2dr\\
        &\leq 2\lambda ^2|t-s|\int _s^t\bbE\left|Y_r-v_r\right|^2dr+2\sigma^2\int_s^t\bbE\left|Y_r-v_r\right|^2dr\\
        &\leq 4(\lambda^2T+\sigma^2) (K+ \|v\|_\infty^2) |t-s|.
    \end{align*}
    By definition of the Wasserstein distance we have $\calW_2(\eta_t,\eta_s) \le \lt(\bbE|Y_t-Y_s|^2\rt)^{\frac12}$, from which we deduce $\calW_2(\eta_t, \eta_s)\leq C|t-s|^\frac{1}{2}$ for some time-independent constant $C>0$. Utilizing Lemma \ref{lem: bounding calB}, with $\mu=\eta_t$, $\nu=\eta_s$, and also using that $\sup_{[0, T]}\int_{\bbR^d} |x|^2d\eta_t\leq C$, we have
    \begin{equation}\label{eq: m_f diff}
        \begin{split}
            |m_f^h[\eta_t]-m_f^h[\eta_s]| &\leq |m_f^h[\eta_t] - \calM(\eta_t) - m_f^h[\eta_s] + \calM(\eta_s)| + |\calM(\eta_t) - \calM(\eta_s)| \leq C\calW_2(\eta_t, \eta_s)  \leq C|t-s|^\frac{1}{2},
        \end{split}
    \end{equation}
    thanks to the elementary estimate $|\calM(\mu)-\calM(\nu)| \le \calW_1(\mu,\nu) \le \calW_2(\mu,\nu)$. In particular, we note that $t\mapsto m_f^h[\eta_t]$ is H\"older continuous with exponent $\frac12$. 

The set $\calT(B_R)$ has uniformly bounded supremum norm and $1/2$-H\"older seminorm. Hence, it is relatively compact in $\calC([0,T];\R^d)$ by Arzel\`a--Ascoli, which proves compactness of $\calT$.

    (Continuity) We consider two dynamical equations driven by the same Wiener process $W_t$ and initialized by the same random variable $Y_0=\hat Y_0$ with law $\eta_0$:
    \begin{equation*}
        \begin{split}
            dY_t=-\lambda(Y_t-v_t)dt+\sigma(Y_t-v_t)dW_t, \\
            d\hat{Y}_t=-\lambda(\hat{Y}_t-\hat{v}_t)dt+\sigma(\hat{Y}_t-\hat{v}_t)dW_t.
        \end{split}
    \end{equation*}
    Applying It\'o's formula and Cauchy-Schwarz inequality, we obtain
    \begin{align*}
        \frac{d}{dt}\bbE|Y_t-\hat{Y}_t|^2&=-2\lambda \bbE\langle Y_t-\hat{Y}_t, Y_t-\hat{Y}_t-(v_t-\hat{v}_t)\rangle +\sigma^2\bbE|Y_t-\hat{Y}_t-(v_t-\hat{v}_t)|^2\\
        &\leq (-\lambda +2\sigma^2)\bbE|Y_t-\hat{Y}_t|^2+(\lambda +2\sigma^2)|v_t-\hat{v}_t|^2.
    \end{align*}
    By Gr\"onwall's inequality, we have
    \[\sup_{t\in [0, T]} \bbE|Y_t-\hat{Y}_t|^2\leq C \|v-\hat{v}\|_\infty^2\]
    for some $C>0$ depending only on $T, \lambda, \sigma$, and hence, $\sup_{t\in [0, T]}\calW_2(\eta_t, \hat{\eta}_t)\leq C\|v-\hat{v}\|_\infty$. Now Lemma \ref{lem: bounding calB} shows, similarly as before, that
    \[\sup_{t\in [0, T]}|m_f^h[\eta_t]-m_f^h[\hat{\eta}_t]|\leq C\|v-\hat{v}\|_\infty,\]
    which implies $\calT$ is continuous on $\calC([0,T];\R^d)$.
    
    (Eigenvector bounds)
    We now prove that the set
    \[\mathscr{A}:=\left\{v\in \calC([0, T], \bbR^d): \exists \, \tau\in [0, 1]\textnormal{ such that } v=\tau \calT v\right\}\]
    is bounded. Let $v\in \mathscr{A}$. With $\eta_t=\operatorname{Law}(Y_t(v))$ as above, the relation $v=\tau\calT v$ means $v_t=\tau m_f^h[\eta_t]$. Hence, for all $t\in(0,T)$,
    \begin{equation}\label{eq: |v|^2}
        |v_t|^2=\tau^2\left|m_f^h[\eta_t]\right|^2\leq \tau^2\Lambdaa\int _{\bbR^d}|x|^2d\eta_t.
    \end{equation}
    Further, using \eqref{eq: |v|^2}, we can compute
    \begin{align*}
        \frac{d}{dt}\int_{\bbR^d}|x|^2d\eta_t&=\int_{\bbR^d} \left[-2\lambda \langle x-v_t, x\rangle +\sigma^2 |x-v_t|^2\right]d\eta_t\\
        &\leq (-2\lambda +\sigma^2+|\lambda-\sigma^2|)\int_{\bbR^d} |x|^2d\eta_t+(\sigma^2+|\lambda-\sigma^2|)|v_t|^2\\
        &\leq C \int_{\bbR^d} |x|^2d\eta_t,
    \end{align*}
    where $C = -2\lambda +(1+\tau^2\Lambdaa)\sigma^2+(1+\tau^2\Lambdaa)|\lambda-\sigma^2|$. From Gr\"onwall's inequality, we deduce that
    \[\int_{\R^d} |x|^2d\eta_t\leq \left(\int_{\R^d} |x|^2d\bar{\rho}_0\right)e^{C t},\]
    and hence $\|v\|_\infty\le (\Lambda_\alpha\swabM_2(\bar\rho_0))^{1/2}e^{C_+T/2}$, where $C_+=\max(C,0)$ and $C$ can be bounded uniformly for $\tau\in[0,1]$, which guarantees the boundedness of $\mathscr{A}$.

    (Uniqueness)
    Up to this point, we have seen that all conditions of the Leray--Schauder--Schaefer fixed point theorem are satisfied. Thus, the existence of a fixed point $v_t = m_f^h[\bar{\rho}_t]$, and thus a solution $\bar{X}_t$ to \eqref{eq: coupling}, is guaranteed. We now verify this solution is unique. Let $\bar{X}_t$ and $\hat{X}_t$ be two solutions to \eqref{eq: coupling} with same initial data $\bar{X}_0$ and driven by the same Wiener process $W_t$. Let $\bar{\rho}_t:=\textnormal{Law}(\bar{X}_t)$ and $\hat{\rho}_t:=\textnormal{Law}(\hat{X}_t).$ Denote the difference $Z_t=\bar{X}_t-\hat{X}_t$, then $Z_t$ satisfies:
    \[dZ_t=-\lambda(Z_t-\Delta m_t)dt+\sigma(Z_t-\Delta m_t)dW_t,\]
    where $\Delta m_t := m_f^h[\bar{\rho}_t]-m_f^h[\hat{\rho}_t]$. Applying the It\'o formula and Cauchy--Schwarz inequality yields:
    \begin{align*}
        \frac{d}{dt}\bbE|Z_t|^2 =-2\lambda \bbE \langle Z_t, Z_t-\Delta m_t\rangle +\sigma^2 \bbE|Z_t-\Delta m_t|^2 \leq (-2\lambda +\sigma^2 +|\lambda-\sigma^2|)\bbE|Z_t|^2+(\sigma^2+|\lambda-\sigma^2|)|\Delta m_t|^2.
    \end{align*}
    Utilizing Lemma \ref{lem: bounding calB}, in the same  manner as in \eqref{eq: m_f diff}, we have $|\Delta m_t|^2\leq C\calW_2(\bar{\rho}_t, \hat{\rho}_t)^2\leq  C\bbE |Z_t|^2$. Substituting this into the differential inequality yields:
    \[\frac{d}{dt}\bbE|Z_t|^2\leq C \bbE|Z_t|^2\]
    for some time independent constant $C>0$. Applying Gr\"onwall's inequality then gives $\bbE|Z_t|^2=0$ for all $t\in [0, T]$.
 \end{proof}


\subsection{Deterministic consensus and a direct weighted-energy estimate}%
 
We now turn to the long-time behavior of the mean-field dynamics \eqref{eq: coupling}, providing the proof of Theorem \ref{thm:mf_main}. The mean-field law is deterministic, and the consensus point obtained below is therefore deterministic. This contrasts with the random consensus point of the finite-particle system. The proof uses the mean-field decay estimate in Lemma \ref{lem: mean field limit exponential decay} under its stated stability threshold.  

Our first objective is to show that the solution $\bar X_t$ indeed collapses to a single point in the large-time limit.
 
\begin{proof}[Proof of Theorem \ref{thm:mf_main} (i)]
    We first recall that $\calM(\bar{\rho}_t) = \bbE[\bar X_t]$ satisfies (by taking expectations in \eqref{eq: coupling})
    \begin{align*}
        \frac{d}{dt}\calM(\bar{\rho}_t) = -\lambda\bbE\left[\bar{X}_t-m_f^h[\bar{\rho}_t]\right].
    \end{align*}
    Hence, by \eqref{eq: EbarX-mrho} and Lemma \ref{lem: mean field limit exponential decay}, we find
    \begin{align*}
        \int_0^\infty \left|\frac{d}{dt}\calM(\bar{\rho}_t)\right|dt&\leq C \int _0^\infty \bbE |\bar{X}_t-m_f^h[\bar{\rho}_t]|dt\\
        &\leq C \int_0^\infty \left(\bbE|\bar{X}_t-m_f^h[\bar{\rho}_t]|^p\right)^\frac{1}{p}dt\\
        &\leq C \int_0^\infty \left(\bbE|\bar{X}_t-\bbE\bar{X}_t|^p\right)^\frac{1}{p}dt \\
        &\leq C \int_0^\infty  e^{-(\lambda -\bbflambda_{p, \alpha, \sigma})t}dt<\infty.
    \end{align*}
    Hence, $\frac{d\calM(\bar{\rho}_t)}{dt}\in L^1(0, \infty)$, and this provides the existence of
    \[x_\infty^\alpha := \calM(\bar{\rho}_0)+\int_0^\infty \frac{d}{dt}\calM(\bar{\rho}_t)dt=\lim_{t\to\infty} \calM(\bar{\rho}_t).\]
    Using Jensen's inequality, we then write
    \[\bbE|\bar{X}_t-x_\infty^\alpha|^p\leq 2^{p-1}\bbE|\bar{X}_t - \calM(\bar\rho_t)|^p+2^{p-1}|\calM(\bar\rho_t) - x_\infty^\alpha|^p.\]
    As $t\to\infty$, the first term on the right-hand side tends to $0$ by applying Lemma \ref{lem: mean field limit exponential decay}, and the second term tends to zero by definition of $x_\infty^\alpha$. 
\end{proof}

Theorem \ref{thm:mf_main} (i) shows that the mean-field dynamics converge to a deterministic point $x_\infty^\alpha$. 

The next result gives a direct fixed-parameter function-value estimate for the deterministic consensus point under the stronger smoothness assumptions.

\begin{proof}[Proof of Theorem \ref{thm:mf_main} (ii)]
    We follow the strategy used for the particle system in Lemma \ref{lem_imp} and Theorem \ref{thm_main1} (ii). Applying It\^o's formula to $\omega_f^\alpha(\bar{X}_t)$ using the McKean--Vlasov SDE \eqref{eq: coupling}, and then taking expectations, yields:

\begin{align*}
\frac{d}{dt}\|\omega_f^\alpha\|_{L^1(\bar\rho_t)}
&=\alpha\lambda\int_{\R^d}e^{-\alpha f(x)}\nabla f(x) \cdot(x-m_f^h[\bar\rho_t])\,\bar\rho_t(dx)\\ 
&\quad+\frac{\alpha\sigma^2}{2}\int_{\R^d}e^{-\alpha f(x)}  (x-m_f^h[\bar\rho_t])^\top  [\alpha\nabla f(x)\nabla f(x)^\top-\nabla^2f(x)]  (x-m_f^h[\bar\rho_t])\,\bar\rho_t(dx).
\end{align*}
To justify the lower differential inequality used below, first stop in bounded sets and apply It\^o's formula. For $\alpha\ge c_1$, $\nabla^2\omega_f^\alpha\succeq-\alpha c_0e^{-\alpha\underline f}I$; the negative parts are controlled by the second moments and $|\nabla\omega_f^\alpha|\le L_{\omega_f^\alpha}$. Fatou's lemma and dominated convergence then give the integrated lower inequality without requiring a global absolute bound on the Hessian term.

    Using Assumptions \ref{assum: omega} and \ref{assump: f hessian}, for $\alpha \geq c_1$, we have
    \begin{equation}\label{eq: lp omega}
        \begin{split}
            \frac{d}{dt}\|\omega_f^\alpha\|_{L^1(\bar{\rho}_t)}&\geq -\lambda L_{\omega _f^\alpha }\int_{\bbR^d} |x-m_f^h[\bar{\rho}_t]|d\bar{\rho}_t-\frac{\alpha \sigma^2 c_0e^{-\alpha \underline{f}} }{2}\int_{\bbR^d} |x-m_f^h[\bar{\rho}_t]|^2d\bar{\rho}_t\\
            &\geq -\lambda L_{\omega _f^\alpha }\sqrt{2\Lambdaa}\left(\int_{\bbR^d}|x-\calM(\bar{\rho}_t)|^2d\bar{\rho}_t\right)^\frac{1}{2}  -\alpha \sigma^2 c_0e^{-\alpha \underline{f}}\Lambdaa\int_{\bbR^d}|x-\calM(\bar{\rho}_t)|^2d\bar{\rho}_t,
        \end{split}
    \end{equation}
    where we used the fact that 
    \[\int_{\bbR^d} |x-m_f^h[\bar{\rho}_t]|^2d\bar{\rho}_t\leq \frac{\iint_{\bbR^d\times \bbR^d} |x-y|^2\psi_h(y)d\bar{\rho}_td\bar{\rho}_t}{\int_{\bbR^d} \psi_h(y)d\bar{\rho}_t}\leq 2\Lambdaa \int_{\bbR^d} |x-\calM(\bar{\rho}_t)|^2d\bar{\rho}_t.\]
    Hence, integrating both sides of \eqref{eq: lp omega} then applying Lemma \ref{lem: mean field limit exponential decay} yields
    \begin{align*}
        \|\omega_f^\alpha\|_{L^1(\bar{\rho}_t)}\geq  \|\omega_f^\alpha\|_{L^1(\bar{\rho}_0)}-\frac{\lambda L_{\omega_f^\alpha}\sqrt{2\Lambdaa}}{\lambda-\bbflambda_{2, \alpha, \sigma}}\textnormal{Var}(\bar{X}_0)^\frac{1}{2}(1-e^{-(\lambda-\bbflambda_{2, \alpha, \sigma})t}) -\frac{\alpha\sigma^2c_0e^{-\alpha \underline{f}}\Lambdaa}{2(\lambda-\bbflambda_{2, \alpha, \sigma})}\textnormal{Var}(\bar{X}_0)(1-e^{-2(\lambda-\bbflambda_{2, \alpha, \sigma})t}). 
    \end{align*}
    Letting $t\to\infty$ and applying the dominated convergence theorem yields
    \begin{align*}
        e^{-\alpha f(x_\infty^\alpha)}&\geq \|\omega_f^\alpha\|_{L^1(\bar{\rho}_0)}-\frac{\lambda L_{\omega_f^\alpha}\sqrt{2\Lambdaa}}{\lambda-\bbflambda_{2, \alpha, \sigma}}\textnormal{Var}(\bar{X}_0)^\frac{1}{2}        -\frac{\alpha\sigma^2c_0e^{-\alpha \underline{f}}\Lambdaa}{2(\lambda-\bbflambda_{2, \alpha, \sigma})}\textnormal{Var}(\bar{X}_0)\\
        &\geq \epsilon\|\omega_f^\alpha\|_{L^1(\bar{\rho}_0)}.
    \end{align*}
Hence, taking logarithms on both sides, we obtain
\[
f(x^\alpha_\infty) \leq -\frac{1}{\alpha} \log \|\omega_f^\alpha\|_{L^1(\bar\rho_0)} -\frac{1}{\alpha}\log\varepsilon .
\]
This is precisely the claimed fixed-parameter estimate and completes the proof.
\end{proof}


\subsection{Consensus-point estimates and the regularization error}\label{subsec: weak mf bridge}
We first prove the comparison stated in Corollary \ref{cor:cons-est}. We then establish the regularized quantitative Laplace estimate, which provides the localization ingredient for the finite-time entrance argument in Section \ref{subsec: explicit global route}.

\begin{proof}[Proof of Corollary \ref{cor:cons-est}]
Fix $T\ge0$. By Theorem \ref{thm_main1}(i),
\[
\|X_\infty^{N,\alpha}-x^*\|_{L^p} \le C_{\rm tail}e^{-\gamma_pT} +\|X_T^i-x^*\|_{L^p}.
\]
By exchangeability,
\[
\|X_T^i-x^*\|_{L^p}^p = \bbE\frac1N\sum_{j=1}^N|X_T^j-x^*|^p = \bbE\calW_p^p(\mu_T^N,\delta_{x^*}).
\]
The Wasserstein triangle inequality and Minkowski's inequality imply
\[
\|X_T^i-x^*\|_{L^p} \le \left( \bbE\calW_p^p(\mu_T^N,\bar\rho_T) \right)^{1/p} + \calW_p(\bar\rho_T,\delta_{x^*}) \le R_{N,p} + \calW_p(\bar\rho_T,\delta_{x^*}),
\]
which proves \eqref{eq:trans-opt-int}.

When Theorem \ref{thm:mf_main}(i) applies, the same argument with $x^*$ replaced by $x_\infty^\alpha$ gives, for every $T\ge0$,
\[
\|X_\infty^{N,\alpha}-x_\infty^\alpha\|_{L^p} \le C_{\rm tail}e^{-\gamma_pT} + R_{N,p} + \calW_p(\bar\rho_T,\delta_{x_\infty^\alpha}).
\]
Letting $T\to\infty$ and using Theorem \ref{thm:mf_main}(i) proves \eqref{eq:cons-lim-comp}.

Finally, under the hypotheses of Theorem \ref{thm: propagation of chaos}, let $\bar X_t^i$ be the synchronously coupled nonlinear copies and set
\[
a_{N,p}(T):= \left( \frac1N\sum_{i=1}^N \bbE|X_T^i-\bar X_T^i|^p \right)^{1/p}.
\]
Minkowski's inequality on the product of $\Omega$ with the uniform probability measure on $\{1,\ldots,N\}$ gives
\[
\|X_\infty^{N,\alpha}-x_\infty^\alpha\|_{L^p} \le C_{\rm tail}e^{-\gamma_pT} + a_{N,p}(T) + \calW_p(\bar\rho_T,\delta_{x_\infty^\alpha}).
\]
For the synchronous initialization, \eqref{eq: synchronous coupling rate} yields
\[
\sup_{T\ge0}a_{N,p}(T)^p \le C N^{-\beta_{p,q}}.
\]
Letting $T\to\infty$ and taking the $p$-th power proves \eqref{eq:cons-lim-r}. This completes the proof.
\end{proof}

We next isolate the part of the mean-field localization argument that does not require differentiating $f$. The following estimate is a regularized version of the quantitative Laplace principle.

\begin{lemma} \label{lem:reg-Lap}
Assume Assumption \ref{assum: weak landscape}, and let $\rho\in\calP_1(\R^d)$ satisfy $\rho(B_r(x^*))>0$ for some $0<r\le R_0$, and set
\[
F_r:=\sup_{x\in B_r(x^*)}F(x).
\]
We assume the existence of $q>0$ such that $q+F_r\le f_\infty$. Then
\begin{equation}\label{eq:reg-Lapl}
|m_f^h[\rho]-x^*| \le \frac{(q+F_r)^\nu}{\eta} +\frac{e^{-\alpha q}+\widehat h_\alpha e^{\alpha F_r}}{\rho(B_r(x^*))} \int_{\R^d}|x-x^*|\,\rho(dx),
\end{equation}
where $\widehat h_\alpha=e^{\alpha\underline f}h(\alpha)$.
\end{lemma}

\begin{proof}
Write $Z_\alpha=\int e^{-\alpha F}\,d\rho$ and let $v_\alpha[\rho]$ be the classical Gibbs mean defined in \eqref{eq:reg-conv-deco}. The identity there shows
\[
|m_f^h[\rho]-x^*| \le |v_\alpha[\rho]-x^*| +\frac{\widehat h_\alpha}{Z_\alpha}|\calM(\rho)-x^*|.
\]
Set $A:=\{x:F(x)\le q+F_r\}$. The far-field part of Assumption \ref{assum: weak landscape} and $q+F_r\le f_\infty$ imply $A\subset\overline B_{R_0}(x^*)$, and hence Assumption \ref{assum: weak landscape} again gives
\[
|x-x^*|\le \frac{(q+F_r)^\nu}{\eta}\quad\text{on }A.
\]
Moreover,
\[
Z_\alpha\ge \int_{B_r(x^*)}e^{-\alpha F(x)}\rho(dx) \ge e^{-\alpha F_r}\rho(B_r(x^*)).
\]
Splitting the numerator defining $v_\alpha$ over $A$ and $A^c$ gives
\[
|v_\alpha[\rho]-x^*| \le \frac{(q+F_r)^\nu}{\eta} +\frac{e^{-\alpha q}}{\rho(B_r(x^*))}\int_{\R^d}|x-x^*|\rho(dx).
\]
Finally,
\[
\frac{\widehat h_\alpha}{Z_\alpha}|\calM(\rho)-x^*| \le \frac{\widehat h_\alpha e^{\alpha F_r}}{\rho(B_r(x^*))} \int_{\R^d}|x-x^*|\rho(dx),
\]
which proves \eqref{eq:reg-Lapl}.
\end{proof}

\begin{remark}%
Estimate \eqref{eq:reg-Lapl} makes the effect of the additive regularization explicit. In addition to the classical Laplace error $e^{-\alpha q}$, one obtains the term $\widehat h_\alpha e^{\alpha F_r}$. Thus an optimizer-compatible large-$\alpha$ regime requires this latter quantity to be small on the relevant localization scale. On the other hand,
\[
\Lambda_\alpha=1+\widehat h_\alpha^{-1},
\]
so making $\widehat h_\alpha$ small enlarges the thresholds in the existing global estimates. A sequential limit fixes parameters when $N\to\infty$, but does not itself prove that the optimization and stability requirements admit a common parameter choice. The alternative finite-horizon and post-entry route below avoids the global $\Lambda_\alpha$ threshold, while retaining its own explicit compatibility conditions.

\end{remark}

For completeness, the Lyapunov computation behind the weak-landscape route is also derivative-free. Define
\[
\calV(t):=\frac12\int_{\R^d}|x-x^*|^2\bar\rho_t(dx),\quad b_t:=\calM(\bar\rho_t)-x^*,\quad \delta_t:=m_f^h[\bar\rho_t]-x^*.
\]
A direct computation from \eqref{eq_pde} gives
\begin{equation}\label{eq:weak-lands-Lyap}
\frac{d}{dt}\calV(t) =-(2\lambda-\sigma^2)\calV(t) +(\lambda-\sigma^2)b_t\cdot\delta_t +\frac{\sigma^2}{2}|\delta_t|^2.
\end{equation}
Since $|b_t|\le\sqrt{2\calV(t)}$, if $2\lambda>\sigma^2$ and, on a time interval, $|\delta_t|\le c\sqrt{\calV(t)}$ with
\[
\sqrt2|\lambda-\sigma^2|c+\frac{\sigma^2}{2}c^2\le \tau(2\lambda-\sigma^2),\quad 0<\tau<1,
\]
then
\[
\frac{d}{dt}\calV(t)\le-(1-\tau)(2\lambda-\sigma^2)\calV(t).
\]

The Lyapunov inequality and the regularized Laplace estimate provide a finite-time entrance criterion once a quantitative local-mass bound is available. For the rank-one diffusion, an estimate proved for an isotropic or componentwise model cannot be imported verbatim. Moreover, transfer to an asymptotic particle optimizer requires a compatible tail bound or post-entry certificate. These are distinct requirements, recorded explicitly below rather than inferred from a function-value estimate.


\subsection{An explicit global-optimization route}\label{subsec: explicit global route}
We now replace the solution-dependent certificates above by verifiable assumptions on the input data. The first observation is that the local mass required by Lemma \ref{lem:reg-Lap} can be estimated from the affine stochastic flow itself; there is no need to control the total variation of the mean path.

\begin{lemma} \label{lem: affine local mass} Let $m:[0,T]\to\R^d$ be deterministic and continuous, and assume
\[
\sup_{0\le t\le T}|m_t-x^*|\le B, \quad 2\lambda\ge\sigma^2, \quad \sigma>0.
\]
Let $X_t$ solve
\[
dX_t=-\lambda(X_t-m_t)dt+\sigma(X_t-m_t)dW_t, \quad X_0\sim\rho_0,
\]
with $X_0$ independent of $W$. Suppose $\operatorname{supp}\rho_0=\R^d$. For $r,R>0$ set
\[
\kappa_{\rho_0}(r,R):= \inf_{|y-x^*|\le R}\rho_0(B_r(y))>0.
\]
Define
\begin{align*}
K_T(B) :=2B^2\left((\lambda+\sigma^2)^2T+\sigma^2\right) \int_0^T e^{(2\lambda+3\sigma^2)s}\,ds,\quad  R_T(B) :=1+2\sqrt{K_T(B)}.
\end{align*}
Then, with $\rho_t=\operatorname{Law}(X_t)$,
\begin{equation}\label{eq:aff-loc-mass}
\inf_{0\le t\le T}\rho_t(B_r(x^*)) \ge \frac12\kappa_{\rho_0}\!\left(\frac r2,R_T(B)\right)>0.
\end{equation}
In particular, the lower bound uses only a uniform bound on $m_t$ and does not involve $\operatorname{TV}(m)$ or a time derivative of $m$.
\end{lemma}
 \begin{proof}
Write $\widetilde X_t:=X_t-x^*$ and $\widetilde m_t:=m_t-x^*$. Then
\[
d\widetilde X_t=-\lambda \widetilde X_t\,dt+\sigma \widetilde X_t\,dW_t +\lambda \widetilde m_t\,dt-\sigma \widetilde m_t\,dW_t.
\]
Let
\[
\Phi_t:=\exp\!\left[-\left(\lambda+\frac{\sigma^2}{2}\right)t+\sigma W_t\right].
\]
Then $\Phi_t>0$ a.s. and $d\Phi_t=-\lambda\Phi_t\,dt+\sigma\Phi_t\,dW_t$. Moreover, It\^o's formula gives
\[
d\Phi_t^{-1}=(\lambda+\sigma^2)\Phi_t^{-1}\,dt-\sigma\Phi_t^{-1}\,dW_t.
\]
Applying It\^o's product rule to $\Phi_t^{-1}\widetilde X_t$, we obtain
\begin{align*}
d(\Phi_t^{-1}\widetilde X_t)  =\Phi_t^{-1}\,d\widetilde X_t+\widetilde X_t\,d\Phi_t^{-1} +d\langle \Phi^{-1},\widetilde X\rangle_t  =(\lambda+\sigma^2)\Phi_t^{-1}\widetilde m_t\,dt -\sigma\Phi_t^{-1}\widetilde m_t\,dW_t.
\end{align*}
Indeed, all the terms containing $\widetilde X_t$ cancel. Integrating over $[0,t]$ yields
\[
\Phi_t^{-1}\widetilde X_t =\widetilde X_0+(\lambda+\sigma^2)\int_0^t\Phi_s^{-1}\widetilde m_s\,ds -\sigma\int_0^t\Phi_s^{-1}\widetilde m_s\,dW_s.
\]
Thus
\[
X_t-x^*=\Phi_t\lt(X_0-x^*-C_t\rt),
\]
where
\[
C_t:=-(\lambda+\sigma^2)\int_0^t\Phi_s^{-1}\widetilde m_s\,ds +\sigma\int_0^t\Phi_s^{-1}\widetilde m_s\,dW_s.
\]

We next estimate $C_t$. Since $\Phi_s^{-2}=\exp\!\left[(2\lambda+\sigma^2)s-2\sigma W_s\right]$, the Gaussian moment formula gives
\[
\bbE\Phi_s^{-2} =e^{(2\lambda+\sigma^2)s}\bbE e^{-2\sigma W_s} =e^{(2\lambda+3\sigma^2)s}.
\]
Using $|\widetilde m_s|\le B$, the inequality $|a+b|^2\le 2|a|^2+2|b|^2$, Cauchy--Schwarz, and It\^o's isometry, we find
\begin{align*}
\bbE|C_t|^2 &\le 2(\lambda+\sigma^2)^2 \bbE\left|\int_0^t\Phi_s^{-1}\widetilde m_s\,ds\right|^2 +2\sigma^2\bbE\left|\int_0^t\Phi_s^{-1}\widetilde m_s\,dW_s\right|^2 \\
&\le 2B^2\left((\lambda+\sigma^2)^2t+\sigma^2\right) \int_0^t e^{(2\lambda+3\sigma^2)s}\,ds \\
&\le 2B^2\left((\lambda+\sigma^2)^2T+\sigma^2\right) \int_0^T e^{(2\lambda+3\sigma^2)s}\,ds \\
&=K_T(B).
\end{align*}
Hence
\[
\sup_{0\le t\le T}\bbE|C_t|^2\le K_T(B).
\]

On the other hand,
\[
\bbE\Phi_t^2 =e^{-(2\lambda+\sigma^2)t}\bbE e^{2\sigma W_t} =e^{(-2\lambda+\sigma^2)t}\le 1,
\]
where the last inequality follows from $2\lambda\ge \sigma^2$. Thus, by Markov's inequality,
\[
\bbP(\Phi_t>2)\le \frac14\bbE\Phi_t^2\le\frac14.
\]
Since $R_T(B)=1+2\sqrt{K_T(B)}\ge 2\sqrt{K_T(B)}$, we also have
\[
\bbP\bigl(|C_t|>R_T(B)\bigr) \le \frac{\bbE|C_t|^2}{R_T(B)^2}\le\frac14.
\]
Consequently,
\[
\bbP\bigl(\Phi_t\le2,\, |C_t|\le R_T(B)\bigr) \ge 1-\bbP(\Phi_t>2)-\bbP\bigl(|C_t|>R_T(B)\bigr)\ge\frac12.
\]

We now condition on the Brownian path. Since $\Phi_t$ and $C_t$ are measurable with respect to $\sigma(W_s:0\le s\le t)$, while $X_0$ is independent of $W$, the representation above gives
\begin{align*}
\rho_t(B_r(x^*)) =\bbP\bigl(|X_t-x^*|<r\bigr) =\bbE\left[\rho_0\left(B_{r/\Phi_t}(x^*+C_t)\right)\right].
\end{align*}
Set
\[
A_t:=\bigl\{\Phi_t\le2,\, |C_t|\le R_T(B)\bigr\}.
\]
On $A_t$, we have $r/\Phi_t\ge r/2$ and $|C_t|\le R_T(B)$, and thus
\[
\rho_0\left(B_{r/\Phi_t}(x^*+C_t)\right) \ge \rho_0\left(B_{r/2}(x^*+C_t)\right) \ge \kappa_{\rho_0}\left(\frac r2,R_T(B)\right).
\]
It follows that
\begin{align*}
\rho_t(B_r(x^*)) \ge \bbE\left[\mathbf 1_{A_t} \rho_0\left(B_{r/\Phi_t}(x^*+C_t)\right)\right] \ge \kappa_{\rho_0}\left(\frac r2,R_T(B)\right)\bbP(A_t) \ge \frac12\kappa_{\rho_0}\left(\frac r2,R_T(B)\right).
\end{align*}
Since the estimate is uniform for $t\in[0,T]$, this proves \eqref{eq:aff-loc-mass}.

It remains to verify that $\kappa_{\rho_0}(r,R)>0$. Set $g_r(y):=\rho_0(B_r(y))$. Since $\operatorname{supp}\rho_0=\R^d$, we have $g_r(y)>0$ for every $y\in\R^d$. We claim that $g_r$ is lower semicontinuous. Indeed, if $y_n\to y$, then for every $\delta\in(0,r)$, $B_{r-\delta}(y)\subset B_r(y_n)$ for all sufficiently large $n$. Hence
\[
\liminf_{n\to\infty}g_r(y_n)\ge \rho_0(B_{r-\delta}(y)).
\]
Letting $\delta\downarrow0$ and using continuity from below of $\rho_0$, we obtain
\[
\liminf_{n\to\infty}g_r(y_n)\ge \rho_0(B_r(y))=g_r(y).
\]
Thus $g_r$ is lower semicontinuous. Since $\overline B_R(x^*)$ is compact and $g_r>0$ there, $g_r$ attains a strictly positive minimum on $\overline B_R(x^*)$. Therefore,
\[
\kappa_{\rho_0}(r,R)=\inf_{|y-x^*|\le R}\rho_0(B_r(y))>0.
\]

Finally, the argument uses $m$ only through the bound
\[
\sup_{0\le t\le T}|m_t-x^*|\le B.
\]
This completes the proof.
\end{proof}

\begin{proposition}[Mean-field entrance from full support]\label{prop: full support entrance}
Assume Assumptions \ref{assum: omega} and \ref{assum: weak landscape} and let $\bar\rho_0\in\calP_2(\R^d)$ satisfy $\operatorname{supp}\bar\rho_0=\R^d$. Let
\[
\calV(t)=\frac12\int_{\R^d}|x-x^*|^2\bar\rho_t(dx), \quad 0<\calV(0)<\infty,
\]
and assume $2\lambda>\sigma^2$. Choose $0<\vartheta<1$ and $c_\vartheta>0$ so that
\begin{equation}\label{eq:entr-ctheta}
\sqrt2|\lambda-\sigma^2|c_\vartheta +\frac{\sigma^2}{2}c_\vartheta^2 \le \vartheta(2\lambda-\sigma^2), \quad \kappa=(1-\vartheta)(2\lambda-\sigma^2).
\end{equation}
For every $0<\varepsilon<\calV(0)$ there are $r\in(0,R_0)$, $q>0$, $\alpha>0$, and $h(\alpha)>0$ such that the corresponding mean-field solution satisfies, at some deterministic $T_*\le T_\varepsilon:=\kappa^{-1}\log(\calV(0)/\varepsilon)$,
\begin{equation}\label{eq:uncond-ent}
\calV(T_*)\le\varepsilon, \quad \calW_2(\bar\rho_{T_*},\delta_{x^*})\le\sqrt{2\varepsilon}.
\end{equation}
\end{proposition}

\begin{proof}
We first choose the parameters. Let $\ell\in(0,f_\infty)$ be such that
\[
\frac{\ell^\nu}{\eta}\le\frac{c_\vartheta}{4}\sqrt\varepsilon.
\]
Since $F(x^*)=0$ and $F$ is continuous, we can choose $r\in(0,R_0)$ sufficiently small so that $F_r:=\sup_{B_r(x^*)}F\le\ell/2$, and set $q=\ell/2$. Thus
\begin{equation}\label{eq:entr-rq-cho}
q+F_r\le f_\infty, \quad
\frac{(q+F_r)^\nu}{\eta}\le\frac{c_\vartheta}{4}\sqrt\varepsilon.
\end{equation}
Set $B=c_\vartheta\sqrt{\calV(0)}$ and
\[
c_{\rm mass}:=\frac12\kappa_{\bar\rho_0} \left(\frac r2,R_{T_\varepsilon}(B)\right)>0,
\]
where $R_T(B)$ is defined in Lemma \ref{lem: affine local mass}. Since $c_{\rm mass}$ is independent of $\alpha$ and $h$, we may choose $\alpha$ sufficiently large and then $h(\alpha)>0$ sufficiently small so that
\begin{equation}\label{eq:entr-a-h-cho}
e^{-\alpha q}\le\frac{c_\vartheta c_{\rm mass}}{8\sqrt2}, \quad \widehat h_\alpha e^{\alpha F_r} \le\frac{c_\vartheta c_{\rm mass}}{8\sqrt2}.
\end{equation}

Let $m_t=m_f^h[\bar\rho_t]$ and $\delta_t=m_t-x^*$. Consider the maximal interval, contained in $[0,T_\varepsilon]$, on which $|\delta_t|<B$ and $\calV(t)>\varepsilon$. For every $t$ in this interval, apply Lemma \ref{lem: affine local mass} on $[0,t]$. Since $t\le T_\varepsilon$, the bound $K_t(B)\le K_{T_\varepsilon}(B)$ shows that the same constant $c_{\rm mass}$ works uniformly, and hence
\[
\bar\rho_t(B_r(x^*))\ge c_{\rm mass}.
\]
Combining Lemma \ref{lem:reg-Lap}, Cauchy--Schwarz, and \eqref{eq:entr-rq-cho}--\eqref{eq:entr-a-h-cho}, we obtain
\begin{align*}
|\delta_t| \le \frac{c_\vartheta}{4}\sqrt\varepsilon +\frac{e^{-\alpha q}+\widehat h_\alpha e^{\alpha F_r}} {c_{\rm mass}}\sqrt{2\calV(t)} \le\frac{c_\vartheta}{2}\sqrt{\calV(t)}.
\end{align*}
The same estimate at $t=0$ starts the bootstrap. Inserting it into \eqref{eq:weak-lands-Lyap} and using \eqref{eq:entr-ctheta} yields $\calV'(t)\le-\kappa\calV(t)$. Hence $\calV$ is decreasing and $|\delta_t|\le B/2$, so the bootstrap cannot end through the mean-path boundary. If the level $\varepsilon$ were not reached before $T_\varepsilon$, Gr\"onwall's inequality would give $\calV(T_\varepsilon)\le\varepsilon$, a contradiction to the definition of the interval. This proves \eqref{eq:uncond-ent}.
\end{proof}

We next construct a globally strongly convex energy that agrees with the objective near $x^*$. Its sufficiently small sublevels will provide a trapping region independent of $\alpha$. The construction uses local nondegeneracy together with the landscape condition away from the minimizer.

\begin{lemma} \label{lem:conv-gau}
Assume Assumptions \ref{assum: weak landscape} and \ref{assum: local nondegenerate}. Let $Q=\nabla^2f(x^*)$ and introduce the affine coordinates
\[
y=Q^{1/2}(x-x^*), \quad \widetilde F(y)=F(x^*+Q^{-1/2}y).
\]
Then $\nabla^2\widetilde F(0)=\bbI_d$. For every $\theta\in(0,1)$ there exist $g\in\calC^2(\R^d)$ and $a>0$ such that
\begin{equation}\label{eq:conv-gau-hes}
g(0)=0,\quad \nabla g(0)=0, \quad (1-\theta)\bbI_d\preceq\nabla^2g(y)\preceq(1+\theta)\bbI_d \quad\text{for all }y,
\end{equation}
and
\begin{equation}\label{eq:conv-gau-mat}
g(y)\le a\ \Longrightarrow\ g(y)=\widetilde F(y), \quad g(y)>a\ \Longrightarrow\ \widetilde F(y)\ge a.
\end{equation}
\end{lemma}

\begin{proof}
Write
\[
\widetilde F(y)=\frac12|y|^2+R(y).
\]
Since $\widetilde F\in\calC^2$ near the origin and $\nabla^2\widetilde F(0)=\bbI_d$,
\[
R(y)=o(|y|^2),\quad \nabla R(y)=o(|y|),\quad \nabla^2R(y)=o(1) \quad (y\to0).
\]
Let $\chi\in C_c^\infty(B_2)$ satisfy $\chi=1$ on $B_1$ and set
\[
g_r(y)=\frac12|y|^2+\chi(y/r)R(y).
\]
The terms arising when two derivatives hit the cutoff are respectively of size $O(\|\nabla^2R\|)$, $O(r^{-1}\|\nabla R\|)$, and $O(r^{-2}\|R\|)$ on $B_{2r}$; all tend to zero as $r\downarrow0$. Thus, fixing $r$ sufficiently small, $g=g_r$ satisfies \eqref{eq:conv-gau-hes} and agrees with $\widetilde F$ on $B_r$.

The transformed weak-landscape condition implies
\[
d_r:=\inf_{|y|\ge r}\widetilde F(y)>0.
\]
Indeed, the inverse-continuity estimate gives strict positivity near the origin, the far-field condition controls the complement of a sufficiently large ball, and continuity together with uniqueness of the minimizer gives a positive minimum on the remaining compact annulus. Since
$g(y)\ge(1-\theta)|y|^2/2$, choose
\[
0<a<\min\left\{d_r,\frac{1-\theta}{2}r^2\right\}.
\]
Then $\{g\le a\}\subset B_r$, where $g=\widetilde F$, and
\eqref{eq:conv-gau-mat} follows.
\end{proof}

\begin{lemma} \label{lem:w-conv-ener}
Let $g,a$ be given by Lemma \ref{lem:conv-gau}, let $\rho$ be a probability measure in the $y$-coordinates, and assume
\[
H(\rho):=\int_{\R^d} g(y)\rho(dy)\le\frac a4, \quad \alpha a\ge2\log2.
\]
For any $\widehat h_\alpha\ge0$, define
\[
\widetilde m_{\alpha,\widehat h}[\rho] :=\frac{\int_{\R^d} y\bigl(e^{-\alpha\widetilde F(y)}+\widehat h_\alpha\bigr)\rho(dy)} {\int_{\R^d}\bigl(e^{-\alpha\widetilde F(y)}+\widehat h_\alpha\bigr)\rho(dy)}.
\]
Then
\begin{equation}\label{eq:w-conv-ene}
g\bigl(\widetilde m_{\alpha,\widehat h}[\rho]\bigr) \le H(\rho).
\end{equation}
\end{lemma}

\begin{proof}
Let $Y\sim\rho$, $U=g(Y)$, $W=e^{-\alpha\widetilde F(Y)}$, and let $(U',W')$ be an independent copy. Put
\[
I=\{U\le a\},\quad O=\{U>a\},\quad C=\{U\le a/2\},
\]
and set $u=e^{-\alpha a/2}$ and $v=e^{-\alpha a}$. By \eqref{eq:conv-gau-mat}, $W=e^{-\alpha U}$ on $I$ and $W\le v$ on $O$. Hence $(U-U')(W-W')\le0$ on $I\times I$ and on the cross regions $I\times O$ and $O\times I$.

To control the only potentially positive part $O\times O$, let $p_O=\bbP(O)$, $p_C=\bbP(C)$, and $Q_O=\bbE[U\mathds1_O]$. Since $\bbE U\le a/4$,
\[
p_O\le\frac14, \quad p_C\ge\frac12.
\]
On $O\times O$,
\[
(U-U')(W-W')\le v(U+U').
\]
On the other hand, on $C\times O$ one has
\[
(U-U')(W-W')\le-\frac{u-v}{2}U',
\]
and symmetrically on $O\times C$. Using $2\operatorname{Cov}(U,W)=\bbE[(U-U')(W-W')]$ and discarding the remaining nonpositive pieces gives
\[
\operatorname{Cov}(U,W) \le\left(vp_O-\frac{u-v}{2}p_C\right)Q_O \le\frac{2v-u}{4}Q_O\le0,
\]
since $\alpha a\ge2\log2$ implies $u\ge2v$. Consequently,
\[
\frac{\int_{\R^d} g(y)(W+\widehat h_\alpha)\rho(dy)} {\int_{\R^d}(W+\widehat h_\alpha)\rho(dy)} \le \int_{\R^d} g(y)\,\rho(dy).
\]
Jensen's inequality for the convex function $g$ proves \eqref{eq:w-conv-ene}.
\end{proof}

\begin{theorem}[Post-entry trapping by a convex energy]\label{thm:conv-trap}
Under the standing local-Lipschitz assumption on the objective, work in the affine coordinates of Lemma \ref{lem:conv-gau}, and let $g,a$ be chosen with a number $\theta\in(0,1)$ satisfying
\begin{equation}\label{eq:trap-lam-cond}
\lambda(1-\theta)>\sigma^2(1+\theta).
\end{equation}
Assume $\alpha a\ge2\log2$ and $\widehat h_\alpha>0$. For the transformed particle system define
\[
H_t^N=\frac1N\sum_{i=1}^Ng(Y_t^i), \quad A_t^N=\frac1N\sum_{i=1}^N|Y_t^i-\widetilde m_t^N|^2,
\]
where $\widetilde m_t^N$ is the transformed regularized consensus point. Set
\[
c_0:=\frac{\lambda(1-\theta)-\sigma^2(1+\theta)}{2}>0.
\]
For every deterministic $T\ge0$ and every $0<b\le a/4$, there is an event $\mathcal G_b^N(T)$ such that
\begin{equation}\label{eq:conv-trap-prob}
\bbP\bigl(\mathcal G_b^N(T)\mid\calF_T\bigr)
\ge\left(1-\frac{H_T^N}{b}\right)_+,
\end{equation}
and on $\mathcal G_b^N(T)$,
\begin{equation}\label{eq:conv-trap-concl}
H_t^N<b\quad\text{for every }t\ge T, \quad \int_T^\infty A_t^N\,dt<\infty,
\end{equation}
all particles converge to a common finite random vector $Y_\infty^N$, and
\begin{equation}\label{eq:conv-trap-loc}
g(Y_\infty^N)\le b, \quad |Y_\infty^N|\le\sqrt{\frac{2b}{1-\theta}}.
\end{equation}
In particular, the trapping sublevel $\{g<b\}$ is independent of $L_\alpha$ and does not shrink as $\alpha$ increases.
\end{theorem}

\begin{proof}
As long as $H_t^N\le a/4$, Lemma \ref{lem:w-conv-ener} gives $g(\widetilde m_t^N)\le H_t^N$. By strong convexity,
\[
\nabla g(Y_t^i)\cdot(Y_t^i-\widetilde m_t^N) \ge g(Y_t^i)-g(\widetilde m_t^N) +\frac{1-\theta}{2}|Y_t^i-\widetilde m_t^N|^2.
\]
Averaging in $i$ and applying It\^o's formula, together with $\nabla^2g\preceq(1+\theta)\bbI_d$, hence yields
\begin{equation}\label{eq:conv-H-ineq}
dH_t^N\le-c_0A_t^Ndt+d\mathcal M_t, \quad \text{where } d\mathcal M_t=\frac\sigma N\sum_i \nabla g(Y_t^i)\cdot(Y_t^i-\widetilde m_t^N)dW_t^i.
\end{equation}
Let $\tau_b=\inf\{t\ge T:H_t^N\ge b\}$. The stopped process $H_{t\wedge\tau_b}^N$ is a nonnegative local supermartingale and hence a supermartingale. Thus, setting
\begin{equation*}
    \calG_b^N(T) := \{\tau_b = \infty\},
\end{equation*}
the conditional maximal inequality proves \eqref{eq:conv-trap-prob}. Moreover, for every finite $S>T$, optional stopping in \eqref{eq:conv-H-ineq} gives
\[
c_0\bbE\!\left[\int_T^{S\wedge\tau_b}A_t^Ndt\,\middle|\,\calF_T\right] \le H_T^N.
\]
Letting $S\to\infty$ and using monotone convergence shows that $\int_T^\infty A_t^Ndt<\infty$ almost surely on the event $\{\tau_b=\infty\}$.

It remains to verify convergence. Write
\[
M_t^N=\frac1N\sum_iY_t^i, \quad V_t^N=\frac1N\sum_i|Y_t^i-M_t^N|^2.
\]
Then
\[
A_t^N=V_t^N+|M_t^N-\widetilde m_t^N|^2.
\]
Set $\widetilde w_\alpha=e^{-\alpha\widetilde F}$. On the event $\{\tau_b=\infty\}$, the bound $H_t^N<b$ and $g(y)\ge(1-\theta)|y|^2/2$ imply
\[
|Y_t^i|\le R_{N,b}:=\sqrt{\frac{2Nb}{1-\theta}} \quad (i=1,\ldots,N,\ t\ge T).
\]
Hence $M_t^N$ lies in the same ball. By this standing local-Lipschitz assumption, $\widetilde w_\alpha$ has a finite Lipschitz constant $\widetilde L_{\alpha,N,b}$ on $\overline B_{R_{N,b}}$. Centeredness then gives
\[
|\widetilde m_t^N-M_t^N| \le\frac{\widetilde L_{\alpha,N,b}}{\widehat h_\alpha}V_t^N.
\]
Indeed, subtract $\widetilde w_\alpha(M_t^N)$ from the numerator of the centered weighted mean. Since $N$, the parameters, and $b$ are fixed and $\widehat h_\alpha>0$, \eqref{eq:conv-trap-concl} implies
\[
\int_T^\infty|\widetilde m_t^N-M_t^N|dt<\infty.
\]
Moreover,
\[
dM_t^N=-\lambda(M_t^N-\widetilde m_t^N)dt +\frac\sigma N\sum_i(Y_t^i-\widetilde m_t^N)dW_t^i,
\]
and the trace of the quadratic variation of the martingale is $\sigma^2N^{-1}A_t^Ndt$. Hence $M_t^N$ converges to a finite random vector. For $R_t^i=Y_t^i-M_t^N$, we find
\[
dR_t^i=-\lambda R_t^idt+dK_t^i,
\]
where $K^i$ is a continuous martingale whose quadratic variation is bounded by $C_N\sigma^2A_t^Ndt$. Thus $K_t^i$ converges. The variation-of-constants formula, or an integration by parts against the convergent martingale, then gives $R_t^i\to0$. Hence, every particle converges to the same $Y_\infty^N$. Finally $H_t^N<b$ and continuity of $g$ imply $g(Y_\infty^N)\le b$, while \eqref{eq:conv-gau-hes} gives $g(y)\ge(1-\theta)|y|^2/2$, proving \eqref{eq:conv-trap-loc}.
\end{proof}

\begin{corollary} \label{cor: convex finite transfer}
Use the notation of Theorem \ref{thm:conv-trap}. Suppose that for a deterministic time $T$ the transformed mean-field solution satisfies $\bbE|\bar Y_T|^2\le2\varepsilon$. Let
\[
\mathfrak a_N(T)^2:=\frac1N\sum_{i=1}^N \bbE|Y_T^i-\bar Y_T^i|^2,
\]
where the nonlinear copies are synchronously coupled. If $0<b\le a/4$, then
\begin{equation}\label{eq:conv-trans-prob}
\bbP\bigl(\mathcal G_b^N(T)\bigr) \ge1-\frac{1+\theta}{2b} \left(\mathfrak a_N(T)+\sqrt{2\varepsilon}\right)^2.
\end{equation}
Consequently, whenever Proposition \ref{prop: finite horizon} applies, $\mathfrak a_N(T)\to0$ for fixed parameters and fixed $T$.
\end{corollary}

\begin{proof}
The upper Hessian bound and $g(0)=\nabla g(0)=0$ give $g(y)\le(1+\theta)|y|^2/2$. By Minkowski's inequality in $L^2(\Omega\times\{1,\ldots,N\})$,
\[
\left(\bbE\frac1N\sum_i|Y_T^i|^2\right)^{1/2} \le\mathfrak a_N(T)+\sqrt{\bbE|\bar Y_T|^2} \le\mathfrak a_N(T)+\sqrt{2\varepsilon}.
\]
Hence
\[
\bbE H_T^N\le\frac{1+\theta}{2} \left(\mathfrak a_N(T)+\sqrt{2\varepsilon}\right)^2.
\]
Integrating \eqref{eq:conv-trap-prob} and using $(1-x)_+\ge1-x$ proves \eqref{eq:conv-trans-prob}.
\end{proof}

\begin{proof}[Proof of Theorem \ref{thm: finite global optimization}] Let $Q=\nabla^2f(x^*)$ and make the affine change $y=Q^{1/2}(x-x^*)$. The scalar-noise CBO system is covariant under this transformation: the transformed consensus point is $Q^{1/2}(m_f^h-x^*)$, and the transformed objective is $\widetilde f(y)=f(x^*+Q^{-1/2}y)$. Full support and finite moments are preserved, and the weak-landscape constants change only by fixed norm-equivalence factors. Thus it suffices to work in the transformed coordinates. From now on, set $\calV(t):=\frac12\bbE|\bar Y_t|^2$ for the transformed mean-field process.

Since $\lambda>\sigma^2$, we choose
\[
0<\theta<\frac{\lambda-\sigma^2}{\lambda+\sigma^2}.
\]
Lemma \ref{lem:conv-gau} gives $g$ and $a$, and \eqref{eq:trap-lam-cond} holds. Set $\mu=1-\theta$ and $L=1+\theta$. If the prescribed radius in the original coordinates is $\delta$, put $\delta_y:=\sqrt{\lambda_{\min}(Q)}\,\delta$ and choose
\[
0<b<\min\left\{\frac a4,\frac{\mu\delta_y^2}{2}\right\}.
\]
Then \eqref{eq:conv-trap-loc} implies $|X_\infty^N-x^*|<\delta$ on $\mathcal G_b^N(T_*)$, where $T_*$ is the entrance time chosen below.

Choose $\varepsilon>0$ so small that
\[
0<\varepsilon<\calV(0), \quad \frac{2L\varepsilon}{b}<\frac{\varepsilon_{\rm fail}}{4}.
\]
Use Proposition \ref{prop: full support entrance} to choose $r,q$ and then choose $\alpha$ large enough to satisfy both inequalities
\[
e^{-\alpha q}\le\frac{c_\vartheta c_{\rm mass}}{8\sqrt2}, \quad \alpha a\ge2\log2.
\]
After this $\alpha$ is fixed, choose $h(\alpha)>0$ so small that the second inequality in \eqref{eq:entr-a-h-cho} holds. Proposition \ref{prop: full support entrance} then gives a deterministic $T_*<\infty$ for these same fixed parameters, with $\bbE|\bar Y_{T_*}|^2\le2\varepsilon$.

For these fixed parameters, Proposition \ref{prop: finite horizon} gives $\mathfrak a_N(T_*)\to0$. Hence, for all sufficiently large $N$,
\[
\frac{L}{2b} \left(\mathfrak a_N(T_*)+\sqrt{2\varepsilon}\right)^2 <\varepsilon_{\rm fail}.
\]
Corollary \ref{cor: convex finite transfer} and Theorem \ref{thm:conv-trap} then show that, with probability at least $1-\varepsilon_{\rm fail}$, all particles converge to a common finite limit whose original-coordinate distance from $x^*$ is less than $\delta$. The coefficients remain unchanged at $T_*$; this time appears only in the proof. This completes the proof.
\end{proof}

\begin{remark}
Theorem \ref{thm: finite global optimization} removes the conditional local-mass and entrance/trapping certificates, but it does so by adding explicit assumptions on the input data. Full support excludes compactly supported initial laws, and Assumption \ref{assum: local nondegenerate} excludes nonsmooth or degenerate minima such as $F(x)=|x|$ or $F(x)=|x|^4$ at the optimizer. Neither condition is a global convexity assumption. In particular, multimodal objectives with a nondegenerate global minimum remain admissible.
\end{remark}


%
%
%
%
%
%
%
%
%
%

%
%
%
%
%
%
%
%
%
%

%
%
%
%
%
%
%
%
%
%

%
%
%
%
%
%
%
%
%
%

\section*{Acknowledgments}
This work is supported by NRF grant no. RS-2024-00406821.

%
%
%
%
%
%
%
%
%
%

\appendix

%
%
%
%
%
%
%
%
%
%

\section[Exchangeability of the solutions]{Exchangeability of the solutions}\label{app_exch}
In this appendix, we verify that the finite-particle system \eqref{I: eq: main} preserves exchangeability over time.

\begin{lemma}
Let $\{X_t^i\}_{i=1}^N$ be the unique strong solution to the particle system \eqref{I: eq: main}
with i.i.d.\ initial data $\{X_0^i\}_{i=1}^N$ and independent Wiener processes $\{W_t^i\}_{i=1}^N$.
Then the family $\{X_t^i\}_{i=1}^N$ is exchangeable for all $t\ge0$.
\end{lemma}

\begin{proof}
    Let $\pi\in \text{Perm}(N)$ be any   permutation of $\{1, 2, \ldots, N\}$. We define a new process $\{\mathbf{\hat{X}}_t\}_{t\ge 0}$ by permuting the components of $\bfX_t$, that is, 
    \[
    \hat{X}_t^i:=X_t^{\pi(i)}, \quad i=1, 2, \ldots, N.
    \]

    First, we rewrite the dynamics  \eqref{I: eq: main} for each particle using the empirical measure 
    \[
    \mu_{\mathbf{X}_t^N}=\frac{1}{N}\sum_{j=1}^N\delta_{X_t^j}
    \] 
    as
\[
        dX_t^i=F\left(X_t^i, \mu_{\mathbf{X}_t^N}\right)dt+G\left(X_t^i, \mu_{\mathbf{X}_t^N}\right)dW_t^i,
\]
    where the drift and diffusion coefficients $F$ and $G$ are given, for any $(x, \mu)\in\bbR^d\times \calP(\bbR^d)$, by
\[
        F(x, \mu)=-\lambda (x-m_f^h[\mu]), \quad G(x, \mu)=\sigma(x-m_f^h[\mu]).
\]
    Since the empirical measure is invariant under permutation, we have
\[
        \mu_{\mathbf{\hat{X}}_t^N}=\frac{1}{N}\sum_{i=1}^N\delta_{\hat{X}_t^i}=\frac{1}{N}\sum_{i=1}^N\delta_{X_t^{\pi(i)}}=\mu_{\mathbf{X}_t^N}.
\]
    This invariance implies that each component $\hat{X}_t^i=X_t^{\pi(i)}$ satisfies the dynamics
\[
        d\hat{X}_t^i=dX_t^{\pi(i)} =F\left(X_t^{\pi(i)}, \mu_{\mathbf{X}_t^N}\right)dt+G\left(X_t^{\pi(i)}, \mu_{\mathbf{X}_t^N}\right)dW_t^{\pi(i)} =F\left(\hat{X}_t^i, \mu_{\mathbf{\hat{X}}_t^N}\right)dt+G\left(\hat{X}_t^i, \mu_{\mathbf{\hat{X}}_t^N}\right)d\hat{W}_t^i,
\]
    where we introduced the permuted Wiener processes $\hat{W}_t^i:=W_t^{\pi(i)}$.

Next, we check the initial and noise distributions:
    \begin{enumerate}
        \item[(i)] Since the initial data $\{X_0^i\}_{i=1}^N$ are i.i.d., the permuted vector $\widehat{\mathbf{X}}_0 = (X_0^{\pi(1)},\dots,X_0^{\pi(N)})^\top$ is identically distributed to $\mathbf{X}_0$.
        \item[(ii)] 
        Since $\{W_t^i\}_{i=1}^N$ are independent one-dimensional Wiener processes, the permuted family $\{\hat{W}_t^i\}_{i=1}^N$ also consists of independent one-dimensional Wiener processes.  Hence, the vector processes $\{\hat{\bfW}_t\}_{t\ge 0}$  and $\{\mathbf{W}_t\}_{t\ge0}$ have the same law.  
    \end{enumerate}
    
Hence, $\{\mathbf{X}_t\}_{t\ge0}$ and $\{\widehat{\mathbf{X}}_t\}_{t\ge0}$ satisfy the same SDE, driven by initial data and Wiener processes that are equal in law.  
By pathwise uniqueness of strong solutions (Theorem \ref{theorem: uniqueness}),  
the two processes coincide in law:
\[
        \textnormal{Law}(\{\hat{\bfX}_t\}_{t\ge0})=\textnormal{Law}\left(\{\bfX_t\}_{t\ge0}\right).
\]
Since this holds for any permutation $\pi\in\textnormal{Perm}(N)$, we conclude that the family $\{X_t^i\}_{i=1}^N$ is exchangeable.
\end{proof}

%
%
%
%
%
%
%
%
%
%
\section{Well-posedness of the modified CBO system}\label{sec: exist}

In this appendix, we provide a detailed proof of Theorem \ref{theorem: uniqueness}. Throughout this section, we fix $N \in \N$ and denote 
\[
\mathbf{X}_t := (X_t^1, X_t^2, \ldots, X_t^N)^{\top}\in\mathbb{R}^{Nd}.
\]
The system \eqref{I: eq: main} can then be rewritten in the vector form
\[
    d\mathbf{X}_t=-\lambda \mathbf{F}_N(\mathbf{X}_t)\,dt
    +\sigma \mathbf{G}_N(\mathbf{X}_t)\,d\mathbf{W}_t,
\]
where $\mathbf{W}_t=(W_t^1,\ldots,W_t^N)^\top\in\mathbb R^N$ is an $N$-dimensional standard Wiener process. The drift term
\[
    \mathbf{F}_N(\mathbf{X})=(F_N^{1}(\mathbf{X}), F_N^{2}(\mathbf{X}),\ldots,F_N^{N}(\mathbf{X}))^{\top}\in \mathbb{R}^{Nd}
\]
has components
\[
    F_N^i(\mathbf{X}) =\frac{\sum_{j=1}^N(X^i-X^j)\psi_h(X^j)}{\sum_{j=1}^N\psi_h(X^j)}\in\R^d.
\]
The diffusion coefficient $\mathbf G_N(\mathbf X)\in\mathbb R^{Nd\times N}$ is defined columnwise: its $i$th column has the vector $F_N^i(\mathbf X)$ in the $i$th $d$-dimensional block and zeros in all other blocks. In particular,
\[
    \|\mathbf G_N(\mathbf X)\|_{\mathrm F}^2
    =\sum_{i=1}^N |F_N^i(\mathbf X)|^2.
\]
Under the local Lipschitz regularity assumption on $f$ and the positivity of $h$, one verifies that each $F_N^i$ is locally Lipschitz and has at most linear growth, hence $\mathbf F_N$ and $\mathbf G_N$ inherit these properties.
\begin{lemma}\label{lem: local lipshitz}
    Let $N\in\mathbb{N}$, $k, \alpha>0$ be arbitrary and assume $h$ is strictly positive function. Then for any $\mathbf{X}, \tilde{\mathbf{X}}\in\mathbb{R}^{Nd}$ and all $i=1, 2, \ldots, N$ with $|\mathbf{X}|, |\tilde{\mathbf{X}}|\leq k$ it holds,
    \begin{align*}
        |F_N^i(\mathbf{X})-F_N^i(\tilde{\mathbf{X}})|&\leq |X^i-\tilde{X}^i|+\left(\frac{\Lambdaa}{\sqrt{N}}+\frac{c_k\left(1+\Lambdaa\right)}{Nh(\alpha)}\sqrt{N|\tilde{X}^i|^2+|\tilde{\mathbf{X}}|^2}\right)|\mathbf{X}-\tilde{\mathbf{X}}|,\quad
        |F_N^i(\mathbf{X})| \leq |X^i|+\frac{\Lambdaa}{\sqrt{N}}|\mathbf{X}|,
    \end{align*}
    where $c_k:=\sqrt{2}\,\operatorname{Lip}(\omega_f^\alpha;B(0,k))<\infty$.
\end{lemma}
\begin{proof}
    Let $\mathbf{X}, \tilde{\mathbf{X}}\in\mathbb{R}^{Nd}$ with $|\mathbf{X}|, |\tilde{\mathbf{X}}|\leq k$ for some $k\geq 0$ and $i\in\{1, 2, \ldots, N\}$ be arbitrary. We write 
\[
        F_N^i(\mathbf{X})-F_N^i(\tilde{\mathbf{X}})=\frac{\sum_{j\neq i}(X^i-X^j)\psi_h(X^j)}{\sum_{j=1}^N\psi_h(X^j)}-\frac{\sum_{j\neq i}(\tilde{X}^i-\tilde{X}^j)\psi_h(\tilde{X}^j)}{\sum_{j=1}^N\psi_h(\tilde{X}^j)}=: I^i + II^i + III^i,
\]
    where  
    \begin{align*}
       |I^i| &=\lt|\frac{\sum_{j\neq i}(X^i-\tilde{X}^i+\tilde{X}^j-X^j)\psi_h(X^j)}{\sum_{j=1}^N\psi_h(X^j)}\rt| \leq |X^i-\tilde{X}^i|+\frac{\Lambdaa}{\sqrt{N}}|\bfX-\tilde{\bfX}|, \\
       |II^i|&=\lt|\frac{\sum_{j\neq i}(\tilde{X}^i-\tilde{X}^j)(\psi_h(X^j)-\psi_h(\tilde{X}^j))}{\sum_{j=1}^N\psi_h(X^j)}\rt|\leq \frac{c_k}{Nh(\alpha)}\sqrt{N|\tilde{X}^i|^2+|\tilde{\mathbf{X}}|^2}|\mathbf{X}-\tilde{\mathbf{X}}|, \\
    | III^i |&=\lt|\sum_{j\neq i}(\tilde{X}^i-\tilde{X}^j)\psi_h(\tilde{X}^j)\frac{\sum_{j=1}^N(\psi_h(\tilde{X}^j)-\psi_h(X^j))}{\sum_{j=1}^N\psi_h(X^j)\sum_{j=1}^N\psi_h(\tilde{X}^j)}\rt|\cr
    &\leq \frac{c_k\Lambdaa}{Nh(\alpha)}\sqrt{N|\tilde{X}^i|^2+|\tilde{\mathbf{X}}|^2}|\mathbf{X}-\tilde{\mathbf{X}}|.
    \end{align*}
    Summing all these together yields the required estimate. As for the estimate of $|F_N^i(\bfX)|$, we obtain 
\bq\label{eq: F_N^i}
        |F_N^i(\bfX)|\leq |X^i|+\frac{\Lambdaa}{\sqrt{N}}|\bfX|.
\eq
This completes the proof.
\end{proof}

Lemma \ref{lem: local lipshitz} implies that both $\mathbf{F}_N$ and $\mathbf{G}_N$ satisfy the local Lipschitz and linear growth conditions. We can invoke standard SDE theory to establish global existence and uniqueness, as stated in Theorem \ref{theorem: uniqueness}.

\begin{proof}[Proof of Theorem \ref{theorem: uniqueness}]
Applying It\^{o}'s formula to $|\mathbf{X}_t|^2$, we obtain
\[
            d|\mathbf{X}_t|^2=-2\lambda \langle \mathbf{X}_t, \mathbf{F}_N(\mathbf{X}_t)\rangle dt+2\sigma \langle \mathbf{X}_t, \mathbf{G}_N(\mathbf{X}_t)d\mathbf{W}_t\rangle  +\sigma^2\|\mathbf{G}_N(\mathbf{X}_t)\|_\textnormal{F}^2\,dt.
  \]
Using the estimate
\[
        -\langle X_t^i, F_N^i(\mathbf{X}_t)\rangle = -\left\langle X_t^i, \frac{\sum_{j=1}^N(X_t^i-X_t^j)\psi_h(X_t^j)}{\sum_{j=1}^N\psi_h(X_t^j)} \right\rangle\leq -|X_t^i|^2+\frac{\Lambdaa}{\sqrt{N}}|X_t^i||\mathbf{X}_t|,
\]
together with \eqref{eq: F_N^i} and H\"older's inequality, we deduce
    \begin{equation}\label{eq: bfX_t esti}
        \begin{split}
            -2\lambda \langle \mathbf{X}_t, \mathbf{F}_N(\mathbf{X}_t)\rangle 
        +\sigma^2\|\mathbf G_N(\bfX_t)\|_\textnormal{F}^2&=\sum_{i=1}^N\left[-2\lambda \langle X_t^i, F_N^i(\mathbf{X}_t)\rangle+\sigma^2|F_N^i(\bfX_t)|^2 \right]\\
        &\leq -2\lambda \sum_{i=1}^N\left[|X_t^i|^2-\frac{\Lambdaa}{\sqrt{N}}|X_t^i||\bfX_t|\right]+2\sigma^2(1+\Lambdaa )|\bfX_t|^2\\
        &\leq  \lt(-2\lambda+2\lambda \Lambdaa+2\sigma^2(1+\Lambdaa)\rt) |\bfX_t|^2.
        \end{split}
    \end{equation}
   
    The local Lipschitz continuity and linear growth of $\bfF_N$ and $\mathbf G_N$, established in Lemma \ref{lem: local lipshitz}, allow us to apply the standard theory of stochastic differential equations (see, e.g. \cite[Theorem 3.1]{Durrett1996}), which guarantees the existence and uniqueness of a local-in-time strong solution $\mathbf{X}_t$ to \eqref{I: eq: main}.
    
    To show that the solution is in fact global, we derive a uniform second-moment bound. It follows from \eqref{eq: bfX_t esti} that 
\[
        \frac{d}{dt}\bbE|\bfX_t|^2\leq b_\alpha\bbE |\bfX_t|^2, \quad b_\alpha:= -2\lambda+2\lambda \Lambdaa+2\sigma^2(1+\Lambdaa).
\]
    Applying Gr\"onwall's inequality then yields:
\[
        \mathbb{E}|\mathbf{X}_t|^2\leq \mathbb{E}|\mathbf{X}_0|^2e^{b_\alpha t} \quad \text{for all }t\geq 0.
\]
This bound ensures that $\mathbb{E}|\mathbf{X}_t|^2$ remains finite for every finite time $t$, 
so no explosion can occur before $t=\infty$. Hence, the local strong solution extends globally in time.

The estimate above proves non-explosion for every fixed $N$. Since $|\mathbf X_t|^2=\sum_{i=1}^N|X_t^i|^2$, the unnormalized quantity $\bbE|\mathbf X_t|^2$ is generally of order $N$ and is not uniform in $N$. If the initial data satisfy
\[
\sup_{N\ge1}\frac1N\bbE|\mathbf X_0|^2<\infty,
\]
then the normalized estimate
\[
\frac1N\bbE|\mathbf X_t|^2\le \frac1N\bbE|\mathbf X_0|^2e^{b_\alpha t}
\]
is uniform in $N$ on every finite time interval.

\end{proof}

%
%
%
%
%

%
%
%
%
%
%
%
%
%
%

\section{Uniform-in-time stability of particle system} \label{app: stability}

In this appendix, we provide the statements and proofs of the concentration and stability estimates for the particle system \eqref{I: eq: main}. 

The first result below adapts the concentration argument of \cite[Lemma 5.9 and Remark 5.11]{GKHV25} to the regularized scalar-noise system and general $p$-moments. We give the proof to make the moment conditions and the dependence on the regularized weight explicit. 
For the reader's convenience, we restate the concentration estimate here and include a sketch of the proof.

\begin{lemma} \label{lem: concentration estimate}
    Let $p,q\geq 2$ and assume $X_{\rm in} \in L^{pq}(\Omega)$. For the particle trajectories, we suppose that $X_0^i$ are i.i.d. with $X_0^i \sim X_{\rm in}$ for all $i=1,\dots, N$. Denoting  
\[
c_{\textnormal{con},p}= 2\Lambdaa^{1-\frac{2}{p}}\bflambda_{p, \alpha, \sigma}+2(p-1)\sigma^2=2(p-1)(1+\Lambda_\alpha)\sigma^2, 
\]
further assume
\[
\lambda > \bflambda_{p,q, \alpha, \sigma}':=\bflambda_{pq, \alpha, \sigma} + c_{\textnormal{con},p}.
\]
    Then for all small enough $\kappa>0$ and every $A>0$, the following concentration estimate holds:
\bq \label{eq: concentration 2} 
\bbP\left[\sup_{t\geq 0}e^{\kappa t}V_{p}(\mu_t^N)\geq \bbE\left[V_{p}(\mu_0^N)\right]+A\right]\leq \frac{C}{A^{q}N^\frac{q}{2}}
\eq
for some constant $C>0$ independent of $A$ and $N$. In particular, \eqref{eq: concentration 2} holds for all fixed $0 < \kappa < p(\lambda -( c_{\textnormal{con},p} \vee \bflambda_{pq,\alpha,\sigma}))$.
\end{lemma}

\begin{proof}
We give the details that are needed for the random empirical centering. The case $N=1$ is immediate, so assume $N\ge2$. Put
\[
K_t^i=X_t^i-\calM(\mu_t^N),\quad A_t^i=X_t^i-m_f^h[\mu_t^N],\quad U_t^i=|K_t^i|^{p-2}K_t^i,\quad \overline U_t=N^{-1}\sum_iU_t^i.
\]
The quadratic variation of the centered particle is
\[
\frac{d}{dt}\operatorname{tr}\langle K^i\rangle_t =\sigma^2\left[(1-2/N)|A_t^i|^2+N^{-2}\sum_j|A_t^j|^2\right].
\]
The Hessian bound for $|\cdot|^p$, H\"older's inequality, and
\[
\left(N^{-1}\sum_i|A_t^i|^p\right)^{2/p} \le 2(1+\Lambda_\alpha^{2/p})V_p(\mu_t^N)^{2/p} \le 2(1+\Lambda_\alpha)V_p(\mu_t^N)^{2/p}
\]
give
\begin{align*}
dV_p(\mu_t^N)  \le-p\left[\lambda-(p-1)(1+\Lambda_\alpha)(1+N^{-1})\sigma^2\right]  V_p(\mu_t^N)dt+dM_t \le-p(\lambda-c_{\mathrm{con},p})V_p(\mu_t^N)dt+dM_t,
\end{align*}
where the martingale coefficient is exactly
\[
dM_t=\frac{p\sigma}{N}\sum_i  (U_t^i-\overline U_t)\cdot A_t^i\,dW_t^i.
\]
Here $c_{\mathrm{con},p}=2(p-1)(1+\Lambda_\alpha)\sigma^2$ agrees with the constant in the statement. All It\^o computations are first localized; for $p=2$ the terms multiplied by $p-2$ are omitted.

Fix $0<\kappa<p(\lambda-(c_{\mathrm{con},p}\vee \bflambda_{pq,\alpha,\sigma}))$. With $E_t=e^{\kappa t}V_p(\mu_t^N)$ and $\widetilde M_t=\int_0^t e^{\kappa s}dM_s$, one has $E_t\le E_0+\widetilde M_t$. We next justify concentration of $E_0$; its summands are not independent since they are centered by the empirical mean. Let $H(x_1,\ldots,x_N)=N^{-1}\sum_j|x_j-\overline x|^p$, and replace $x_i$ by an independent copy $x_i'$ to obtain $H^{(i)}$. The inequality $||u|^p-|v|^p|\le p(|u|+|v|)^{p-1}|u-v|$ gives
\[
|H-H^{(i)}|\le\frac{C_p}{N}|x_i-x_i'| \left(|x_i|^{p-1}+|x_i'|^{p-1} +N^{-1}\sum_j|x_j|^{p-1}\right).
\]
Thus $\bbE |H-H^{(i)}|^q\le C N^{-q}\bbE |X_{\rm in}|^{pq}$. Apply the discrete martingale moment inequality to the Doob martingale obtained by revealing $X_0^1,\ldots,X_0^N$ one at a time. Its $i$th difference is a conditional expectation of $H-H^{(i)}$; hence
\[
\bbE |E_0-\bbE E_0|^q \le C_q N^{q/2-1}\sum_i\bbE |H-H^{(i)}|^q \le C N^{-q/2}.
\]

Let $\gamma_{pq}=\lambda-\bflambda_{pq,\alpha,\sigma}>0$. Lemma \ref{lem: diff Lp} and Jensen's inequality imply
\[
\max_i\left(\|K_t^i\|_{L^{pq}}+\|A_t^i\|_{L^{pq}}\right)
\le C e^{-\gamma_{pq}t}.
\]
Writing $H_t^i=(U_t^i-\overline U_t)\cdot A_t^i$, H\"older and Jensen hence yield
\[
\max_i\bbE |H_t^i|^q\le C e^{-pq\gamma_{pq}t}.
\]
The independent driving Brownian motions have zero mutual quadratic covariations, even though their integrands are interacting. Consequently, Burkholder--Davis--Gundy, Jensen, and Minkowski give
\begin{align*}
\bbE \sup_{t\ge0}|\widetilde M_t|^q &\le \frac{C}{N^{q/2}}\frac1N\sum_i \bbE \left(\int_0^\infty e^{2\kappa t}|H_t^i|^2dt\right)^{q/2} \le\frac{C}{N^{q/2}} \left(\int_0^\infty e^{2(\kappa-p\gamma_{pq})t}dt\right)^{q/2} \le\frac{C}{N^{q/2}(p\gamma_{pq}-\kappa)^{q/2}}.
\end{align*}
Finally,
\[
\bbP \left(\sup_{t\ge0}E_t\ge\bbE E_0+A\right) \le\frac{2^q}{A^q} \left(\bbE |E_0-\bbE E_0|^q+ \bbE \sup_{t\ge0}|\widetilde M_t|^q\right),
\]
which proves \eqref{eq: concentration 2}. Constants may depend on the chosen decay exponent $\kappa$, as well as on the fixed model parameters.
\end{proof}

We now proceed to prove the stability estimates for $p\ge 2$ in Theorem \ref{thm: stability}.  The proof relies on Lemma \ref{lem: bounding calB}, which allows us to decompose the difference between the consensus points and the empirical means of the particles.  This decomposition introduces terms involving the empirical $p$-th variances, which are subsequently controlled through the concentration estimate obtained in Lemma \ref{lem: concentration estimate}.  The latter ensures that these variance terms remain uniformly bounded with high probability.

\begin{theorem}\label{thm: stability}
    Let $p\ge 2$, $q>2$, and assume that $\lambda >\bflambda_{p,q, \alpha, \sigma}'$. Suppose that the objective function $f$ satisfies Assumption \ref{assum: omega}. Consider two copies $\{X_t^i\}_{i=1}^N$ and $\{Y_t^i\}_{i=1}^N$ of the particle system \eqref{I: eq: main}, driven by the same one-dimensional Wiener processes $\{W_t^i\}_{i=1}^N$, but with different i.i.d. initial data: for each $1\leq i\leq N$, assume $X_0^i\sim X_\textnormal{in}$ and $Y_0^i\sim Y_\textnormal{in}$, where  $X_\textnormal{in}, Y_\textnormal{in}\in L^{pq}(\Omega)$. Then the following stability estimates hold:
    \begin{enumerate}
        \item if $p=2$, 
        \begin{equation*}
            \frac{1}{N}\sum_{i=1}^N\bbE|X_t^i-Y_t^i|^{2}\leq \frac{C}{N}\sum_{i=1}^N\bbE|X_0^i-Y_0^i|^{2}+\frac{C}{N^\frac{q-2}{2}} ,
        \end{equation*}
        \item  if $p>2$, 
        \begin{equation*}
            \frac{1}{N}\sum_{i=1}^N\bbE|X_t^i-Y_t^i|^{p}\leq \frac{C}{N}\sum_{i=1}^N\bbE|X_0^i-Y_0^i|^{p}+\frac{C}{N^{\frac{p}{2}\wedge \frac{q-2}{2}}}
        \end{equation*}
    \end{enumerate}
    holds, where $C>0$ depends only on $\lambda,\sigma,p,q$, $\Lambda_\alpha$, $K_\alpha$, $X_{\rm in}$, $Y_{\rm in}$ and fixed admissible decay rates chosen in the proof.
\end{theorem}
\begin{proof}
    We adopt the same notations as in the proof of Theorem \ref{thm: propagation of chaos}, replacing $\bar X_t^i$ there by $Y_t^i$ :
\[
        Z_t^i=X_t^i-Y_t^i,\quad \nu_t^N := \frac{1}{N}\sum_{i=1}^N \delta_{Y_t^i}, \quad \Delta_t^m=m_f^h[\mu_t^N]-m_f^h[\nu_t^N], \quad \Delta_t^\calM=\calM(\mu_t^N)-\calM(\nu_t^N).
\]
We also set
\[
        \calG_{t, p}:=\frac{1}{N}\sum_{i=1}^N|Z_t^i|^{p}, \quad \calD_{t, p}:=\frac{1}{N}\sum_{i=1}^N|Z_t^i-\dcal|^{p}, \quad \calO_{t, p}:=|\dcal|^{p}.
\]
The overall strategy is similar to that of Theorem \ref{thm: propagation of chaos}. Our goal is to control  $\calG_{t, p}$ by estimating $\calD_{t, p}$ and $\calO_{t, p}$. The first crucial step is to derive a bound for $\bbE|\dcal-\dm|^p$. From Lemma \ref{lem: bounding calB}, we recall
\[
    |\dcal-\dm|^{p}\leq \frac{1}{2}\left(\frac{2L_{\omega_f^\alpha}\Lambda_\alpha}{h(\alpha)}\right)^{p}\left(V_{p}(\mu_t^N)+V_{p}(\nu_t^N)\right)\calW_{p}^{p}(\mu_t^N, \nu_t^N).
\]
Using the concentration estimate from Lemma \ref{lem: concentration estimate}, we can proceed as in Lemma \ref{lem: deltas} and deduce that for some $0<\eta<\kappa$,
\begin{equation}\label{eq: dcal-dm}
        \bbE|\dcal-\dm|^{p}\leq  \left[\frac{C}{N^{\frac{q-2}{2}}}+C\bbE\calG_{t, p}\right]e^{-\eta t},
\end{equation}
where $C>0$ depends only on $\lambda,\sigma,p,q$, $\Lambda_\alpha$, $K_\alpha$, $X_{\rm in}$, $Y_{\rm in}$ and the chosen admissible decay rate. 

We now consider the two cases $p=2$ and $p>2$ separately.

\noindent $\bullet $ (The case $p=2$) Applying the It\^o formula to $\calG_{t, 2}$ leads to
\[
    d\calG_{t, 2}=-\left[\frac{2\lambda}{N}\sum_{i=1}^N|Z_t^i|^2-2\lambda \langle \dcal, \dm\rangle-\frac{\sigma^2}{N}\sum_{i=1}^N|Z_t^i-\dm|^2\right]dt+\frac{2\sigma}{N}\sum_{i=1}^N\langle Z_t^i, Z_t^i-\dm\rangle dW_t^i.
\]
By using the elementary identity
\[
    \frac{1}{N}\sum_{i=1}^N |z_i|^2 = \frac{1}{N}\sum_{i=1}^N |z_i - z^{\rm avg}|^2 + |z^{\rm avg}|^2, \quad z^{\rm avg} := \frac{1}{N}\sum_{i=1}^N z_i,
\]
and taking the expectation of both sides, we obtain for some $0<\widetilde\eta<\frac{\eta}{2}$ 
\begin{align*}%
    \frac{d}{dt}\bbE\calG_{t, 2}&= -\frac{2\lambda -\sigma^2}{N}\sum_{i=1}^N\bbE|Z_t^i-\dcal|^2-2\lambda \bbE\langle \dcal, \dcal -\dm\rangle+\sigma^2\bbE|\dcal-\dm|^2\\
    &\leq 2\lambda \left(\bbE\calG_{t, 2}\right)^\frac{1}{2}\left(\bbE|\dcal-\dm|^2\right)^\frac{1}{2}+\sigma^2\bbE|\dcal-\dm|^2\nonumber\\
    &\leq \lambda \bbE\calG_{t, 2}e^{-\widetilde\eta t}+(\lambda e^{\widetilde\eta t}+\sigma^2)\bbE|\dcal-\dm|^2.\nonumber
\end{align*}
Substituting the estimate \eqref{eq: dcal-dm} into the preceding inequality and using $0<\widetilde\eta<\eta/2$ yields
\[
    \frac{d}{dt}\bbE\calG_{t, 2}\leq C\bbE \calG_{t, 2}e^{-\widetilde\eta t}+\frac{C}{N^\frac{q-2}{2}}e^{-\widetilde\eta t}.
\]
Applying Gr\"onwall's inequality yields the desired estimate for the $p=2$ case:
\[
    \bbE\calG_{t, 2}\leq \left(\bbE\calG_{0, 2}+\frac{C}{N^{\frac{q-2}{2}}}\right)e^{C/\widetilde\eta}.
\]

\noindent $\bullet $ (The case $p>2$) In the same way as in \eqref{eq: relation}, we have
\begin{equation}\label{eq: relation2}
        \calG_{t, p} \leq 2^{p-1}(\calD_{t, p}+\calO_{t, p}) \leq (2^{2p-1}+2^{p-1})\calG_{t, p}.
\end{equation}

Applying It\^o's formula to $\calD_{t,p}$ and using Cauchy--Schwarz gives
\begin{align*}
    \frac{d}{dt}\bbE\calD_{t, p}&\leq-p\left({\lambda}-(p-1)\sigma^2\right) \bbE\calD_{t, p} +\frac{p(p-1)\sigma^2}{N}\sum_{i=1}^N\bbE\left[|Z_t^i-\dcal|^{p-2}|\dcal-\dm|^2\right]\\
    &\quad +\frac{p(p-1)\sigma^2}{2N^3}\sum_{i,j=1}^N \bbE\left[|Z_t^i-\dcal|^{p-2}|Z_t^j-\dm|^2\right].
\end{align*}
Since the leading term is strictly negative, the remaining ones are handled as in Lemma \ref{lem: barcalD}, using Lemmas \ref{lem: diff Lp} and \ref{lem: concentration estimate}, together with Young's inequality and \eqref{eq: dcal-dm}.  Thus, there exist a constant $C>0$, depending on $\lambda,\sigma,p,q$, $\Lambda_\alpha$, $K_\alpha$, $X_{\rm in}$, $Y_{\rm in}$ and the chosen admissible decay rates, and $\widehat\eta>0$ such that
\bq\label{eq: calD}
    \frac{d}{dt}\bbE\calD_{t, p}\leq C \bbE\calD_{t,p}e^{-\widehat\eta t}+C \bbE\calO_{t, p}e^{-\widehat\eta t}+\frac{C}{N^{\frac{p}{2}\wedge \frac{q-2}{2}}}e^{-\widehat\eta t} .
\eq

Next, applying It\^{o}'s formula to $\calO_{t, p}$ and applying the Cauchy-Schwarz inequality shows
\[
    \frac{d}{dt}\bbE\calO_{t, p} \leq  -p\lambda\bbE|\Delta_t^\mathcal{M}|^{p-2}\langle \Delta_t^\mathcal{M}, \Delta_t^\mathcal{M} - \Delta_t^m \rangle +\frac{p(p-1)\sigma^2}{2N^2}\bbE\left[|\Delta_t^\mathcal{M}|^{p-2}\sum_{i=1}^N |Z_t^i -\dm|^2\right].
\]
Then, by using a similar argument as in Lemma \ref{lem: baecalO}, we obtain
\begin{align}\label{eq: calO}
    \frac{d}{dt}\bbE\calO_{t, p}\leq C \bbE\calO_{t, p}e^{-\widehat\eta t}+C \bbE\calD_{t, p}e^{-\widehat\eta t}+\frac{C }{N^{\frac{p}{2}    \wedge \frac{q-2}{2}}}e^{-\widehat\eta t}
\end{align}
for some $C>0$, which depends on $\lambda,\sigma,p,q$, $\Lambda_\alpha$, $K_\alpha$, $X_{\rm in}$, $Y_{\rm in}$ and the fixed admissible decay rates.

Combining \eqref{eq: calD} and \eqref{eq: calO}, we have 
\[
    \frac{d}{dt}\left(\bbE\calD_{t, p}+\bbE\calO_{t, p}\right)\leq C \left(\bbE\calD_{t, p}+\bbE\calO_{t, p}\right)e^{-\widehat\eta t}+\frac{C }{N^{\frac{p}{2}\wedge\frac{q-2}{2}}}e^{-\widehat\eta t}
\]
for some positive constant $C$, depending on $\lambda,\sigma,p,q$, $\Lambda_\alpha$, $K_\alpha$, $X_{\rm in}$, $Y_{\rm in}$ and the fixed admissible decay rates.
Applying Gr\"onwall's inequality gives 
\[
    \bbE\calD_{t, p}+\bbE\calO_{t, p}\leq \left[\bbE\calD_{0, p}+\bbE\calO_{0, p}+\frac{C }{N^{\frac{p}{2}\wedge \frac{q-2}{2}}}\right]e^{C/\widehat\eta}.
\]
Finally, using \eqref{eq: relation2}, we deduce
\[
    \bbE\calG_{t, p}\leq C\left[\bbE\calG_{0, p}+\frac{1}{ N^{\frac{p}{2}\wedge \frac{q-2}{2}}}\right]e^{C /\widehat\eta},
\]
completing the proof of Theorem \ref{thm: stability}.
\end{proof}

%
%
%
%
%
%
%

\end{document}